\documentclass{ws-ijbc}
\usepackage{ws-rotating}     
\usepackage{graphicx}
\usepackage{hyperref}
\usepackage{xcolor}
\usepackage{float}
\usepackage{array}
\usepackage[normalem]{ulem}
\usepackage{cancel}
\newcommand{\stkout}[1]{\ifmmode\text{\sout{\ensuremath{#1}}}\else\sout{#1}\fi}

\begin{document}


\markboth{G. Yigit et al.}{Instability in a reaction-diffusion system}

\title{Parameter spaces for cross-diffusive-driven instability in a reaction-diffusion system on an annular domain}


\author{Gulsemay Yigit}

\address{Department of Mathematics, Faculty of Engineering and Natural Sciences,\\ Bahcesehir University, Istanbul, T\"{u}rkiye\\
Mathematics Department,
Mathematics Building, 1984 Mathematics Road,\\
             The University of British Columbia, Vancouver, BC Canada V6T 1Z2.\\
gulsemay.yigit@bau.edu.tr}

\author{Wakil Sarfaraz}
\address{Corndel Ltd., 410 Highgate Studio 53-79 Highgate Road, \\
London NW5 1TL, United Kingdom\\
ICT,  Bahrain Polytechnic,\\ PO Box 33349, Isa Town, Kingdom of Bahrain\\
wakil.sarfaraz@corndel.com}

\author{Raquel Barreira}
\address{Instituto Polit\'{e}cnico de Set\'{u}bal,  Escola Superior de Tecnologia de Set\'{u}bal
\\Campus do IPS Estefanilha, 2914-508 Set\'{u}bal, Portugal\\
Centro de Matem\'{a}tica, Aplicac\~{o}es Fundamentais e Investigac\~{a}o Operacional (CMAFcIO), \\
Universidade de Lisboa, Portugal\\
raquel.barreira@estsetubal.ips.pt}

\author{Anotida Madzvamuse\footnote{Author for correspondence}}
\address{Mathematics Department,
Mathematics Building, 1984 Mathematics Road,\\
             The University of British Columbia, Vancouver, BC Canada V6T 1Z2.\\
Department of Mathematics and Applied Mathematics, \\University of Pretoria, Pretoria 0132, South Africa\\
Department of Mathematics and Applied Mathematics, \\University of Johannesburg, \\PO Box 524 Auckland Park, 2006, South Africa\\
University of Zimbabwe, Department of Mathematics and Computational Science, \\
    Mt Pleasant, Harare, Zimbabwe\\
am823@math.ubc.ca}


\maketitle


\begin{abstract}
In this work,  the influence of geometry and domain size on spatiotemporal pattern formation is investigated to establish parameter spaces for a cross-diffusive reaction-diffusion model on an annulus.  By applying linear stability theory, we derive conditions which can give rise to Turing,  Hopf and transcritical types of diffusion-driven instabilities.  We explore whether selection of a sufficiently large domain size, together with the appropriate selection of parameters, can give rise to the development of patterns on non-convex geometries e.g. annulus.  Hence, the key research methodology and outcomes of our studies include: a complete analytical exploration of the spatiotemporal dynamics in an activator-depleted reaction-diffusion system; a linear stability analysis to characterise the dual roles of cross-diffusion and domain size of pattern formation on an annulus region; the derivation of the instability conditions through lower and upper bounds of the domain size;  the full classification of the model parameters,  and a demonstration of how cross-diffusion relaxes the general conditions for the reaction-diffusion system to exhibit pattern formation. To validate theoretical findings and predictions,  we employ the  finite element method to reveal spatial and spatiotemporal patterns in the dynamics of the cross-diffusive reaction-diffusion system within a two-dimensional annular domain. These observed patterns resemble those found in ring-shaped cross-sectional scans of hypoxic tumours. Specifically, the cross-section of an actively invasive region in a hypoxic tumour can be effectively approximated by an annulus.
\end{abstract}

\keywords{Reaction-diffusion systems; cross-diffusion; pattern formation; parameter space; spatiotemporal dynamics; annular domain; finite element method; standing wave.}

\section{Introduction}
\noindent Patterns varying in space and time are abundantly seen in biological systems, from the development of embryos, to the patterning of animal skin pigmentation \cite{murray2001mathematical,baker2008partial,gierer1972theory,satnoianu2000turing,madzvamuse2000numerical}.
The concept of pattern formation in reaction-diffusion systems originates from Turing's seminal work in the field of biological morphogenesis \cite{turing1990chemical}. A primary mechanism proposed to be responsible for driving the formation of these patterns is the instability propelled by considering the effect of the diffusion in the dynamical interplay of species undergoing the process of reaction-diffusion. Since the initial work of Alan Turing, the theory of reaction-diffusion has been significantly expanded and explored by subsequent researchers in many fields of science \cite{gierer1972theory,schnakenberg1979simple,murray2001mathematical,baker2008partial,gierer1972theory,satnoianu2000turing,madzvamuse2000numerical,landge2020pattern,diez2023turing}.

Cross diffusion, has an immense effect in the study of pattern formation, biological systems, and chemical reactions, where the interactions between different components play a crucial role in the spatiotemporal dynamics of the reaction-diffusion systems. Cross-diffusion refers to a special type of a diffusion process in which one of the species influences the flux of the concentration gradient of others. Cross-diffusion coefficients are positioned in the off-diagonal elements of the diffusion matrix. Empirical observations reveal that these coefficients can attain positive, negative, zero and in some cases, the diffusion coefficients can be of substantially different magnitudes. The negative sign of cross-diffusion term indicates that one species is attracted towards higher concentrations of another species \cite{vanag2009cross}.
Turing instabilities, emerging as a consequence of cross-diffusion, have been widely studied and presented in the literature within a substantial class of predator-prey systems involved in ecological or competitive interactions 
\cite{gambino2018cross,li2019cross, zhang2020characterizing, song2023cross, yang2023cross}. Studies include applications of cross-diffusive systems to pattern formation  \cite{yi2021turing, kersner2022competition, ritchie2022turing, kumari2022cross, liu2022dynamics,yang2022cross},
bacterial chemotaxis \cite{bellomo2022chemotaxis, kim2023modeling,gaffney2023spatial},
epidemology \cite{duan2019turing,chang2020cross, hu2021turing, mohan2021positive,hu2022dynamics} and so forth.  Most of these studies include linear case of cross-diffusion terms, in some instances the research also includes nonlinear cross-diffusion, which can be found in \cite{gambino2012turing,gambino2013pattern}. Another application relevant to cross-diffusive systems is illustrated by the process of electrodeposition \cite{lacitignola2018cross,lacitignola2021turing}.
A study on spatiotemporal pattern formation for a cross-diffusive reaction-diffusion system is given by \cite{yang2022cross}.  Comparing our contribution to the study of \cite{yang2022cross}, our distinctive results show the effect of the ring shape domain size for the system to exhibit spatiotemporal pattern formation.

Our research introduces a novel aspect in the field of cross-diffusive reaction diffusion systems through the analysis and numerical simulations of spatiotemporal patterns, with a particular emphasis on the ring shape domain. To the best of our knowledge, to date, very little analysis and simulations of reaction-diffusion systems characterising the effects of domain shape size and and cross-diffusion have been undertaken in the literature.. In this respect, we aim to contribute insights to the understanding of spatiotemporal patterns on the ring shape domain. By extending the idea of spatiotemporal pattern formation without cross-diffusion with the relation between the domain size system parameters presented in \cite{sarfaraz2017classification,sarfaraz2018domain,sarfaraz2020stability}, we provide novel results for the cross-diffusive systems. The focus of the works of \cite{sarfaraz2017classification,sarfaraz2018domain,sarfaraz2020stability} primarily addresses spatiotemporal pattern formation without cross-diffusion. Therefore, the current work constitutes as substantial extension of current-state-of-the-art analysis of reaction-diffusion systems by exploring the collective effects of the cross-diffusion, the geometry, and the domain size on the dynamical properties of the reaction-diffusion system (RDS). On the other hand,  understanding pattern formation in reaction-diffusion systems on growing domains is crucial for modeling dynamic biological processes.  Studies such as \cite{cotterell2015local,krause2021introduction,madzvamuse2010stability,klika2017history,van2021turing,krause2023concentration,tsubota2024bifurcation} highlight the significant influence of domain growth on patterning.   Our work distinguishes from such studies,  by deriving conditions on domain size for studying instability in reaction-diffusion systems on a stationary domain. Extensions of this work to growing domains and evolving surfaces is the subject of our current studies and is beyond the scope of this manuscript.


Obtaining closed-form solutions to reaction diffusion systems is not practical or possible for the majority of cases involving nonlinear reaction kinetics. This complexity is overcome by using numerical approaches which provide indispensable tools for gaining insights into the spatiotemporal dynamics of RDS. Therefore, the significance of numerical methods becomes evident in their ability to address, computationally, the nature of RDSs and facilitate their comprehensive understanding.  Within this context, numerous computational works have been contributed to explore the numerical simulations of RDSs. One such powerful numerical method, widely acknowledged in the literature, is the finite element method, specifically used for the spatial discretisation of RDSs. This numerical method enhances our ability to analyse and interpret the spatiotemporal dynamical behaviour of RDSs. Examples of studies which investigate the computational aspects of RDSs on different types of geometries including stationary and growing domains can be found in \cite{madzvamuse2007velocity,barreira2011surface,madzvamuse2014fully,lakkis2013implicit,tuncer2017projected,frittelli2021bulk,song2022efficient,frittelli2023bulk,frittelli2017lumped}.
We have employed the finite element method to enable the visualisation of patterning structures resulting from Hopf and transcritical type of bifurcations as well as from Turing instabilities. 

This work introduces a set of novel theoretical findings by deriving bifurcation results through an analytical approach motivated by the dynamical systems' theory for a cross-diffusive system defined on a non-convex geometry.  The significance of these results have the potential to offer profound insights in understanding both the effects of domain size on the dynamics and the impact of cross-diffusion on annular regions in the context of pattern formation.  Our distinctive contributions includes a comprehensive framework elucidating the influence of geometry on the development of spatiotemporal pattern formations for a cross-diffusive reaction-diffusion model as well. Our preliminary studies \cite{yigit2024domain,sarfaraz2024understanding} leading to this study,  finally, allowed us to fully understand and compare the geometric impact of spatiotemporal pattern formation of cross-diffusive systems on flat, circular and annular two-dimensional geometries.  The focus of \cite{yigit2024domain,sarfaraz2024understanding} primarily examines the role of convex geometries on Turing's theory of pattern formation. The significant novelty in the present study lies in the expansion of the methodology,  where we analytically provide new bifurcation results for a cross-diffusion reaction diffusion systems within a non-convex geometry. A key finding of this study is the parameter spaces generated on non-convex geometries are substantially different from those of convex geometries.

Therefore, this article is organised as follows: In Section \ref{Modeleqn}, we present the non-dimensional reaction-diffusion equations with linear cross-diffusion on the annular region, a non-convex geometry. We provide the necessary conditions for the well-posedness of the model equations. A detailed linear stability analysis of the cross-diffusive reaction-diffusion system is presented in Section \ref{sec:linearstability} providing the geometric effect of the annular region on the eigenfunctions and eigenmodes of the corresponding eigenvalue problem. In this section, we also present new bifurcation results for the cross-diffusive system to exhibit Hopf/transcritical and Turing types of instabilities. Section \ref{sec:ParameterSpaces} provides a comprehensive classification of parameter spaces corresponding to Hopf/transcritical type of bifurcation regions and Turing regions. These regions are generated using the theoretical results obtained in Section \ref{sec:linearstability}. In Section \ref{sec:FemSim}, we show the numerical simulations of cross-diffusive reaction-diffusion systems using the finite element method to confirm the proposed classification of parameter spaces and to validate the predicted dynamical behaviour based on theoretical considerations. In Section \ref{sec:Conclusion}, we provide the conclusive remarks and possible research directions of the present work.

\subsection{Model equations}\label{Modeleqn}
We consider an \emph{activator-depleted} reaction-diffusion system (RDS) with linear cross-diffusion modelling the dynamics of two chemical species $u(x,y,t)$ and $v(x,y,t)$ within a non-convex circular domain $\Omega = \{(x,y)\in\mathbb{R}^2: a^2<x^2+y^2<b^2 \}$. Here, the domain consists of an annulus with an inner radius $a$ and outer radius $b$. The boundary is represented as $\partial\Omega = \{(x,y)\in\mathbb{R}^2: x^2+y^2=a^2 \} \cup\{(x,y)\in\mathbb{R}^2: x^2+y^2=b^2 \}$. The RDS in the presence of linear cross-diffusion in non-dimensional form on the defined annulus under Neumann boundary conditions is given by,
\begin{equation}\label{r1}
    \begin{cases}
     \begin{cases}
     \displaystyle{\frac{\partial u}{\partial t} =\Delta_r u + d_v \Delta_r v + \gamma f(u,v)},  &
\smash{\raisebox{-1.6ex}{$(r,\theta)\in \Omega$, $t>0$ }}
\\[6pt]
        \displaystyle{\frac{\partial v}{\partial t} =d_u \Delta_r u + d\Delta_r v +  \gamma g(u,v),}\\
       \end{cases} \\
    \begin{cases}
 \displaystyle{\left(\frac{\partial u}{\partial r} + d_v \frac{\partial v}{\partial r}\right)\Big{|}_{r=a}} = \left(d_u
\frac{\partial u}{\partial r} + d \frac{\partial v}{\partial r}\right)\Big{|}_{r=a} = 0,   &\smash{\raisebox{-1.6ex}{$(r,\theta)\in \partial \Omega$, $t>0$, }}
\\[6pt]
\displaystyle{\left(\frac{\partial u}{\partial r} + d_v \frac{\partial v}{\partial r}\right)\Big{|}_{r=b}} = \left(d_u
\frac{\partial u}{\partial r} + d \frac{\partial v}{\partial r}\right)\Big{|}_{r=b} = 0, 
\end{cases}\\
u(r,\theta,0)=u_0(r,\theta), ~v(r,\theta,0)=v_0(r,\theta),~ (r,\theta)\in \Omega,~  t=0, 
    \end{cases}
\end{equation}
where  $\Delta_r $ denotes the Laplace operator in $(r, \theta)$ coordinates, provided later by equation \eqref{Laplace}.
The macroscopic dispersion approximated by mean-field effect of random movements (Brownian motion) of microscopic particles of the reacting and diffusing species motivates the use of the diffusion operator, which by incorporating the Fickian law of diffusion and conservation of mass (both for self-diffusive and cross-diffusive species) justifies the derivation of the cross-diffusive reaction-diffusion system given by \eqref{r1} \cite{murray2001mathematical,madzvamuse2015cross}.
The parameter $d$ in (\ref{r1}) is a non-dimensional diffusion coefficient representing the ratio of the dimensional diffusion coefficients defined as $d=\dfrac{D_v}{D_u}$. Similarly the constants $d_u$ and $d_v$ are non-dimensional ratios of cross-diffusion coefficients defined as $d_u=\dfrac{D_{uv}}{D_u}$ and $d_v=\dfrac{D_{vu}}{D_u}$, respectively. The quantities $D_u>0$, and $D_v>0$ are the dimensional diffusion coefficients for the specie $u$ and $v$, respectively. The cross-diffusion coefficient of $u$ influenced by the presence of the concentration of $v$ is denoted by $D_{uv}$ and that of $v$ influenced by the concentration presence of $u$ is represented by $D_{vu}$. A detailed calculation of the non-dimensionalization process resulting in system \eqref{r1} is presented in \cite{madzvamuse2015cross}.
The functions $f(u,v)=\alpha-u+u^2v$ and $g(u,v)=\beta-u^2v$ are the nonlinear reaction kinetics coupling the two species, also known as the activator-depleted kinetics with $\alpha>0$ and $\beta>0$. \cite{murray2001mathematical,gierer1972theory,schnakenberg1979simple}. 

{\bf Remark.} Note that, in order for the cross-diffusive system to be well-posed, it is necessary for the main diffusion parameter $d$ and the cross-diffusion parameters $d_u$ and $d_v$ to be given such that the inequality $d - d_u d_v > 0$ holds true \cite{madzvamuse2015cross}. Therefore, the condition on the positivity of the determinant of the diffusion matrix is satisfied and this enables us to write the boundary conditions as,
\begin{equation}
\begin{cases}
\displaystyle{{\frac{\partial u}{\partial r}\Big{|}_{r=a}}}=\displaystyle{{\frac{\partial v}{\partial r}\Big{|}_{r=a}}}=0, 
&\smash{\raisebox{-1.6ex}{$(r,\theta)\in \partial \Omega$, $t>0$, }}
\\[6pt]
\displaystyle{{\frac{\partial u}{\partial r}\Big{|}_{r=b}}}=\displaystyle{{\frac{\partial v}{\partial r}\Big{|}_{r=b}}}=0.
\end{cases}
\end{equation}

\section{Linear stability analysis for the cross-diffusive system on two dimensional annular region}\label{sec:linearstability}
	
System \eqref{r1} has a constant uniform steady state solution represented by $(u_s,v_s)=(\alpha+\beta,\frac{\beta}{(\alpha+\beta)^2})$ \cite{,murray2001mathematical,madzvamuse2015cross, madzvamuse2015stability}. This steady state is a unique positive stationary point that satisfies the nonlinear algebraic equations obtained by the $f(u_s,v_s)=g(u_s,v_s)=0$, along with the zero-flux boundary conditions imposed on System \eqref{r1}.
Involving local perturbation of System \eqref{r1}, the standard approach of linear stability theory around the uniform steady-state $(u_s+\epsilon\bar{u},v_s+\epsilon\bar{u})$ with $0<\epsilon \ll 1$ is used.
The perturbation variables $(\bar{u},\bar{v})$ are assumed to be sufficiently smooth satisfying the conditions for Taylor expansion of two variables. We proceed with neglecting $O(\epsilon^2)$ and higher order terms. Therefore, we obtain the following linearized matrix-vector form as,
\begin{equation}\label{rr2}
{\bar {\textbf{w}_t}}=
\textbf{D}\Delta\bar {\textbf{w}}
+\gamma
\textbf{J}_\textbf{F} \bar{\textbf{w}}
\end{equation}
where $\bar{\textbf{w}}$ is the solution vector, $\textbf{D}$ denotes the diffusion coefficients matrix, $\textbf{F}$ is reaction terms vector, and  $\textbf{J}_\textbf{F}$ represents the Jacobian matrix. These can be written explicitly as, 
\begin{equation}\label{diffmatrix}
 \bar{\textbf{w}}=
\begin{bmatrix}
\bar{u}\\
\bar{v}
\end{bmatrix}, \; 
\textbf{D}=
\begin{bmatrix}
1& d_v\\
d_u & d
\end{bmatrix}, \; 
\textbf{F}(u,v)=
\begin{bmatrix}
f(u_s,v_s)\\
g(u_s,v_s)
\end{bmatrix},
\;
\textbf{J}_{\textbf{F}}=
\begin{bmatrix}
f_u(u_s,v_s)& f_v(u_s,v_s)\\
g_u(u_s,v_s)& g_v(u_s,v_s).
\end{bmatrix}.
\end{equation}
To determine the analytical results of bifurcation properties, we employ the linearization procedure and seek to find the eigenfunctions of the Laplace operator that satisfy the homogenous Neumann boundary conditions. These eigenfunctions are obtained by solving the following eigenvalue problem on an annulus domain,
\begin{equation}\label{eigen1}
     \begin{cases}
     \Delta_r w=-k^2 w, \qquad 
 \quad k\in \mathbb{R},\qquad (r,\theta) \in \Omega,
\\[10pt]
        \displaystyle{\frac{\partial w}{\partial r} \Big{|}_{r=a}={\frac{\partial w}{\partial r} \Big{|}_{r=b}=0}} \qquad \quad a,b \in \mathbb{R}_+\setminus \{0\} .
       \end{cases}
\end{equation}
where $\Delta_r $ represents the Laplace operator in polar coordinates as,

\begin{equation}\label{Laplace}
\Delta_r u(r,\theta)=\frac{1}{r}\frac{\partial}{\partial r}\left( r \frac{\partial u}{\partial r}\right) +\frac{1}{r^2}\frac{\partial^2 u}{\partial \theta^2}.
\end{equation}
The solution of the eigenvalue problem \eqref{eigen1} is obtained in the form of $w(r,\theta)=R(r)\Theta(\theta)$ under Neumann boundary conditions on an annular region. The following theorem is presented for the solution of the eigenvalue problem \eqref{eigen1}. The solution of the eigenvalue problem on flat ring with Neumann Boundary Conditions is exploited from the results in \cite{sarfaraz2018domain}, which is provided in the following Theorem. 
\begin{theorem}\label{theoeigen}
(Solution of the eigenvalue problem) Let $w(r,\theta)$ satisfies the eigenvalue problem \eqref{eigen1} under the homogeneous Neumann boundary conditions on annulus domain, with $n$ the order of Bessel's function satisfying $\mathbb{R}\setminus \dfrac{1}{2} \mathbb{Z}$. Then, for a fixed  pair $(m,n)$, $m \in\mathbb{Z}^{+}$,  there exists an infinite set of eigenfunctions of the Laplace operator which is expressed as,
\begin{equation}
w_{m,n}(r,\theta)=[R_{m,n}^1(r)+R_{m,n}^2(r)]\Theta_n(\theta)
\end{equation}
where, $R_{m,n}^1(r)$ and $R_{m,n}^2(r)$ represents the Bessel functions of first kind given as
\begin{equation}
\begin{cases}
R_{m,n}^1(r) = \sum\limits_{j=0}^{\infty} \frac{(-1)^jC_0(k_{m,n}r)^{2j+n}}{4^jj!(n+j)(n+j-1)\cdots(n+1)} \\
R_{m,n}^2(r) = \sum\limits_{j=0}^{\infty} \frac{(-1)^jC_0(k_{m,n}r)^{2j-n}}{4^jj!(-n+j)(-n+j-1)\cdots(-n+1)}
\end{cases}\label{besselsol}
\end{equation}
for $j=2m$ with
$ \Theta_n(\theta)=e^{(n\theta i)}$. Here, $C_0$ is the first coefficient of the Bessel series.
The eigenvlues to the eigenvalue problem  \eqref{eigen1}, $k_{m,n}^2$ exist as, 
\begin{equation}
k_{m,n}^2=\frac{4(a^nb+ab^n)(2m+1)(n+2m+1)(n+4m)}{ab(a^{n+1}+b^{n+1})(n+4m+2)}.
\end{equation}
\end{theorem}
\begin{proof}
The detailed proof of Theorem \ref{theoeigen} is step-by-step presented in \cite{sarfaraz2020stability}. 
\end{proof}

\textbf{Remark.}
The eigenvalues given by $k_{m,n}^2$ can be represented using $b=\rho+a$ as follows,
\begin{equation}\label{eigenvalueExpression}
k_{m,n}^2=f(\rho,n)\frac{4(2m+1)(n+2m+1)(n+4m)}{(n+4m+2)},
\end{equation}
where 
\begin{equation}
f(\rho,n)=\dfrac{a^{n-1} +(\rho+a)^{n-1}}{a^{n+1}+ (\rho+a)^{n+1}}.
\end{equation}
Before we proceed with the domain size conditions to understand the dynamics of the cross-diffusive system, the properties of $f(\rho,n)$ need to be explored at the limiting case with respect to the order of the Bessel's functions.  We also explore the properties of $f(\rho,n)$ by the variation of the inner radius $a$, and the thickness of the ring which is denoted by $\rho$. This can be accomplished through considering the upper bounds of $f(\rho,n)$ for the following two possible cases namely $n<0$ and $n>0$,
\begin{equation}\label{eigenValueSup1}
{\underset{0>n \in \mathbb{R}\setminus \frac{1}{2} \mathbb{Z} }\sup} f(\rho,n)= \lim_{n\rightarrow -\infty} f(\rho,n)= \dfrac{2}{a(\rho+a)},
\end{equation}
and 
\begin{equation}\label{eigenValueSup2}
{\underset{0<n \in \mathbb{R}\setminus \frac{1}{2} \mathbb{Z} }\sup} f(\rho,n)= \lim_{n\rightarrow 0} f(\rho,n)= \dfrac{1}{a(\rho+a)}.
\end{equation}
Therefore, the conditions are derived using the above bounds of the domain-weighting function $f(\rho,n)$ \cite{sarfaraz2020stability}. A numerical demonstration of these limiting cases is simulated with various combinations of $a$ and $\rho$ as shown in Fig. \ref{domainfactor}.
\begin{figure}[H]
\begin{tabular}{cc}
\includegraphics[width=0.45\textwidth]{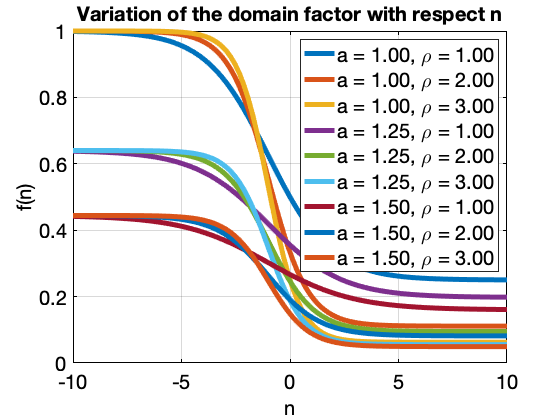} \\
{\small (a) $f(\rho,n)$ by the variation $a$ and $\rho$.}
\end{tabular}
\begin{tabular}{cc}
\includegraphics[width=0.45\textwidth]{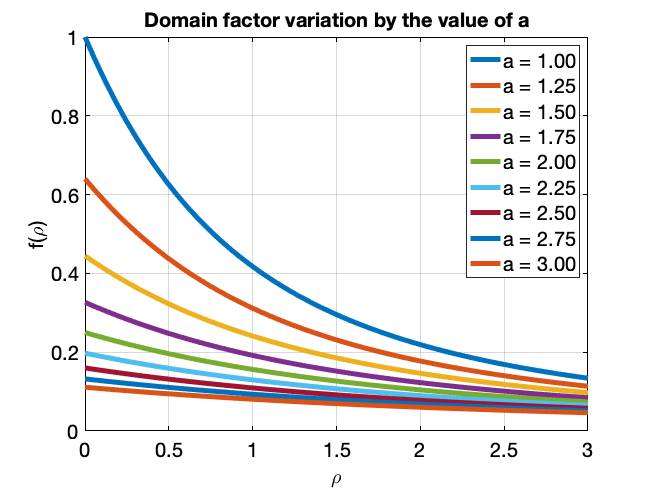} \\
{\small (b) $f(\rho,n)$ by the variation of $a$}
\end{tabular}
\caption{(a)-(b) Analysis of the domain factor on the eigenvalue of the Laplace operator demonstrated for the limiting cases with respect to the order of Bessel's function $n$, the inner radius $a$, and the thickness of the ring $\rho$.}
\label{domainfactor}
\end{figure}

The demonstration above is necessary to verify that in the limiting case when the inner radius of the ring approaches zero, the domain factor on the eigenvalues of the Laplace operator on a ring coincide to those on the disc, which is $\lim_{a \rightarrow 0}f(\rho,n)=\frac{1}{\rho^2}$. Note that this is the case since
\begin{align*}
\lim_{a \rightarrow 0}\dfrac{a^{n-1} +(\rho+a)^{n-1}}{a^{n+1}+ (\rho+a)^{n+1}} = \frac{1}{\rho^2}.
\end{align*}
The solution to system \eqref{r1} can be expressed through the method of separation of variables, where the ansatz takes the form of the product of the eigenfunctions of the Laplacian $\Delta_r$ on annulus and an exponential with respect to the time variable $t$, written as,
\begin{equation}\label{eigen2}
{{u}}(r,\theta,t)=\sum\limits_{m=0}^{\infty}\sum\limits_{n=0}^{\infty}U_{m,n}e^{-k^2_{m,n}t} R_{m,n}\Theta_n(\theta),  \; \text{and} \; 
{{v}}(r,\theta,t)=\sum\limits_{m=0}^{\infty}\sum\limits_{n=0}^{\infty}V_{m,n}e^{-k^2_{m,n}t} R_{m,n}\Theta_n(\theta),
\end{equation}
where $U_{n,m}$ and $V_{n,m}$ are the coefficients of the eigenfunctions in the eigen expansion of the series solution to \eqref{eigen1}. Substitution of \eqref{eigen2} into the vectorised representation \eqref{rr2} provides the fully linearised form of \eqref{r1},
\begin{align}\label{smat}
\begin{bmatrix}
\bar{u}_t\\
\bar{v}_t
\end{bmatrix}
=
\begin{bmatrix}
-k_{m,n}^2+\gamma f_u(u_s,v_s) &-d_vk_{m,n}^2+\gamma f_v(u_s,v_s)\\
-d_uk_{m,n}^2+\gamma g_u(u_s,v_s)&-dk_{m,n}^2+\gamma g_v(u_s,v_s)
\end{bmatrix}
\begin{bmatrix}
\bar{u}\\
\bar{v}
\end{bmatrix}.
\end{align}
The characteristic polynomial is written in terms of the trace and determinant of the stability matrix as,
\begin{equation}\label{polyn}
\lambda^2-\mathcal{T}(\alpha,\beta)\lambda+\mathcal{D}(\alpha,\beta)=0.
\end{equation}
In \eqref{polyn} above, the symbols $\mathcal{T}(\alpha,\beta)$ and $\mathcal{D}(\alpha,\beta)$ denote the trace and the determinant of the stability matrix given by  \eqref{smat}, respectively. 
Thus, the pair of the eigenvalues for the stability matrix of the system are obtained by calculating the roots of \eqref{polyn}, which are given by
\begin{equation}\label{eigen}
\lambda_{1,2}=\frac{\mathcal{T}(\alpha,\beta)\mp\sqrt{\mathcal{T}^2(\alpha,\beta)-4\mathcal{D}(\alpha,\beta)}}{2}.
\end{equation}
In the case where both eigenvalues $\lambda_{1,2}$ given in \eqref{eigen} are real, the stability of the uniform steady-state $(u_s,v_s)$ depends on the signs of the eigenvalues. For the real case of both eigenvalues $\lambda_{1,2}$, the uniform steady state is unstable if at least one of the eigenvalues $\lambda_1$ or $\lambda_2$ is positive.   On the other hand, if both eigenvalues form a complex-conjugate pair, then the stability of the uniform steady-state $(u_s,v_s)$ is determined by the sign of the real part of the eigenvalues. To predict the dynamics of the system \eqref{r1}, regarding each type of instability, we investigate the parameter regions on the top right quadrant of the cartesian plane i.e. $(\alpha,\beta) \in \mathbb{R}_+^{2}$. This is because the only permissible values for parameters $\alpha$ and $\beta$ by definition of the model are positive real values.

\subsection{Spatiotemporal pattern formation on annular region }

To determine the relationship between the model system parameters and domain-size which is represented by $\rho=b-a$ showing the thickness of the annulus, we first analyse the roots of the characteristic polynomial $\eqref{polyn}$. 
The parameter plane $(\alpha, \beta) \in \mathbb{R}^2_+$ features a partition defined by a curve. This curve serves to distinguish regions on the plane that constrain $\lambda_{1,2}$ to either a real pair or a complex conjugate pair. The implicit equation governing this curve is
\begin{equation}
\label{partitionb}
\mathcal{T}^2(\alpha,\beta) -4\mathcal{D}(\alpha,\beta)=0.
\end{equation}
Equation \eqref{partitionb} signifies that one side of the curve corresponds to the region where $\lambda_{1,2}\in\mathbb{R}$, while the other side corresponds to $\lambda_{1,2}\in\mathbb{C}\backslash\mathbb{R}$. It follows that the eigenvalues $\lambda_{1,2}$ form a complex conjugate pair if the parameters $(\alpha, \beta)\in\mathbb{R}^2_+$ satisfy the inequality
\begin{equation} \label{complexpair}
\mathcal{T}^2(\alpha,\beta) -4\mathcal{D}(\alpha,\beta)<0.
\end{equation}
The immediate requirement of \eqref{complexpair} gives the inequality $\mathcal{D}(\alpha,\beta)>0$. As mentioned above, the trace of the stability matrix \eqref{smat} does not include the cross-diffusion constants given by $d_u$ and $d_v$. Therefore, we deduce that, conditions on the domain-size in the absence of cross-diffusion remain the same as in the work \cite{sarfaraz2018domain}. This leads us to explore the positivity of $\mathcal{D}(\alpha,\beta)$ to fully understand the effect of the cross-diffusion in the reaction-diffusion system. Before we proceed to the analysis on the positivity $\mathcal{D}(\alpha,\beta)$ of the stability matrix \eqref{smat}, we recall the following theorems which depend on the analysis of trace $\mathcal{T}^2(\alpha,\beta)$ of the stability matrix \cite{sarfaraz2018domain}.

\begin{theorem}\label{theo1}
 (Condition for Hopf/Transcritical bifurcation) Let $u$ and $v$ satisfy the cross-diffusive reaction-diffusion system given by the Eq. (\ref{r1}) with real positive parameters $\alpha>0$, $\beta>0$, $d>0$, and $\gamma>0$ on a domain with the thickness of ring $\rho=b-a$  where $a$ and $b$ represent the inner and outer boundaries of domain respectively. For the system to exhibit Hopf and/or transcritical bifurcation, the domain-size controlling parameter $\rho$ must be large enough satisfying, 
\begin{equation}\label{Theo1}
\rho \geq\dfrac{8(d+1)(2m+1)(n+2m+1)(n+4m)-\gamma a^2(n+4m+2)}{\gamma a(n+4m+2)}
\end{equation}
where $n \in\mathbb{R}\setminus \dfrac{1}{2} \mathbb{Z}$ is the associated order of the Bessel's equation and $m$ is any positive integer. 
\end{theorem}

\begin{proof}
The proof of this theorem strategically focuses on examining the positivity of the trace $\mathcal{T}(\alpha,\beta)$ of the stability matrix. It is noteworthy that the trace $\mathcal{T}(\alpha,\beta)$,  remains invariant when introducing cross-diffusive coefficients to the system. Consequently, we infer that the (necessary) condition for the system to exhibit Hopf/transcritical-type bifurcation remains unaltered compared to the standard reaction-diffusion system without cross-diffusion. It must be noted that only the self-diffusion coefficient $d$ enters into the condition through the definition of the trace of the stability matrix. The proof establishing the domain size and system parameters without cross-diffusion can be found in \cite{sarfaraz2020stability}.
\end{proof}

\begin{theorem}\label{theo2}
(Turing diffusion-driven instability) Let $u$ and $v$ satisfy the cross-diffusive reaction-diffusion system given by (\ref{r1}) with real parameters $\alpha>0$, $\beta>0$, $d>0$, and $\gamma>0$ on the ring  $\Omega \subset \mathbb{R}^2$ whose thickness is given by $\rho=b-a$ where $a$ and $b$ represent the inner and outer boundaries of ring domain, respectively. If the domain-size controlling parameter $\rho$ satisfies the condition, 
\begin{equation}
\rho<\dfrac{4(d+1)(2m+1)(n+2m+1)(n+4m)-\gamma a^2(n+4m+2)}{\gamma a(n+4m+2)}
\end{equation}
then, the instability of the cross-diffusive system \eqref{r1} is restricted to Turing type only, implying that this condition forbids the temporal periodicity in the dynamics. In \eqref{theo2}, $n \in\mathbb{R}\setminus \dfrac{1}{2} \mathbb{Z}$ represent the order of Bessel's equation and $m$ is a positive integer.
\end{theorem}
\begin{proof}
The proof of this theorem involves examining the real part of the eigenvalues $\lambda_{1,2}$ when the discriminant $\mathcal{T}^2-4\mathcal{D}$ of the characteristic polynomial is negative. Similar to the scenario in Theorem \eqref{theo1}, it is important to highlight that the trace $\mathcal{T}(\alpha,\beta)$ of the stability matrix remains unaffected by the values of $d_u$ and $d_v$. Consequently, we conclude that the conditions outlined in Theorem \ref{theo2} persist unchanged in the absence of cross-diffusion, as discussed and proven in \cite{sarfaraz2020stability}. Hence, we note again that only the self-diffusion coefficient $d$ enters into the condition through definition of the trace of the stability matrix.
\end{proof}
\noindent \textbf{Remark.} It is important to note that Theorem \ref{theo1} represents the necessary condition for the system to exhibit Hopf/transcritical type of bifurcations, on the other hand it allows for the possibility that the system may undergo Turing-type instability. However, Theorem \ref{theo2} precludes temporal periodicity, permitting only spatial pattern formation. Both Theorems \ref{theo1} and \ref{theo2} are independent of  cross-diffusion terms since the strategy used in these theorems is constructed based on the exploitation of the sign of $\mathcal{T}(\alpha,\beta)$, which is independent of cross-diffusion constants of System \eqref{r1}. Conditions in Theorems \ref{theo1} and  \ref{theo2} remain valid in the form of necessary requirements in the case of cross-diffusive systems. 

Our analysis extends to establishing the sufficient conditions regarding the thickness of the ring $\rho$ in relation to the parameters of the reaction-cross-diffusive system, namely $d$, $d_u$, $d_v$, and $\gamma$. This is achieved by examining the positivity of the determinant $\mathcal{D}(\alpha,\beta)$ of the stability matrix. The process involves expanding the determinant of the stability matrix \eqref{eigen2} and rearranging $\mathcal{D}(\alpha,\beta)$ into the form of an implicit cubic polynomial in $\beta$. Each term in the polynomial shares a common factor of $\frac{1}{\alpha+\beta}$ with all other parameters consolidated as coefficients. This expression reads as
\begin{equation}\label{dnew}
\mathcal{D(\alpha,\beta)}=p_0 + p_1 \beta + p_2 \beta^2 + p_3 \beta^3, 
 \end{equation}
 where 
 \begin{align}
p_0 = \frac{1}{\alpha+\beta}\kappa_0(\alpha), ~~
p_1  = \frac{1}{\alpha+\beta} \kappa_1(\alpha),  ~~
p_2   = \frac{1}{\alpha+\beta}\kappa_2(\alpha),  ~~
p_3  = \frac{1}{\alpha+\beta}\kappa_3(\alpha).\notag 
 \end{align}
 Here, the terms $\kappa_i$'s  ($i=0,1,2, 3$) are expressed in terms of all the remaining parameters, which are given by 
\begin{eqnarray} 
 &&\kappa_0(\alpha)=\alpha^3\gamma^2 + \alpha^3\gamma k_{m,n}^2 + \alpha dk_{m,n}^4 - \alpha d_u d_vk_{m,n}^4+\alpha d \gamma k_{m,n}^2 +\alpha^3d_u\gamma k_{m,n}^2,   \nonumber \\
&&\kappa_1(\alpha)=dk^4 - d_ud_vk_{m,n}^4- d\gamma k_{m,n}^2 + 3\alpha^2\gamma^2 + 3\alpha^2\gamma k_{m,n}^2 +3\alpha^2d_u\gamma k_{m,n}^2 - 2d_v\gamma k_{m,n}^2, \nonumber \\
&&\kappa_2(\alpha)= 3\alpha\gamma( k_{m,n}^2(d_u+1)+ \gamma), \nonumber \\
&&\kappa_3(\alpha)=\gamma k_{m,n}^2 + \gamma^2 + d_u\gamma k_{m,n}^2 =\gamma( k_{m,n}^2(d_u+1)+ \gamma). \nonumber
\end{eqnarray}
Note that $(\alpha,\beta)\in\mathbb{R}_+^2$ are by definition non-zero positive constants, therefore, we assert that $\mathcal{D}(\alpha,\beta)>0$ requires the positivity of the cubic polynomial in $\beta$ given in \eqref{dnew}. We proceed by writing the polynomial such that the coefficient of $\beta^3$ is one, which means we have 
\begin{equation}\label{standardCubic}
\beta^3+\dfrac{\kappa_2(\alpha)}{\kappa_3(\alpha)}\beta^2+\dfrac{\kappa_1(\alpha)}{\kappa_3(\alpha)}\beta+\dfrac{\kappa_0(\alpha)}{\kappa_3(\alpha)}>0.
 \end{equation}

Note that the domain-size controlling parameter $\rho$ representing the thickness of the flat-ring is implicitly embedded within the expressions of the coefficients of the cubic polynomial described by \eqref{standardCubic}. This is through the fact that the coefficients of \eqref{standardCubic} directly depend on $k_{m,n}^2$, which represent the eigenvalues of the diffusion operator, and $\rho$ is a parameter of the expression for $k_{m,n}^2$ as shown in \eqref{eigenvalueExpression}.  We consider the following theorem from \cite{qi2020positivity} to obtain the conditions required for \eqref{standardCubic} to be positive in terms of the thickness of ring $\rho$ along with $d$, $d_u$, $d_v$, $\gamma$.

%

\begin{proposition}\label{prop1}
Suppose that $\mathcal {D}(\beta)=\beta^3+a\beta^2+b\beta+c$ be a non-degenerate cubic polynomial. Let the cubic polynomial $\mathcal{D}(\beta)$ defined as follows, 
\begin{equation}
\mathcal{D}(\beta)=h(\beta)\beta+c
\end{equation}
where $h(\beta)=\beta^2+a\beta+b$, with $c>0$ and $\beta\geq 0$.  For the cubic polynomial $\mathcal{D}(\beta)$ to be strictly positive, the quadratic polynomial $h(\beta)$ must satisfy,$a\geq 0$, $b\geq 0$, or  $b>0$, $4b\geq a^2$.
\end{proposition}
\begin{proof}
The proof of this proposition is  presented in \cite{qi2020positivity,schmidt1988positivity}.
\end{proof}
The conditions on $a$, $b$, and $c$ in Proposition \ref{prop1} are represented in terms of the reaction-diffusion system parameters as follows,

 \begin{equation} \label{trivial}
a=\dfrac{\kappa_2(\alpha)}{\kappa_3(\alpha)}=\dfrac{3\alpha\gamma^2 + 3\alpha\gamma k_{m,n}^2 + 3\alpha d_u\gamma k_{m,n}^2}{\gamma k_{m,n}^2 + \gamma^2 + d_u\gamma k_{m,n}^2 }  \geq  0,
 \end{equation}
\begin{eqnarray}\label{ntrivial2}
b=&\dfrac{\kappa_1(\alpha)}{\kappa_3(\alpha)}=\dfrac{(d - d_ud_v)k_{m,n}^4- (d\gamma  - 3\alpha^2\gamma -3\alpha^2d_u\gamma  - 2d_v\gamma) k_{m,n}^2}{\gamma k_{m,n}^2 + \gamma^2 + d_u\gamma k_{m,n}^2 } \\ \nonumber
&\hspace{2cm}+\dfrac{ 3\alpha^2\gamma^2}{\gamma k_{m,n}^2 + \gamma^2 + d_u\gamma k_{m,n}^2 }  \geq 0 ,
 \end{eqnarray}
and
\begin{eqnarray}\label{ntrivial1}
c=\dfrac{\kappa_0(\alpha)}{\kappa_3(\alpha)}=\dfrac{(d-d_ud_v)\alpha k_{m,n}^4 + (\alpha^3\gamma   + \alpha d \gamma  +\alpha^3d_u\gamma) k_{m,n}^2+\alpha^3\gamma^2}{\gamma k_{m,n}^2 + \gamma^2 + d_u\gamma k_{m,n}^2 }  > 0.
 \end{eqnarray}
With this setup in mind, we state Theorem 4 and provide its proof subsequently.
\begin{theorem}\label{Maincond}
 (Condition on $\rho$ for spatiotemporal pattern formation of cross-diffusive reaction-diffusion system) Let $u$ and $v$ satisfy System \eqref{r1} with cross-diffusive constants $d_u$ and $d_v$ as presented therein with real positive parameters $\alpha>0$, $\beta>0$, $d>0$, $\gamma>0$ on an annulus $\Omega$ with thickness $\rho$ and the inner radius $a$. For the thickness of the annulus denoted by $\rho$ satisfying
\begin{equation}\label{Theo4}
\rho> \dfrac{8(d-d_ud_v) (2m+1)(n+2m+1)(n+4m)-\gamma a^2(7d+8d_v)(n+4m+2)}{(7d+8d_v)(n+4m+2)\gamma},
\end{equation}
is a sufficient condition for the parameter space $(\alpha,\beta)\in\mathbb{R}_+^2$ to admit values for which System \eqref{r1} will undergo spatiotemporal dynamics i.e. (Hopf and/or transcritical bifurcation) pattern formation.
\end{theorem}
\begin{proof}
The strategy of this proof involves analysing conditions on $a$, $b$, and $c$ as provided in \eqref{trivial}-\eqref{ntrivial1}. Direct algebraic operations on \eqref{trivial} leads to a trivial requirement of the form  $3\alpha \geq 0$, which allows to assert that \eqref{trivial} holds true. Conditions \eqref{ntrivial2} and \eqref{ntrivial1} in Proposition \ref{prop1}, which together ensure the positivity of $\mathcal{D}(\beta)$, are investigated and a set of two inequalities of the form 
\begin{equation}\label{newtrivial1}
    \begin{cases}
3\alpha^2+\dfrac{(d-d_u d_v)k_{m,n}^4-(d+2d_v)\gamma k_{m,n}^2}{(1+d_u)\gamma k_{m,n}^2+\gamma^2} \geq \dfrac{9\alpha^2}{4},
\\[10pt]
\alpha^2+\dfrac{(d-d_ud_v)k_{m,n}^4+d\gamma k_{m,n}^2}{(1+d_u)\gamma k_{m,n}^2+\gamma^2} > 0, 
    \end{cases}
\end{equation}
is deduced.
The inequalities provided by $\eqref{newtrivial1}$ are identical to those derived for a convex domain investigated in \cite{yigit2024domain}. This entails a universality in obtaining sufficient conditions for spatiotemporal dynamics in reaction-diffusion systems across different geometries regardless of the convexity condition for the domain. However, the domain-dependent distinction emerges when the explicit expression for the eigenvalues of the Laplace operator corresponding to the geometry and boundary conditions is substituted and the inequalities are simultaneously solved for the domain-size controlling parameter, which in this case is $\rho$.
Note that inequalities in \eqref{newtrivial1} can be re-arranged to be written in the form
\begin{equation}\label{sysIneqnew}
    \begin{cases}
     \dfrac{4(d-d_ud_v)k_{m,n}^4-4(d+2d_v)\gamma k_{m,n}^2}{(1+d_u)\gamma k_{m,n}^2+\gamma^2}= -3\alpha^2
\\[10pt]
       \dfrac{3(d-d_ud_v)k_{m,n}^4+3d\gamma k_{m,n}^2}{(1+d_u)\gamma k_{m,n}^2+\gamma^2} > -3\alpha^2 . 
   \end{cases}
\end{equation}
Without loss of generality,  in \eqref{sysIneqnew}, we substitute for $-3\alpha^2$ in the second inequality to obtain
\begin{align}
\dfrac{3(d-d_ud_v)k_{m,n}^4+3d\gamma k_{m,n}^2}{(1+d_u)\gamma k_{m,n}^2 +\gamma^2} > \dfrac{4(d-d_ud_v)k_{m,n}^4-4(d+2d_v)\gamma k_{m,n}^2 }{(1+d_u)\gamma k_{m,n}^2+\gamma^2}. \label{combinednew}
\end{align}
Solving \eqref{combinednew} for $k_{m,n}^2$ leads to the desired condition on the domain-size controlling parameter $\rho$ that will ensure the positivity of the cubic polynomial \eqref{standardCubic} and in turn that of the determinant of the stability matrix given by \eqref{dnew}. 

To derive the desired condition, the sign of the denominator of both sides of \eqref{combinednew} is exploited, which requires the analysis to independently investigate the  case when the denominator on both sides of \eqref{combinednew} is either positive or negative. 
Such a requirement enforces two independent cases to explore, namely  $d_u>-1$ or $d_u<-1$, subject to $d-d_u d_v >0$.  Hence, the current analysis admits the concept of cross-diffusion for one of the species to be either negative or positive. 
However, the determinant of the diffusion matrix must be positive to guarantee the regularity of System \eqref{r1}. We exploit the experimental investigation presented in \cite{vanag2009cross} and results therein, where negative and positive cross-diffusion give rise to Turing type behaviour in the dynamics. Such experimental findings create the platform to explore the cross-diffusive parameter $d_u$ across the full spectrum of the real line in particular, both for the choice of negative and positive real values. Therefore, exploiting such observations we consider $d_u>-1$, which corresponds to the positivity of the denominator of both sides of \eqref{combinednew}. As a result, we write
\begin{equation}\label{combinednew1}
3(d-d_ud_v)k_{n,m}^2+3d\gamma >4(d-d_ud_v)k_{n,m}^2-4(d+2d_v)\gamma
\end{equation}
or equivalently
\begin{equation}\label{combinednew2}
3d\gamma +4(d+2d_v)\gamma >(d-d_ud_v)k_{n,m}^2.
\end{equation}
Inequality \eqref{combinednew2} is re-written by isolating $k_{m,n}^2$, resulting in
\begin{equation}\label{combinednew3}
k_{n,m}^2< \dfrac{(7d+8d_v)\gamma}{d-d_u d_v}.
\end{equation}

To establish \eqref{combinednew2}, we substitute $k_{m,n}^2$ with the maximum of the two suprema, accounting for the limiting cases of the domain-size factor $f(\rho)$ as outlined in \eqref{eigenValueSup1}:
\begin{equation}\label{combinednew4}
\dfrac{8(2m+1)(n+2m+1)(n+4m)}{a(\rho+a)(n+4m+2)}< \dfrac{(7d+8d_v)\gamma}{d-d_u d_v}.
\end{equation}
The proposed sufficient condition of Theorem \ref{Maincond} is attained by solving \eqref{combinednew4} for $\rho$, which leads to
\begin{equation}
\rho> \dfrac{8(d-d_ud_v) (2m+1)(n+2m+1)(n+4m)-\gamma a^2(7d+8d_v)(n+4m+2)}{(7d+8d_v)(n+4m+2)\gamma},
\end{equation}
thereby completing the proof.
\end{proof}

\section{Geometric comparison of cross-diffusion induced parameter spaces for convex and non-convex domains}\label{sec:ParameterSpaces}
In the analysis of the parameter spaces indicative of Turing instability, Hopf bifurcation, and Transcritical curves within the $(\alpha, \beta)\in\mathbb{R}_+^2$ bifurcation plane, emphasis is placed on the imperative requirement on the partitioning curve that classifies the bifurcation plane into real and complex roots, which must satisfy the implicit equation given by $\mathcal{T}^2(\alpha,\beta) -4\mathcal{D}(\alpha,\beta)=0$. A numerical verification procedure was initially undertaken to ascertain the stable and unstable regions via analysis of the real part of eigenvalues, conducted across distinct two-dimensional geometries encompassing rectangular, circular, and annular domains, i.e. convex and non-convex geometries. Fig. \ref{fullclass1} presents a comprehensive classification of these stable and unstable regions, illustrated according to conditions outlined in Table \ref{Tab1}. Notably, regions characterized by positive real parts of the eigenvalues, regardless of their real or complex nature, denote instability of the solution of System \eqref{r1}, which in turn makes the steady state $(u_s,v_s)$ asymptotically unstable. Conversely, regions featuring negative real parts of the eigenvalues represent dynamic stability of the uniform steady state of System \eqref{r1}. Subsequently, parameter regions depicted in Fig. \ref{fullclass1} are color-coded to signify distinct characteristics.  The colour magenta represents regions corresponding to real distinct eigenvalues, green colour regions with complex eigenvalues and negative real parts, red colour regions with complex eigenvalues and positive real parts, and blue colour regions with real eigenvalues and at least one positive root. Both magenta and green regions denote stability, while the observation of a Hopf bifurcation in system dynamics occurs within the red region. Additionally, the emergence of Turing-type patterns is observed by selecting parameters from within the blue regions. Further analysis conducted with varying domain sizes reveal a variety of shifts in unstable spaces across different geometries, elucidating the intricate interplay between geometry and domain size parameters. Specifically, while the unstable spaces remain intact in terms of area, however, the numerical procedure demonstrates that Turing spaces diminish with increasing domain size, where the area of unstable regions is conserved by expanding Hopf bifurcation regions to conserve the total area corresponding to diffusion-driven instability. Hopf bifurcation regions display variability contingent upon the geometry under consideration. This comprehensive analysis sheds light and thereby adds to the novelty of formalizing a universal understanding of the influences of geometry and domain size on system dynamics within distinct parameter spaces.

The unstable parameter spaces depicted in Figs. \ref{HopfTransParameter}, \ref{TuringTheo1}, and \ref{TuringTheo2} are explored through numerical representation, with a focus on the variations of system parameters $d$, $d_u$, $d_v$, and $\gamma$ within the constraints presented in Section \ref{sec:linearstability}, specified for an annular domain. Fig. \ref{HopfTransParameter} illustrates Hopf bifurcation regions alongside transcritical curves, adhering to the conditions stated in Theorems \ref{theo1} and \ref{Maincond}. It is recognized from the detailed analysis presented in Figs. \ref{HopfTransParameter} (a)-(b) that augmenting the self-diffusion coefficient $d$ while maintaining positive and negative cross-diffusion coefficients $d_u$ induces a reduction in the Hopf bifurcation regions, coupled with a shift of transcritical curves from $t_4$ to $t_1$. An important point of verification associated with these findings is that we observe that including cross-diffusion terms compels System \eqref{r1} to manifest unstable spaces in the plane $(\alpha,\beta)\in\mathbb{R}_+^2$ for $d=1$. This entails that the proven case for reaction-diffusion systems of self-diffusive species in the existing literature precluding the manifestation of unstable spaces for $d=1$ is limited to the absence of cross-diffusion \cite{madzvamuse2015cross}, and is violated when cross-diffusion is added to the system. Furthermore, Figure \ref{HopfTransParameter} (c) illustrates a diminishing trend in Hopf bifurcation regions with an increase in positive $d_u$, whereas transcritical curves remain invariant. Figures \ref{HopfTransParameter} (c)-(f) collectively illustrate that elevating cross-diffusion parameters $d_u$ and $d_v$ while holding $d$ constant gives rise to progressively smaller Hopf bifurcation regions, while limit cycle curves remain unchanged across each case.  Fig. \ref{TuringTheo1} shows the regions where eigenvalues are real with at least one positive value corresponding to the Turing spaces in light of Theorems \ref{theo1} and \ref{Maincond}. From the parameters chosen according to Fig. \ref{TuringTheo1}, we expect the system dynamics to exhibit exclusively spatial patterns. Fig. \ref{TuringTheo1} (a) shows that an increase in the self-diffusion coefficient $d$ results in the enlargement of the Turing regions. In Fig. \ref{TuringTheo1} (b), we observe that increasing $d$ with the negative cross-diffusion coefficient $d_u$, entails the expansion of the Turing regions. It is important to note that, the main difference between Fig. \ref{TuringTheo1} (a) and Fig. \ref{TuringTheo1} (b) is the choice of cross-diffusion coefficient $d_u$ being positive or negative. Figs. \ref{TuringTheo1} (a)-(b)-(c) are generated by fixing $d$ and $d_v$ and considering $d_u$ being positive, zero and negative respectively. Fig. \ref{TuringTheo1} (f) shows the effect of $d_v$ by fixing $d$ and $d_u$. We observe that an increase in self- and cross-diffusion coefficients has a positive effect on the enlargement of Turing spaces. Fig. \ref{TuringTheo2} presents the Turing spaces conditional to Theorems \ref{theo2} and \ref{Maincond}. It is essential to note that,  Fig. \ref{TuringTheo2} considers the fact that the condition of Theorem \ref{theo2} forbids the existence of Hopf and transcritical type of bifurcations, allowing only spatial patterns in the system dynamics. Fig. \ref{TuringTheo2} is generated considering that the eigenvalues are real with at least one of them being positive. Fig. \ref{TuringTheo2} (a) presents the effect of self-diffusion coefficient $d$, and Fig. \ref{TuringTheo2} (b) shows the effect of cross-diffusion coefficient $d_v$. We observe that, unlike the Hopf bifurcation regions, an increase in self-diffusion and cross-diffusion coefficients results in an enlargement of the Turing spaces. 
\begin{figure}[H]
\centering
\begin{tabular}{cc}
\includegraphics[width=0.28\textwidth]{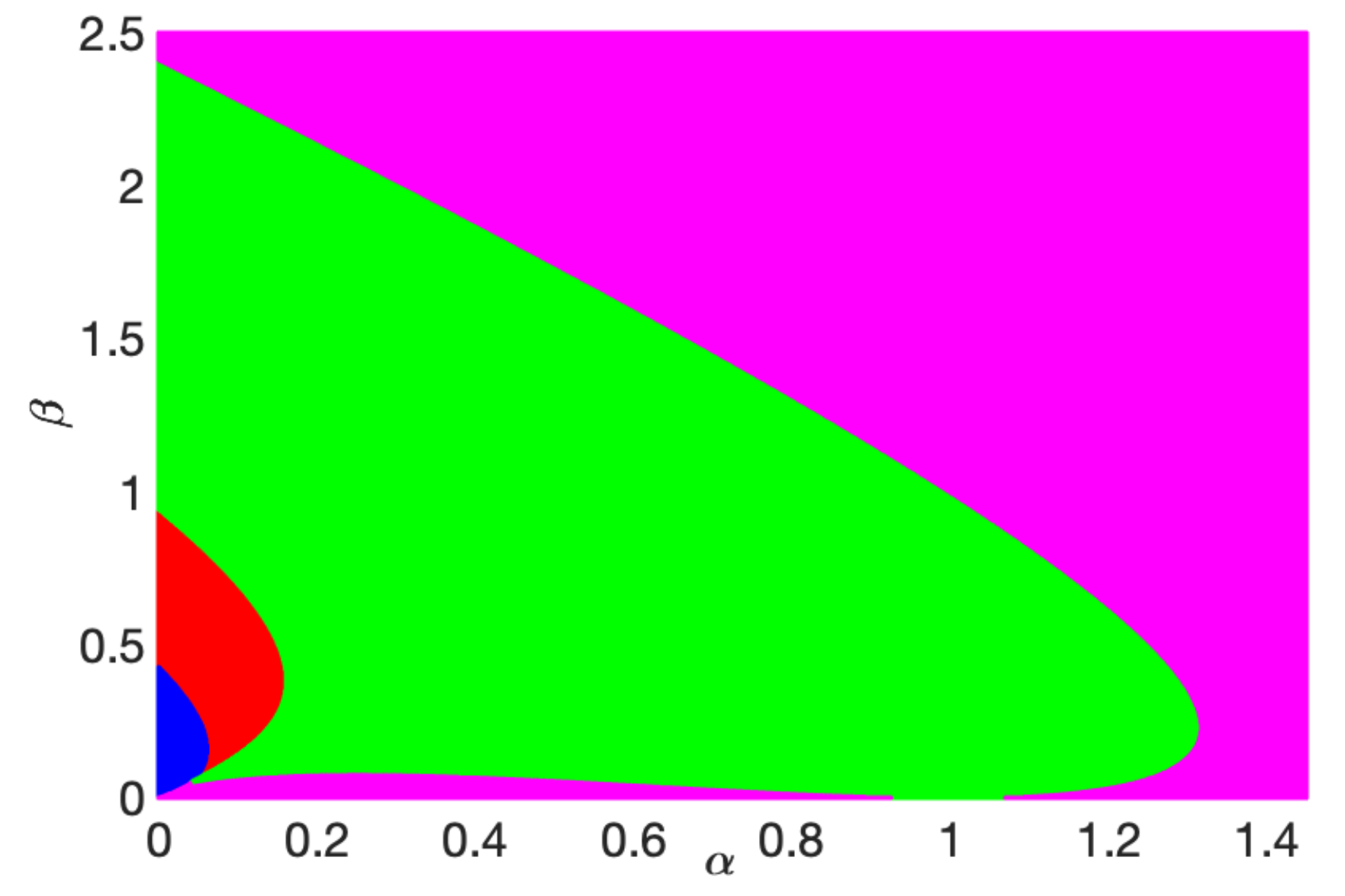} \\
{\small (a) }
\end{tabular}
\begin{tabular}{cc}
\includegraphics[width=0.31\textwidth]{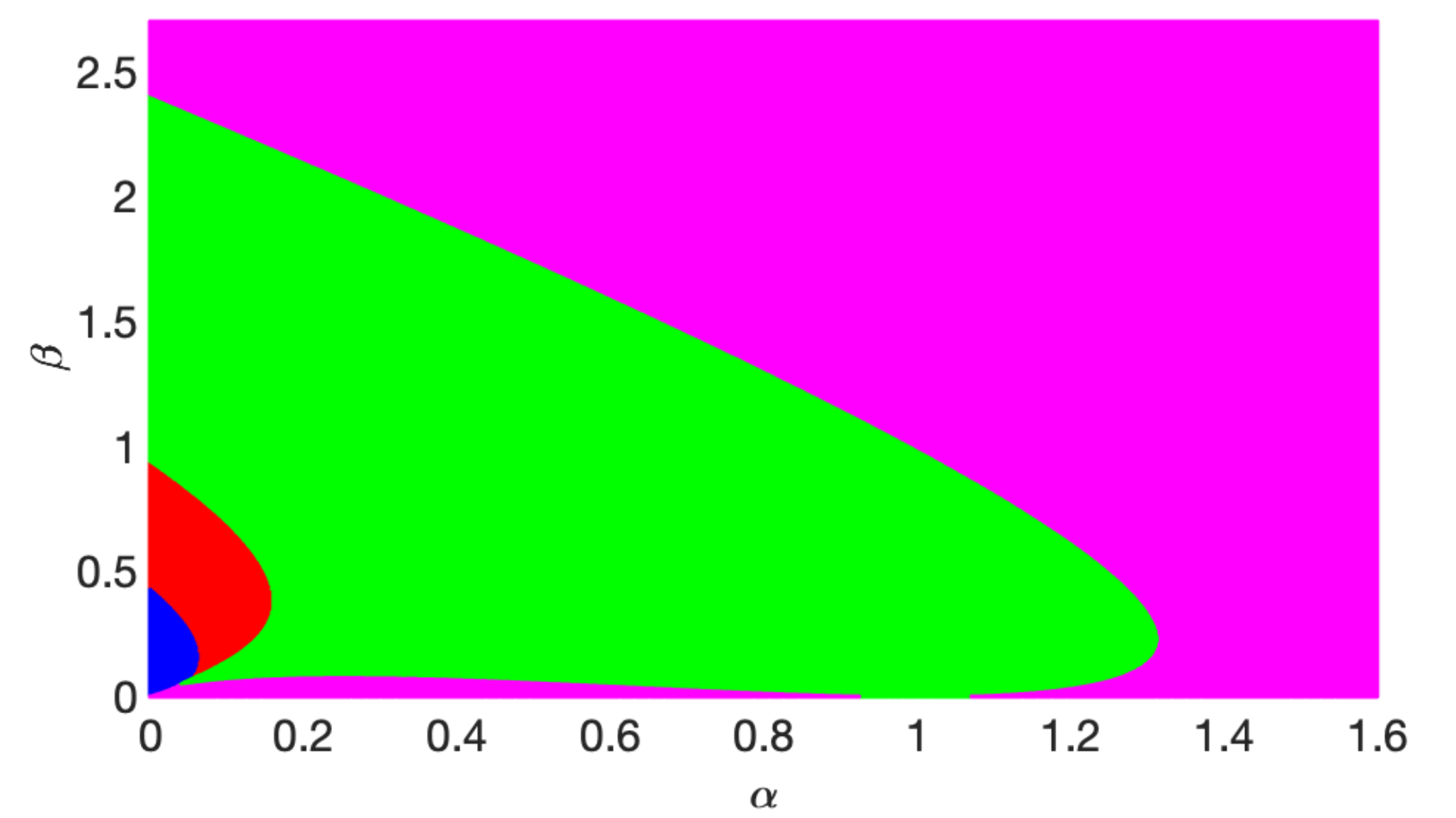} \\
{\small (b)}
\end{tabular}
\begin{tabular}{cc}
\includegraphics[width=0.3\textwidth]{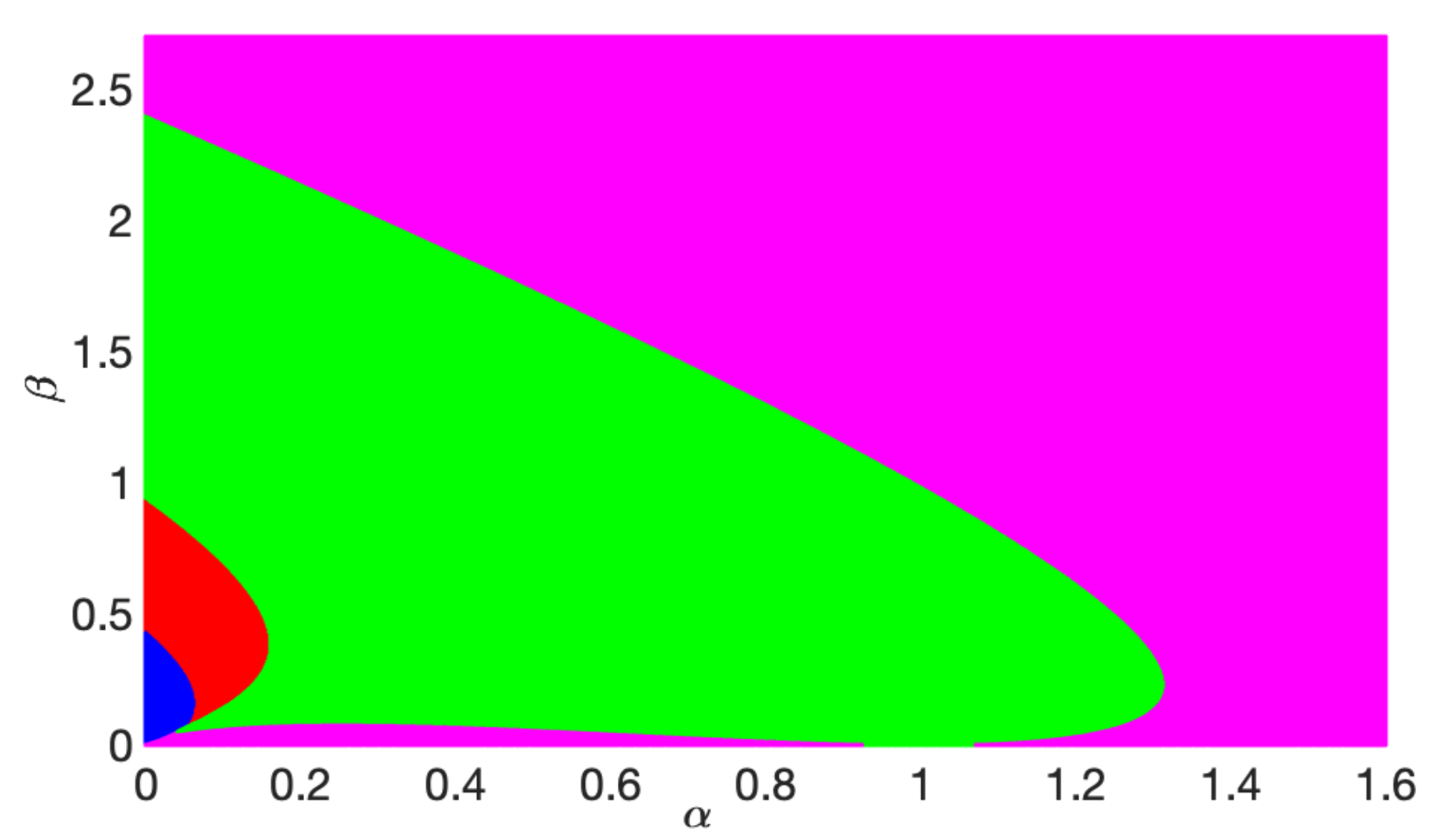} \\
{\small (c)}
\end{tabular}
\begin{tabular}{cc}
\includegraphics[width=0.3\textwidth]{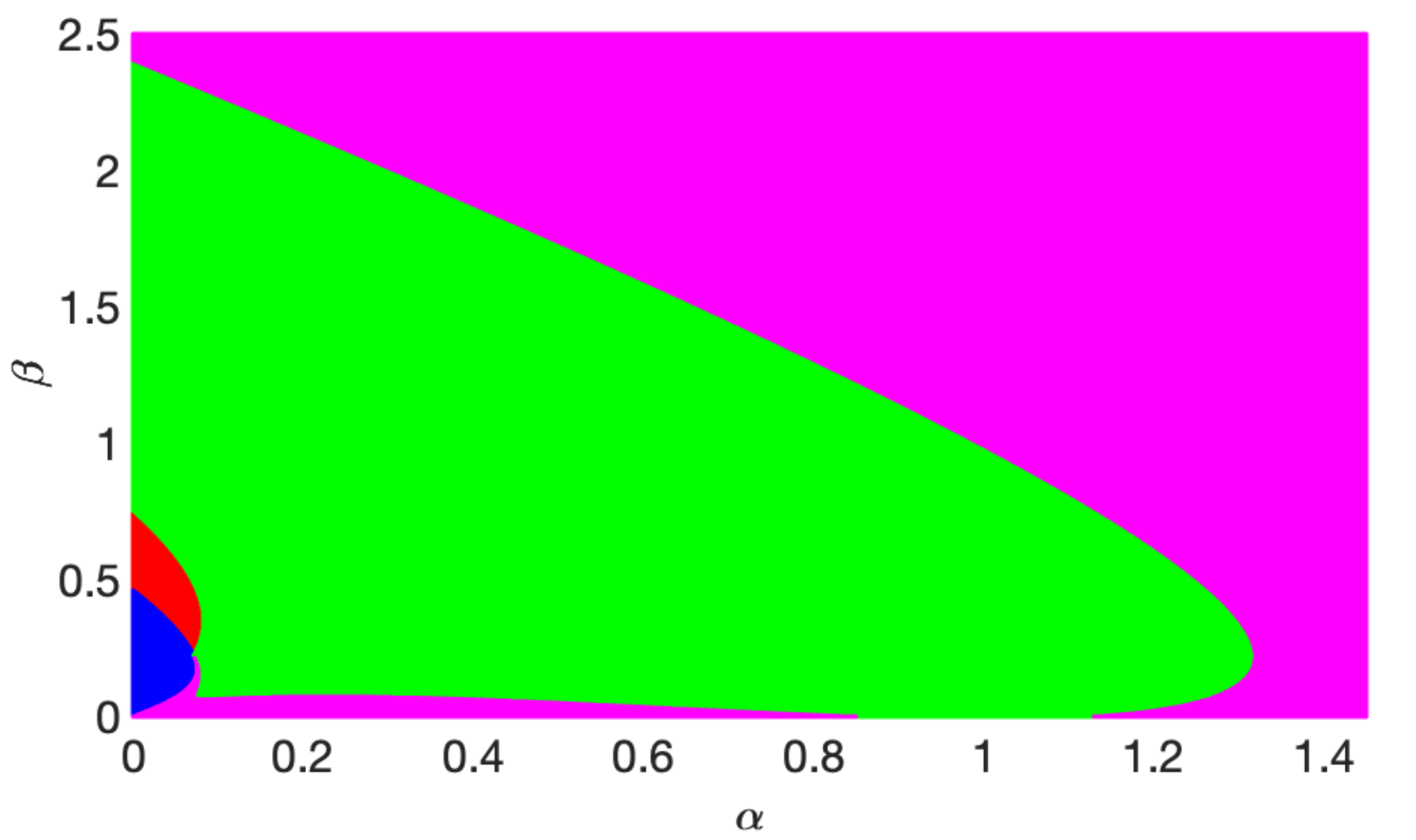} \\
{\small (d) }
\end{tabular}
\begin{tabular}{cc}
\includegraphics[width=0.3\textwidth]{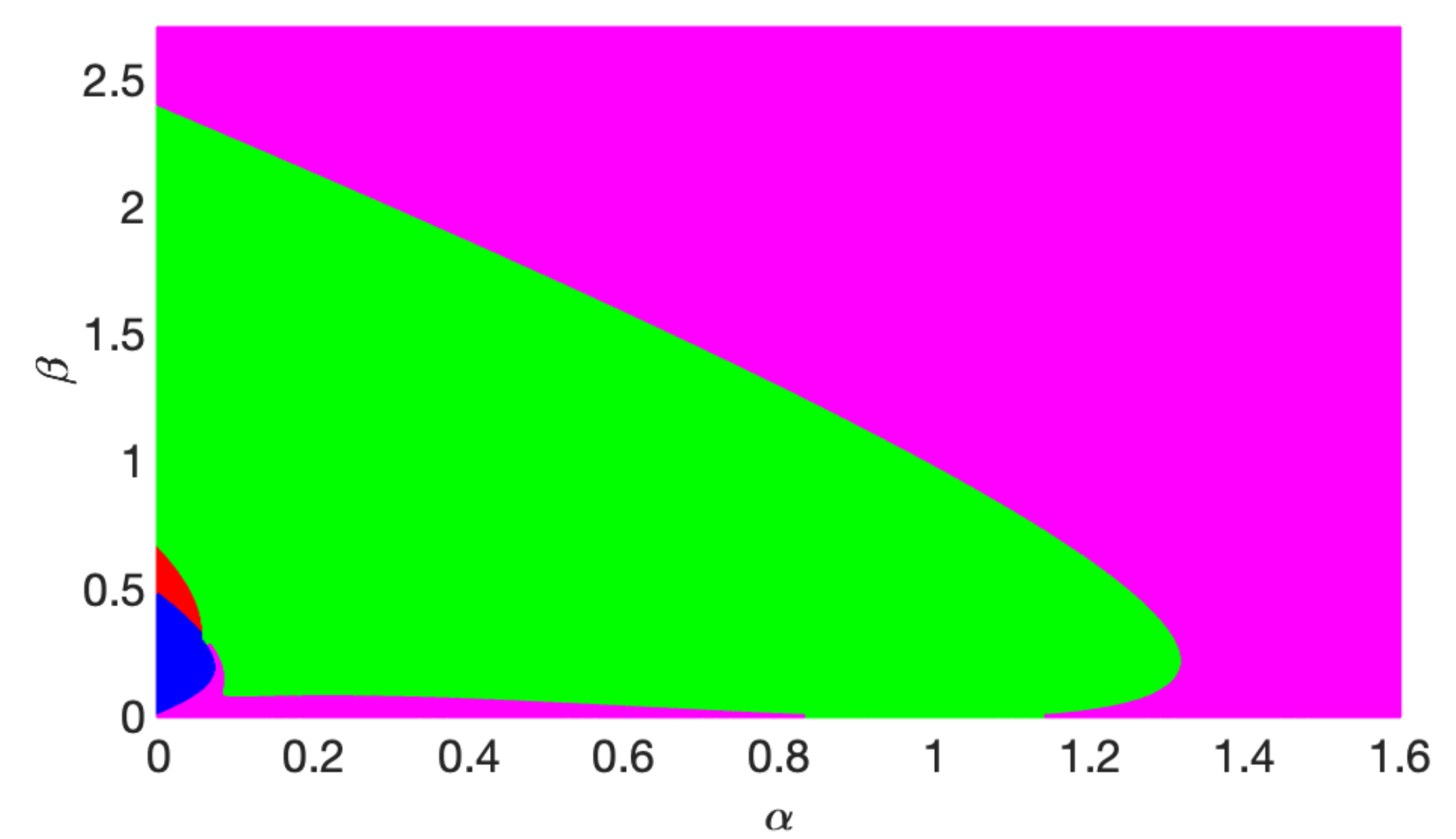} \\
{\small (e) }
\end{tabular}
\begin{tabular}{cc}
\includegraphics[width=0.3\textwidth]{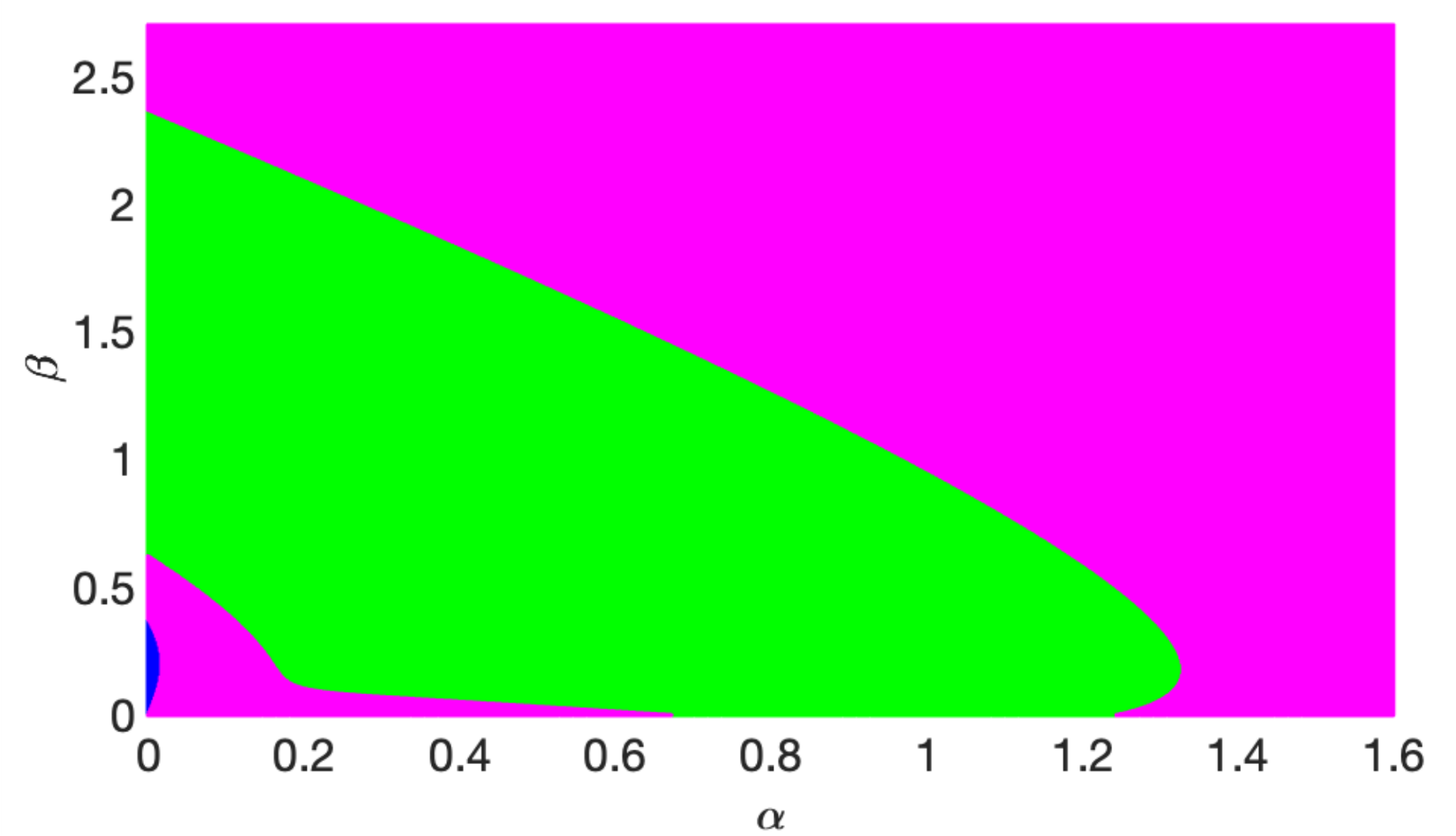} \\
{\small (f)}
\end{tabular}
\begin{tabular}{cc}
\includegraphics[width=0.3\textwidth]{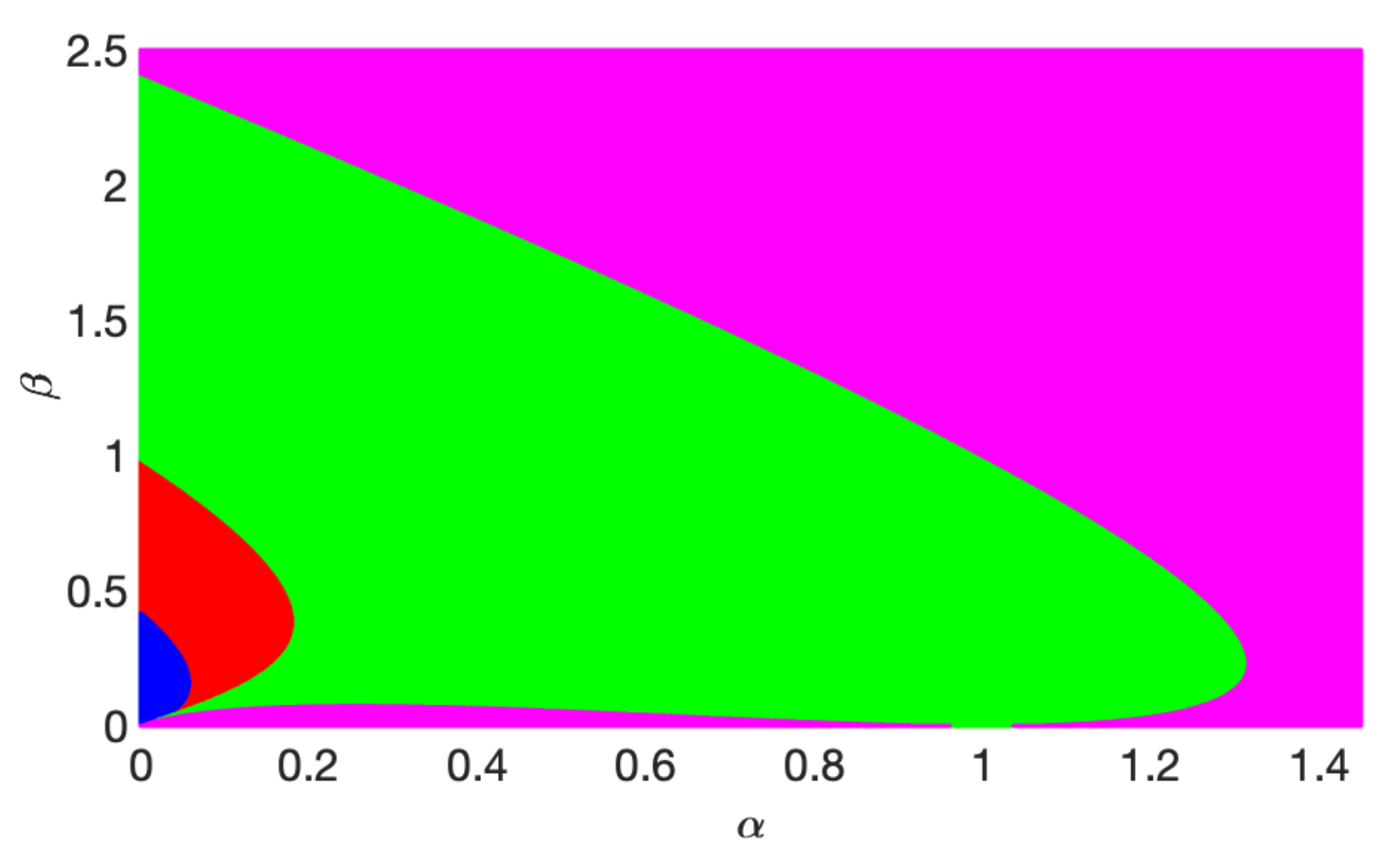} \\
{\small (g) }
\end{tabular}
\begin{tabular}{cc}
\includegraphics[width=0.3\textwidth]{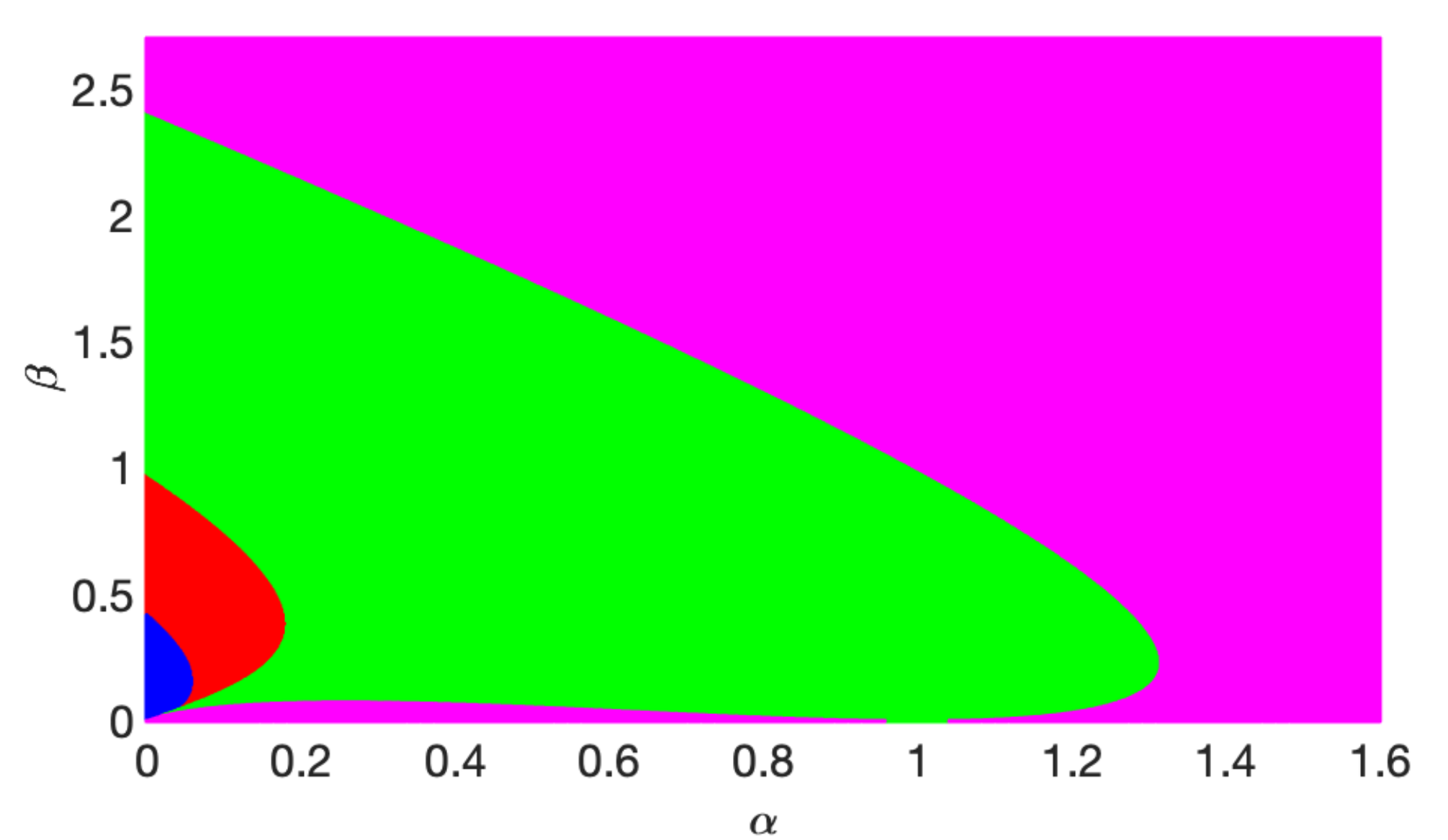} \\
{\small (h) }
\end{tabular}
\begin{tabular}{cc}
\includegraphics[width=0.3\textwidth]{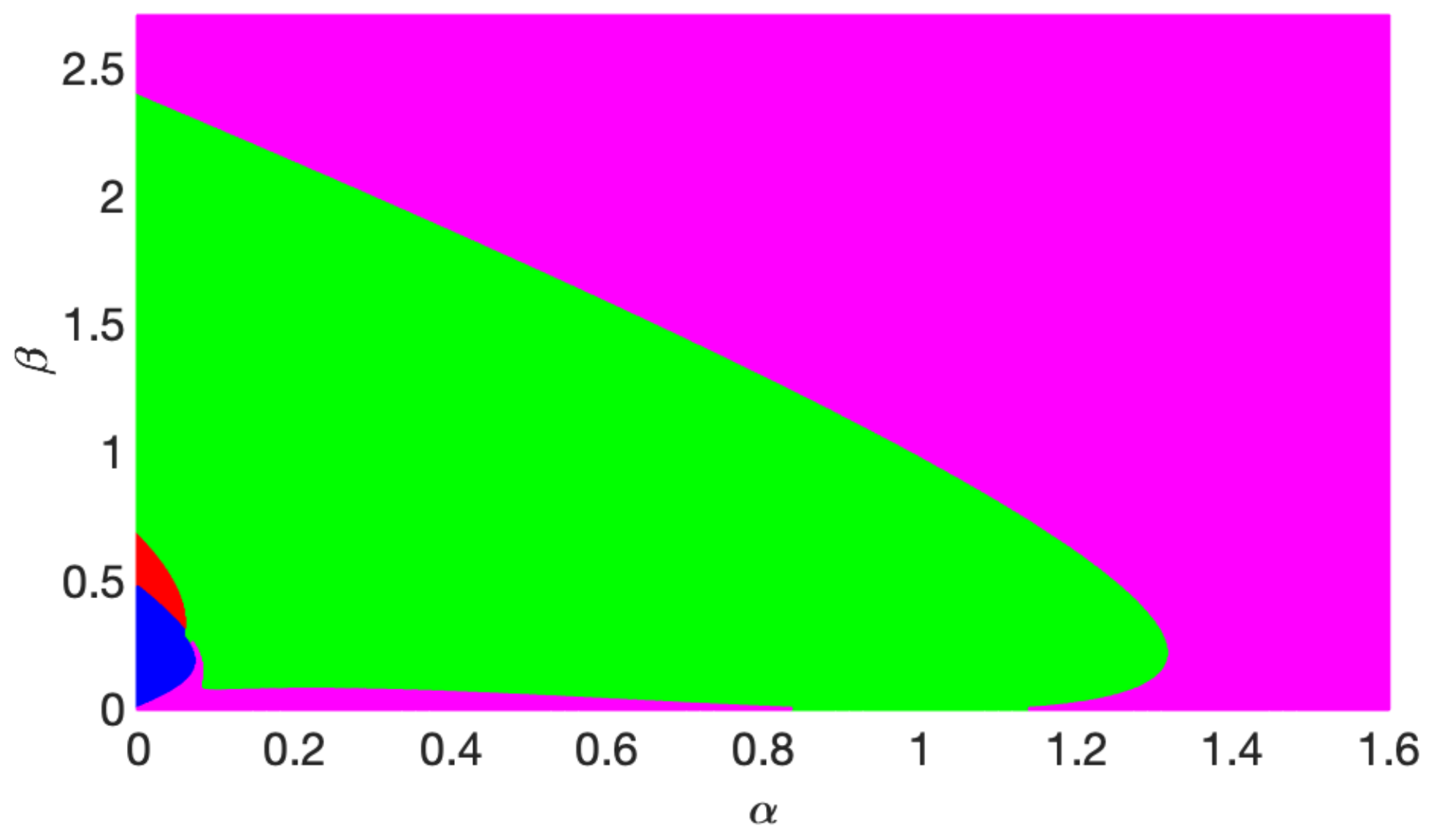} \\
{\small (i)}
\end{tabular}
\begin{tabular}{cc}
\includegraphics[trim = 0mm 0mm 0mm 0mm,  clip, width=0.3\textwidth]{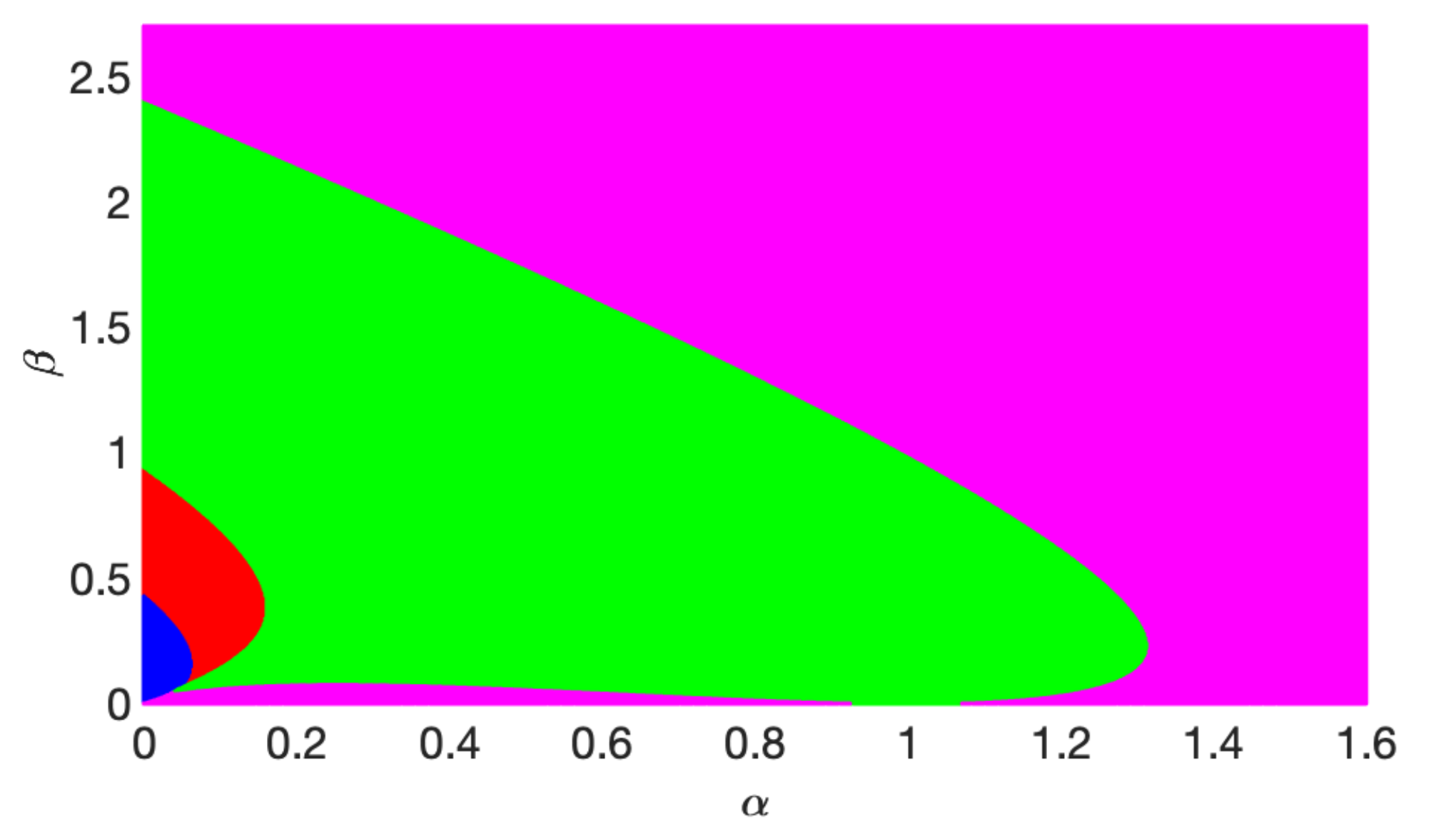} \\
{\small (j) }
\end{tabular}
\begin{tabular}{cc}
\includegraphics[trim = 0mm 0mm 0mm 0mm,  clip, width=0.3\linewidth]{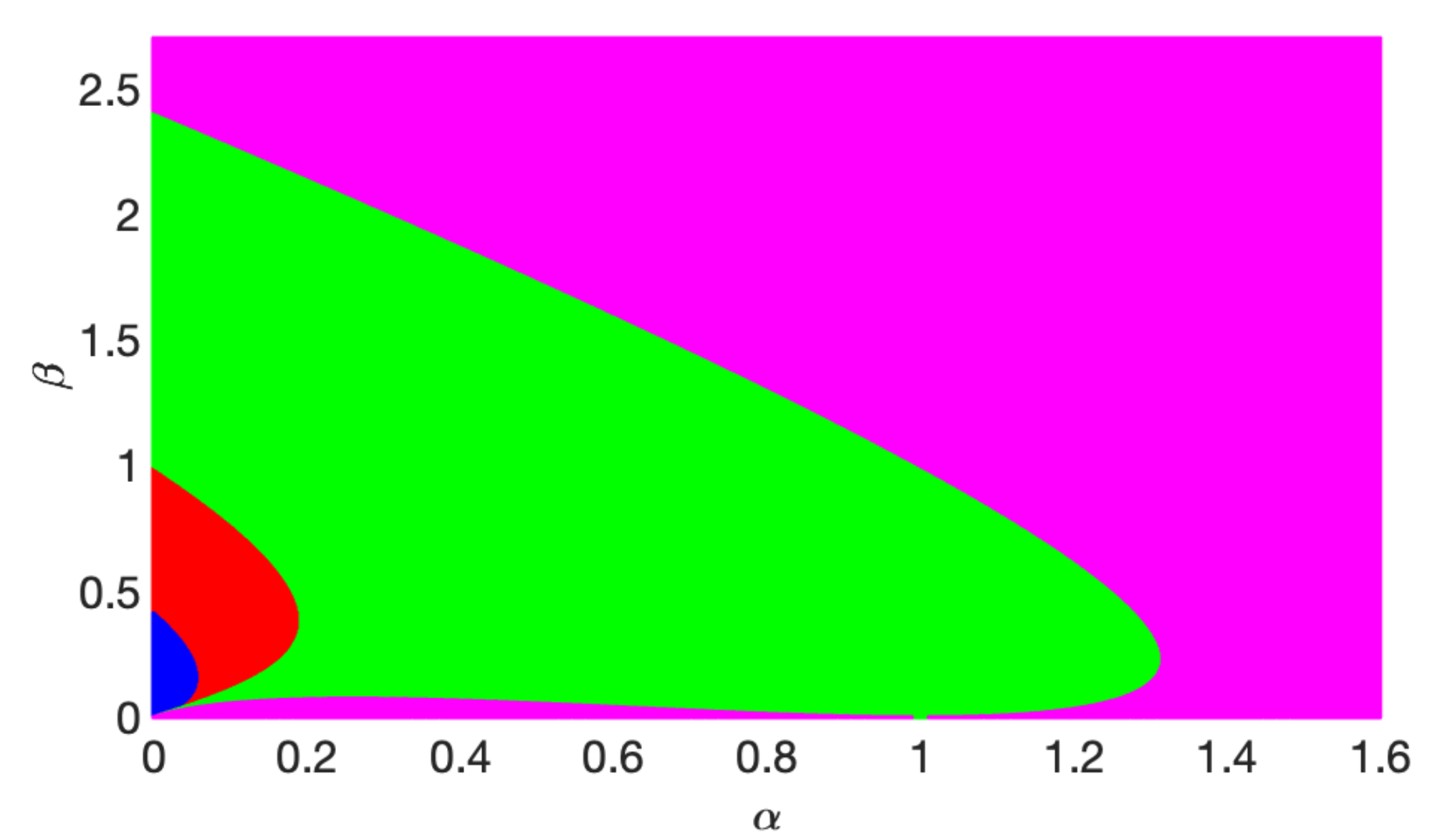}\\
{\small (k)}
\end{tabular}
\begin{tabular}{cc}
\includegraphics[trim = 0mm 0mm 0mm 0mm,  clip, width=0.3\textwidth]{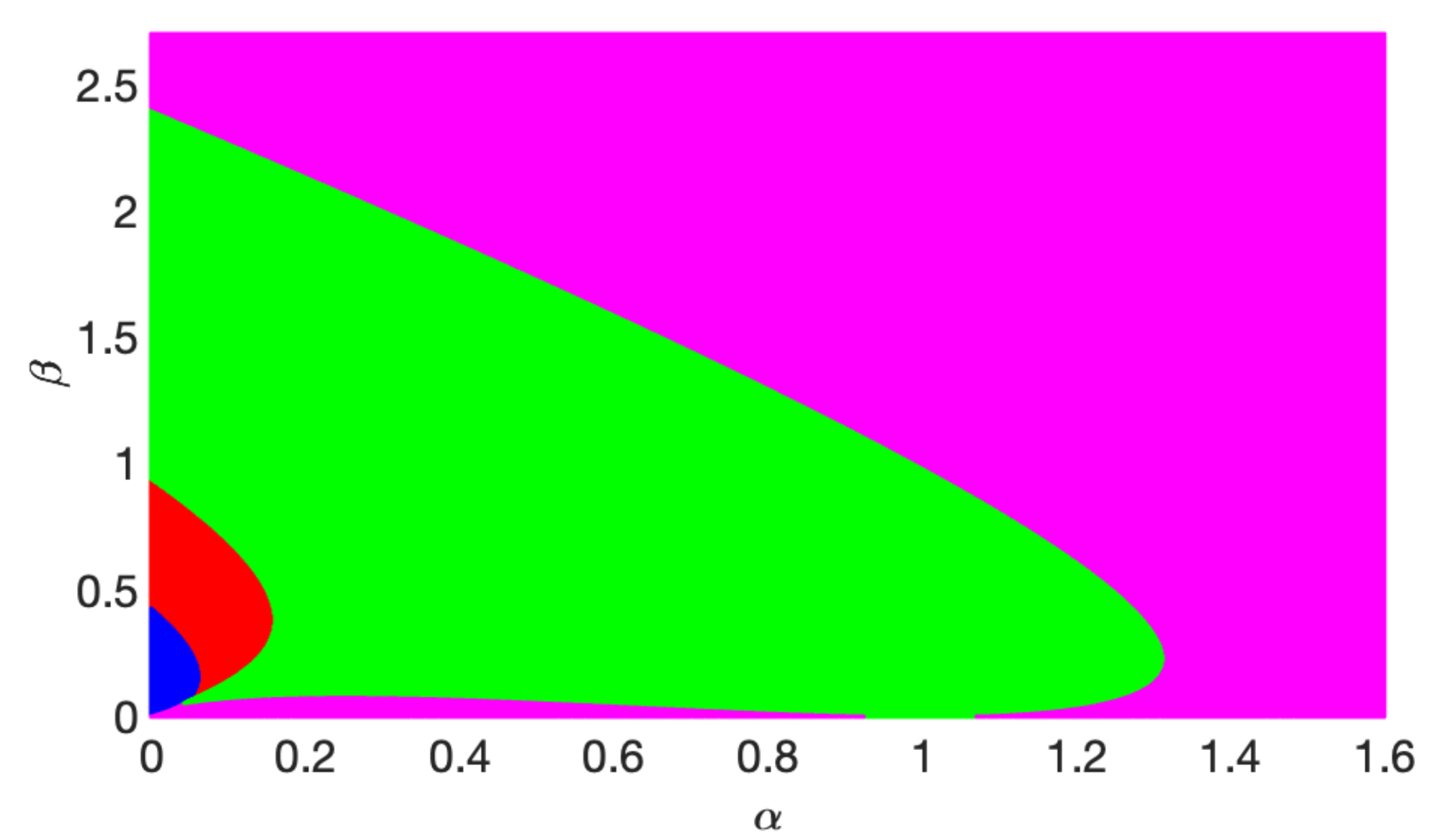} \\
{\small (l)}
\end{tabular}
\caption{First column (rectangular domain), Second column (disc-shape domain), Third Column (ring-shape domain). Full parameter classification with equal case of (a)-(c) $k^2(m,n)=1$, (d)-(f) equal case of domain size $L=1$, $\rho=1$ (g)-(i) equal case domain sizes as $L=3.8$, $\rho=3.8$, (j)-(l) equal case of domain sizes $L=15$ and $\rho=15$. Colour coding; magenta regions correspond to the real-distinct negative eigenvalues, green regions correspond to a complex-conjugate pair with a negative real part, red regions correspond to a complex-conjugate pair with a positive real part, the blue region corresponds to a real-distinct pair with at least one positive eigenvalue.}
\label{fullclass1}
\end{figure}
\begin{table}[ht]
\tbl{Spatiotemporal pattern formation conditions for cross-diffusive reaction-diffusion systems on two-dimensional geometries}
{\begin{tabular}{l c  c c l}\\[-1pt]
\toprule
Type of geometry&Eigenmodes $k^2_{m,n}$&Sufficient condition on domain size\\[4pt]
\hline\\[-1pt]
Rectangular &$k^2_{m,n}=\frac{(m^2+n^2)\pi^2}{L^2}$  &$L^2> \frac{(d-d_ud_v) (m^2+n^2)\pi^2}{(7d+8d_v)\gamma}$ \\[6pt]
Circular (disc)  &$k^2_{m,n}=\frac{4(2m+1)(n+2m+1)(n+4m)}{\rho^2(n+4m+2)}$ & $ \rho^2> \frac{4(d-d_ud_v) (2m+1)(n+2m+1)(n+4m)}{(7d+8d_v)(n+4m+2)\gamma}$\\[6pt]
Flat ring (annulus)  &$k_{m,n}^2=\frac{8(2m+1)(n+2m+1)(n+4m)}{a(\rho+a)(n+4m+2)}$ &$\rho> \frac{8(d-d_ud_v) (2m+1)(n+2m+1)(n+4m)-\gamma a^2(7d+8d_v)(n+4m+2)}{(7d+8d_v)(n+4m+2)\gamma}$\\[2pt]
\botrule
\end{tabular}}
\label{Tab1}
\end{table}

\begin{figure}[H]
\begin{tabular}{cc}
\includegraphics[width=0.45\textwidth]{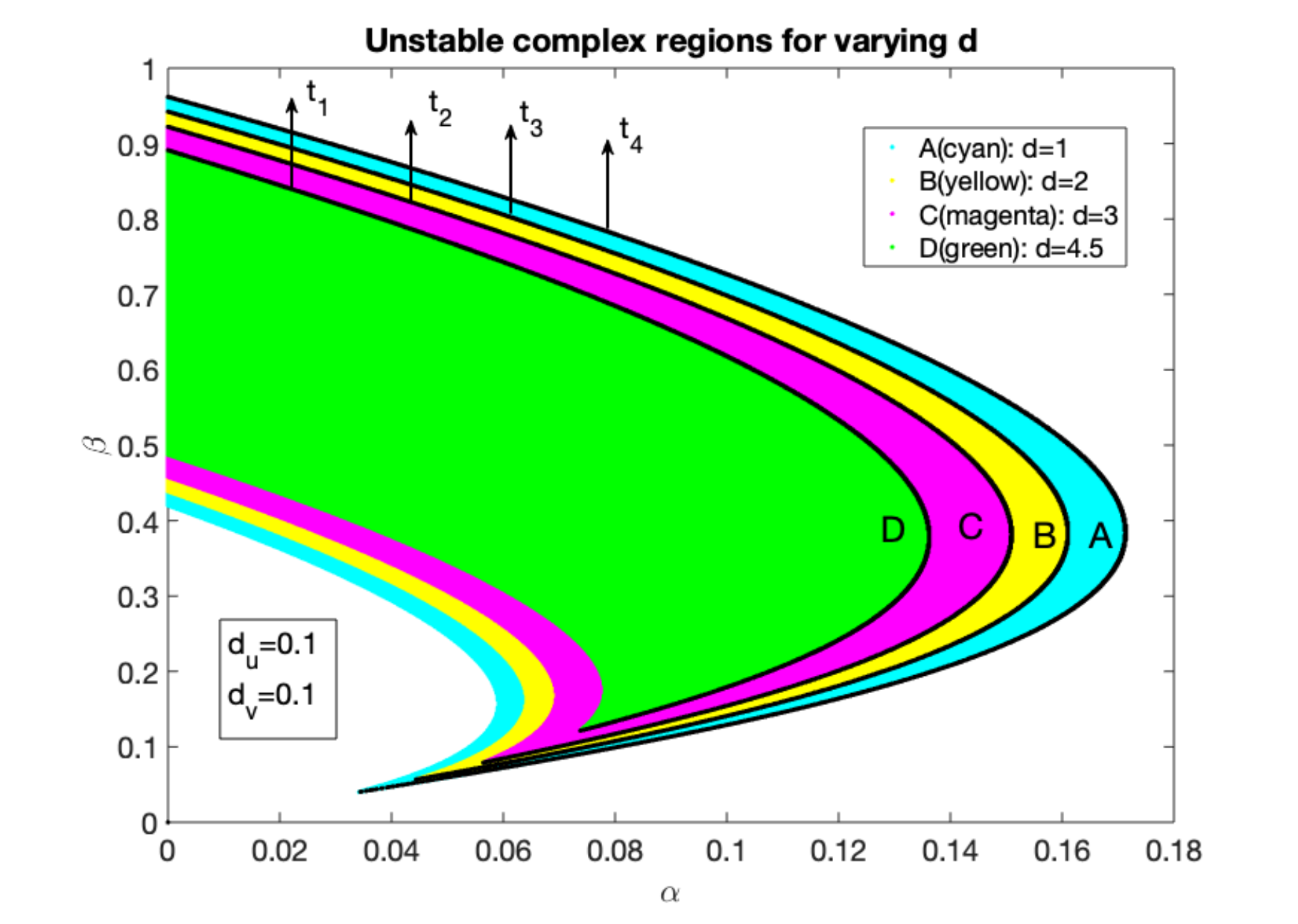} \\
{\small (a) Varying $d$}
\end{tabular}
\begin{tabular}{cc}
\includegraphics[width=0.45\textwidth]{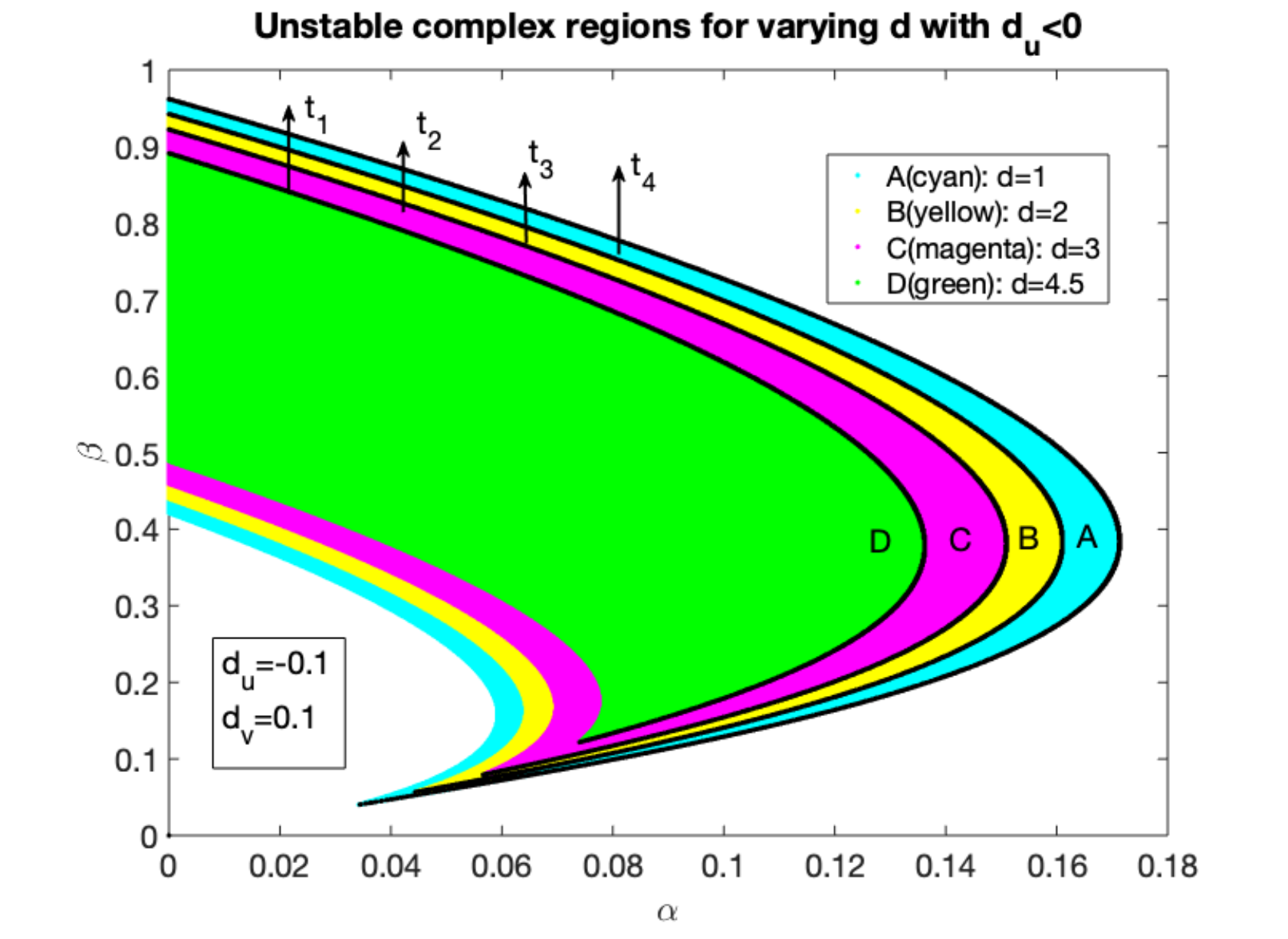} \\
{\small (b) Varying $d$ with $d_u<0$}
\end{tabular}
\begin{tabular}{cc}
\includegraphics[width=0.45\textwidth]{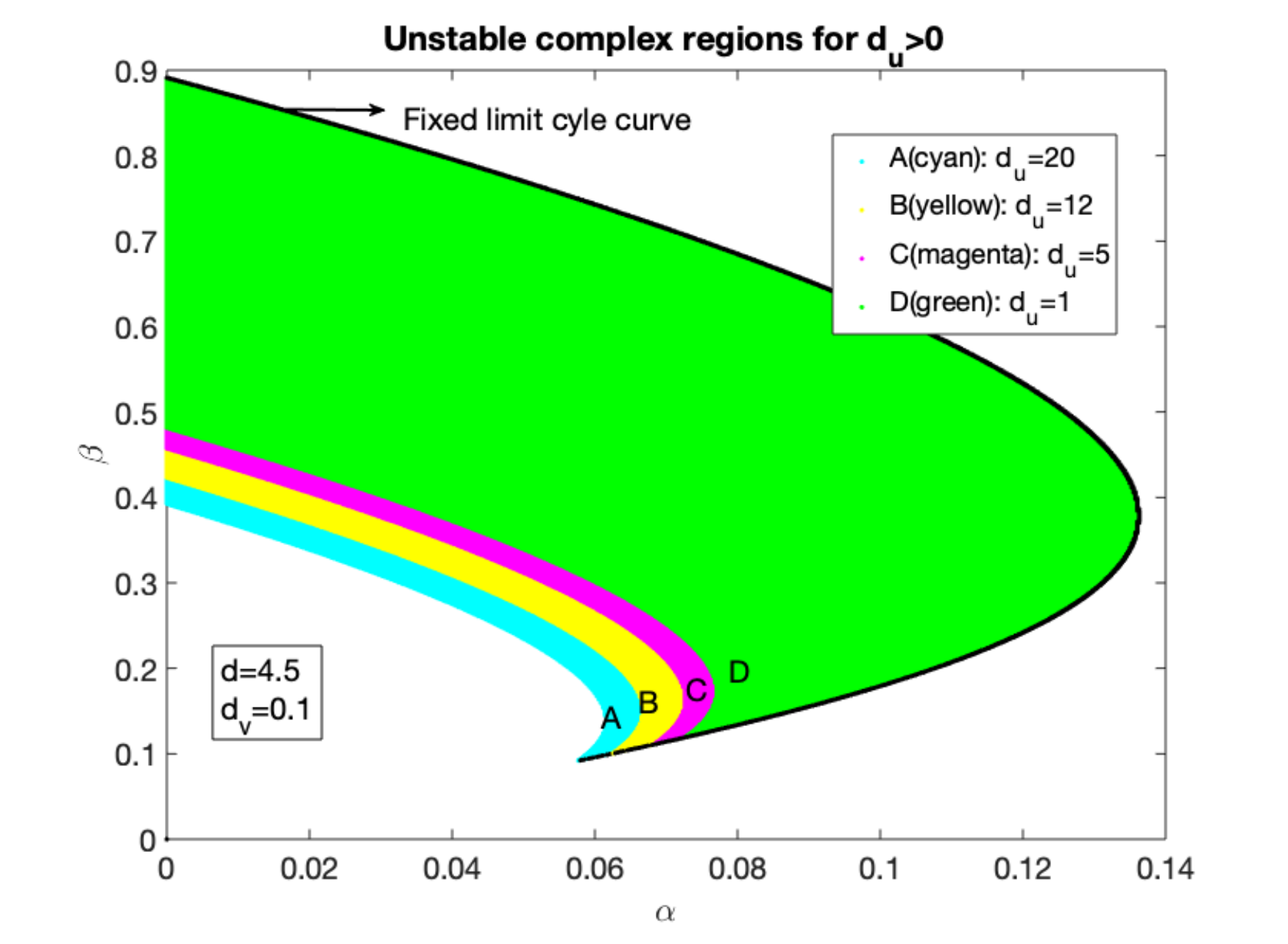} \\
{\small (c) Varying $d_u$ positively}
\end{tabular}
\begin{tabular}{cc}
\includegraphics[width=0.45\textwidth]{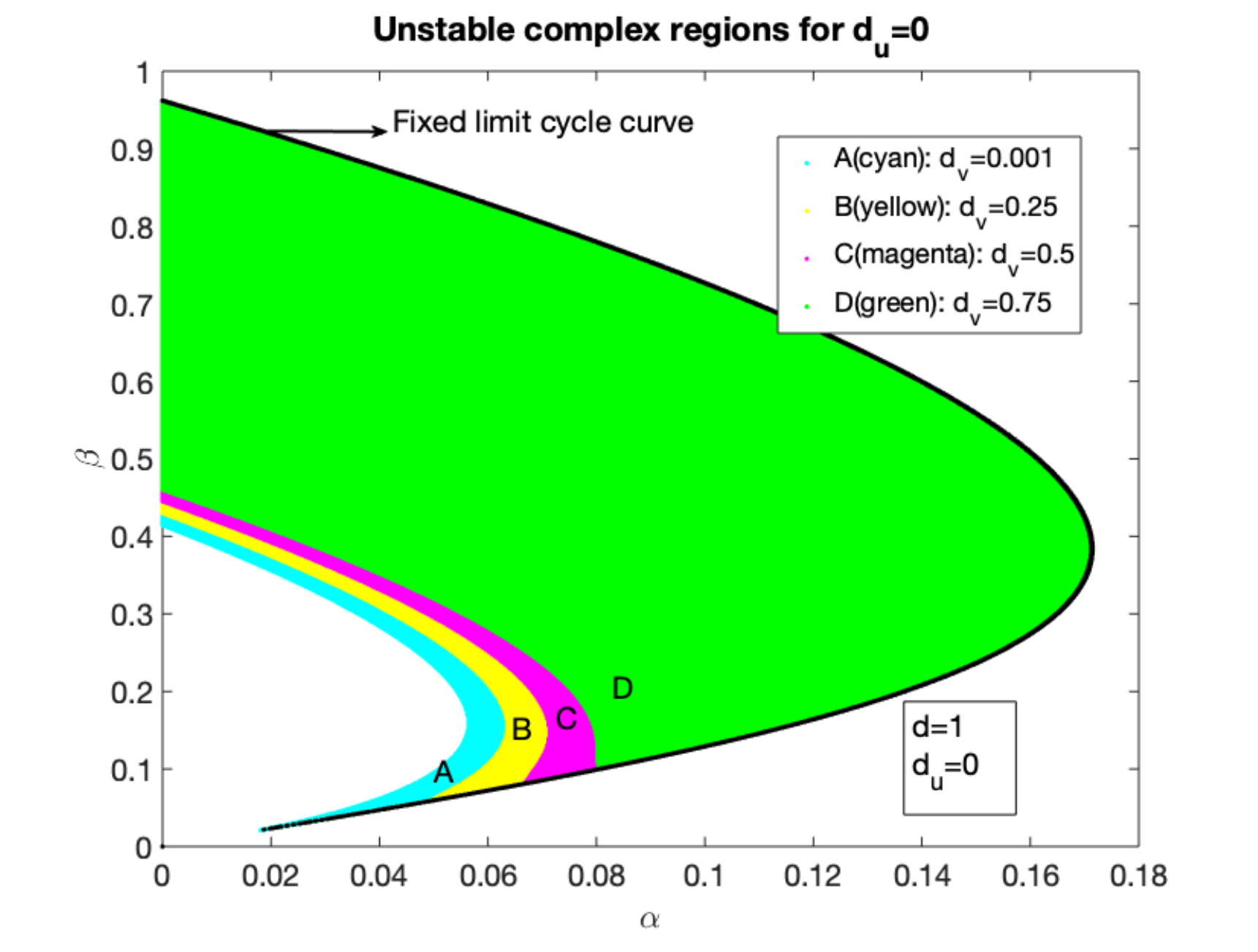} \\
{\small (d) Varying $d_v$ for $d_u=0$}
\end{tabular}
\begin{tabular}{cc}
\includegraphics[width=0.45\textwidth]{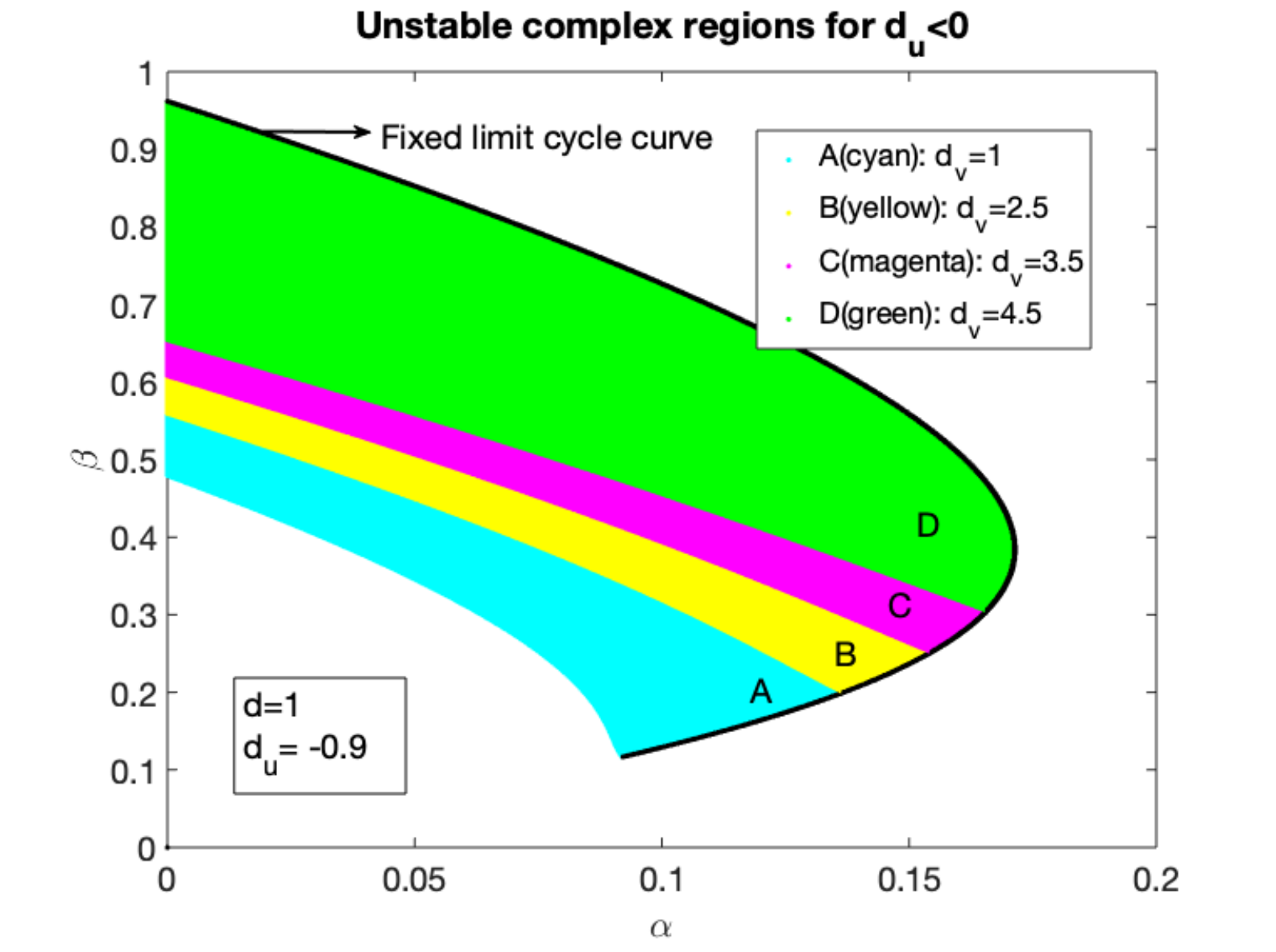} \\
{\small (e) Varying $d_v$ for $d_u<0$}
\end{tabular}
\begin{tabular}{cc}
\includegraphics[width=0.45\textwidth]{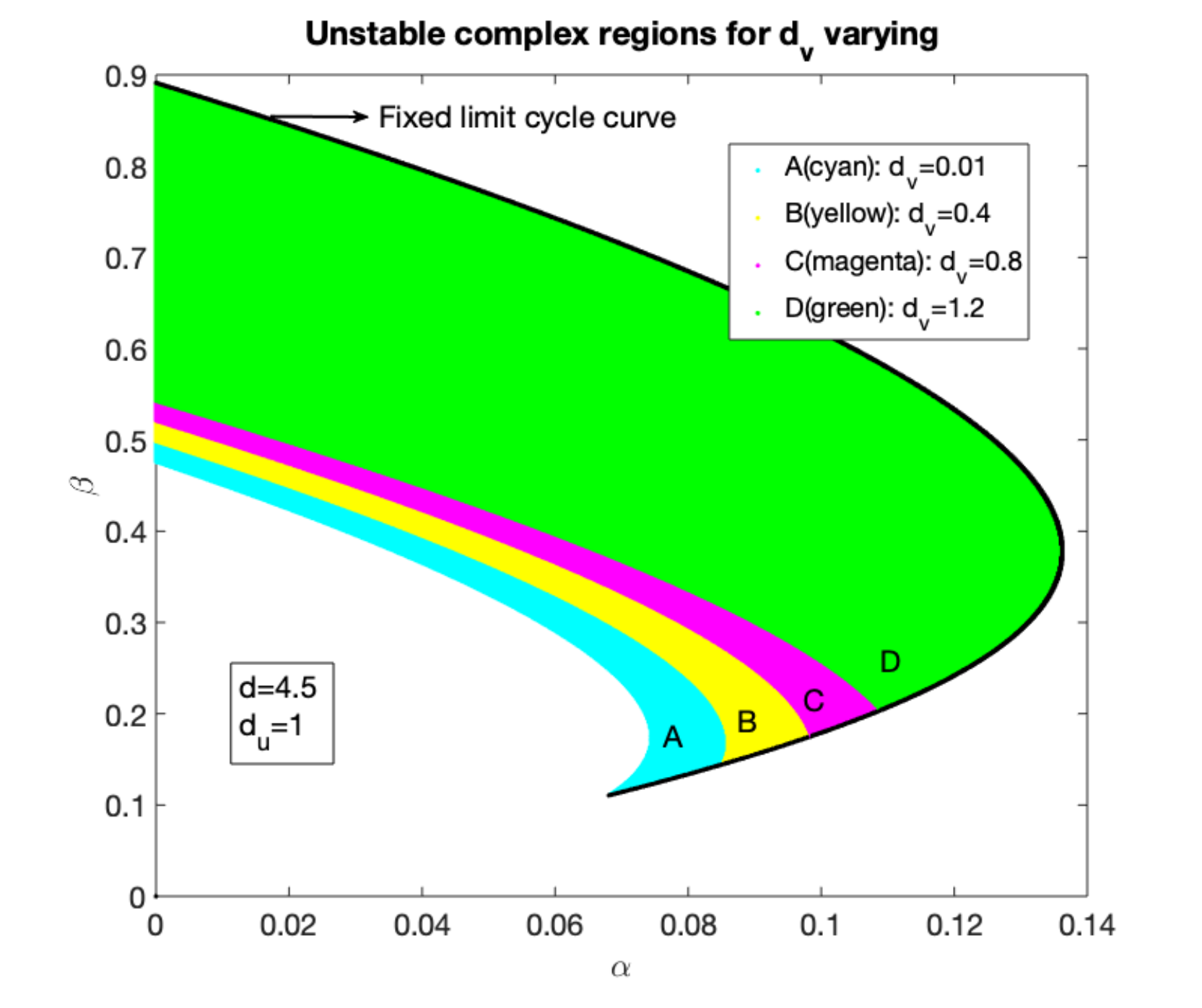} \\
{\small (f) Varying $d_v$ for $d_u=1$}
\end{tabular}
\caption{(a)-(f) Parameter spaces illustrating Hopf bifurcation regions and limit cycle curves with domain-size $\rho$  restricted to satisfy conditions of Theorems \ref{theo1} and  \ref{Maincond} for various system parameters. $n=0.1$, $m=1$, $\rho=14.8$, $\gamma=100$.}
\label{HopfTransParameter}
\end{figure}

\begin{figure}[H]
\begin{tabular}{cc}
\includegraphics[width=0.45\textwidth]{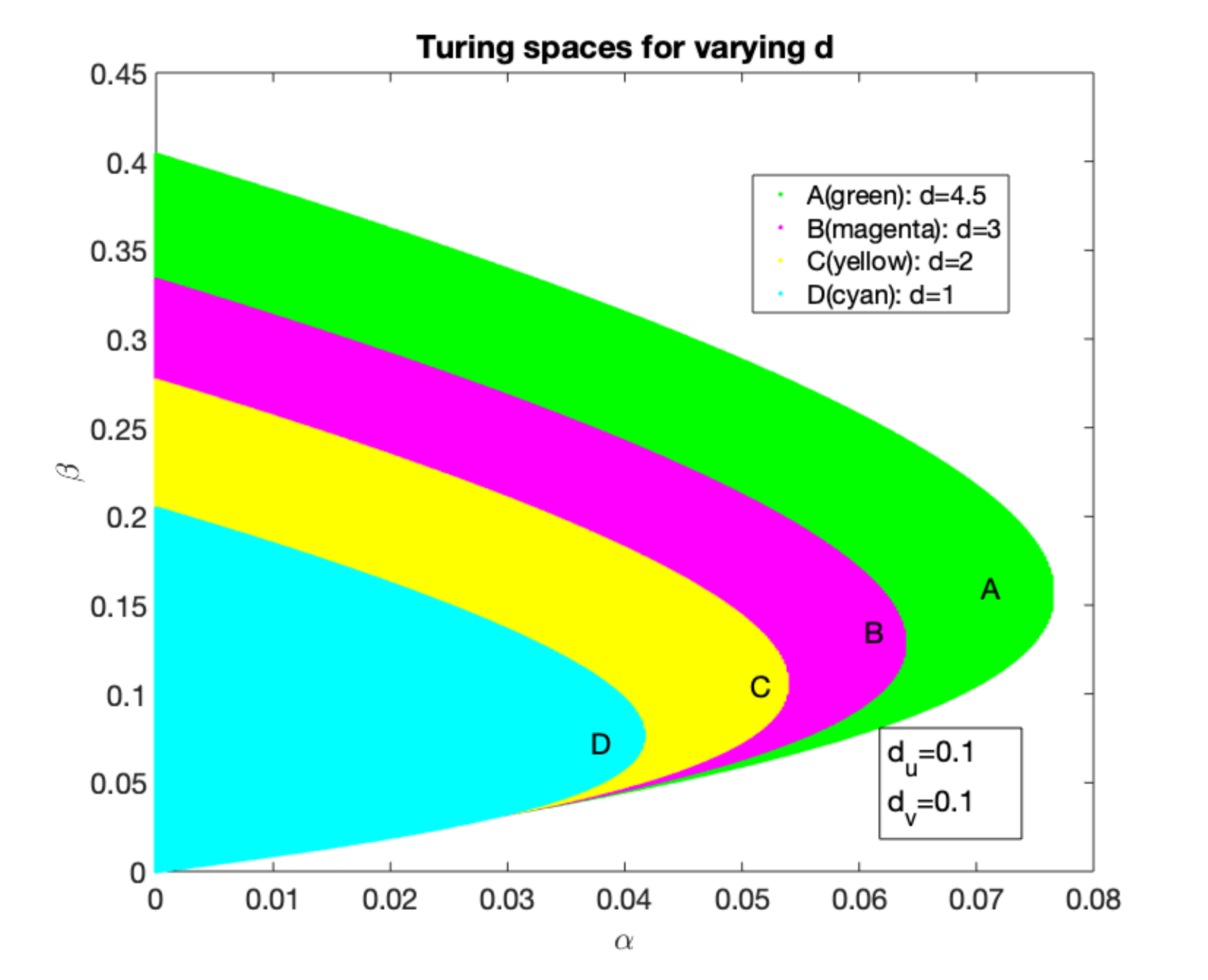} \\
{\small (a) Varying $d$}
\end{tabular}
\begin{tabular}{cc}
\includegraphics[width=0.45\textwidth]{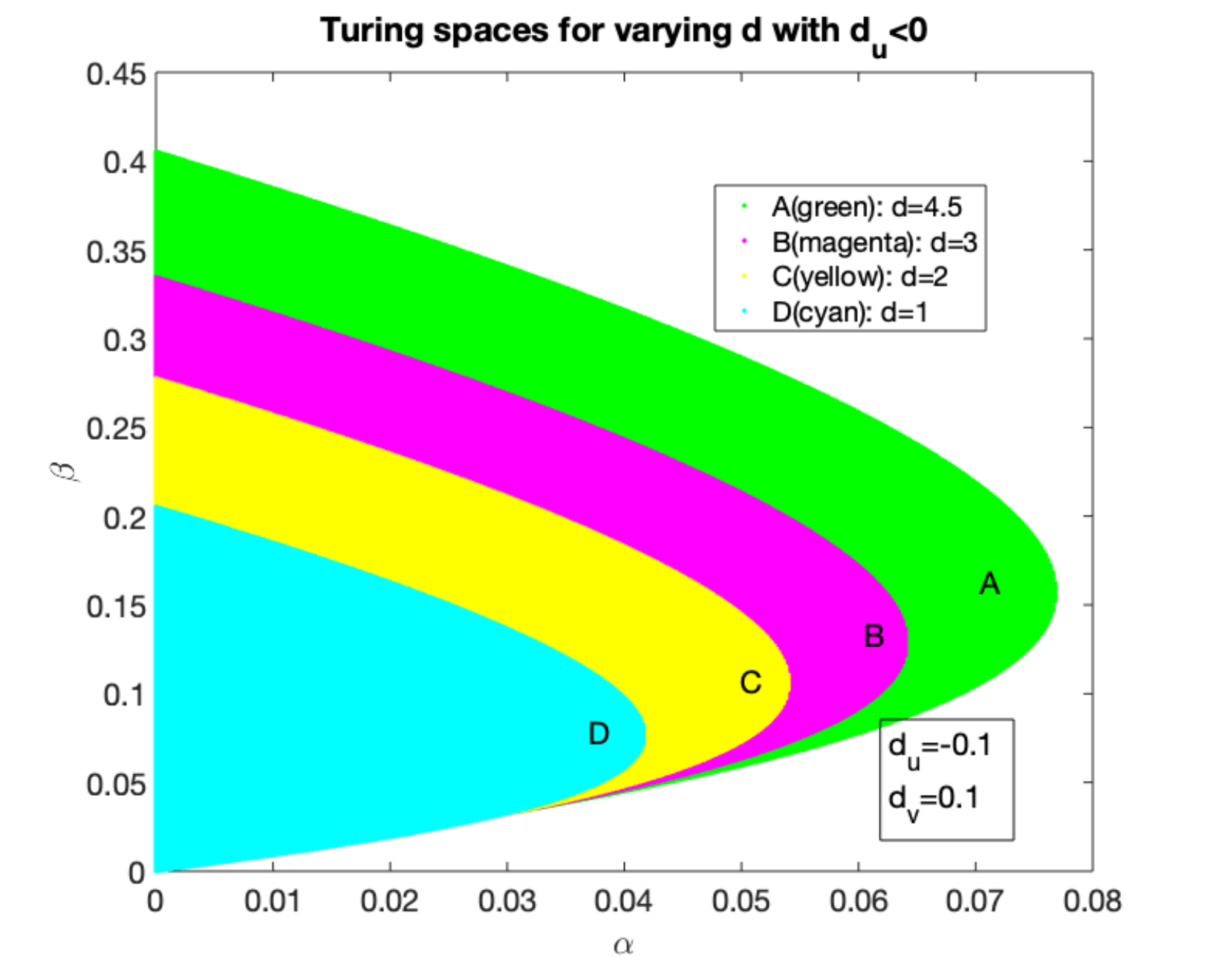} \\
{\small (b) Varying $d$ with $d_u<0$}
\end{tabular}
\begin{tabular}{cc}
\includegraphics[width=0.45\textwidth]{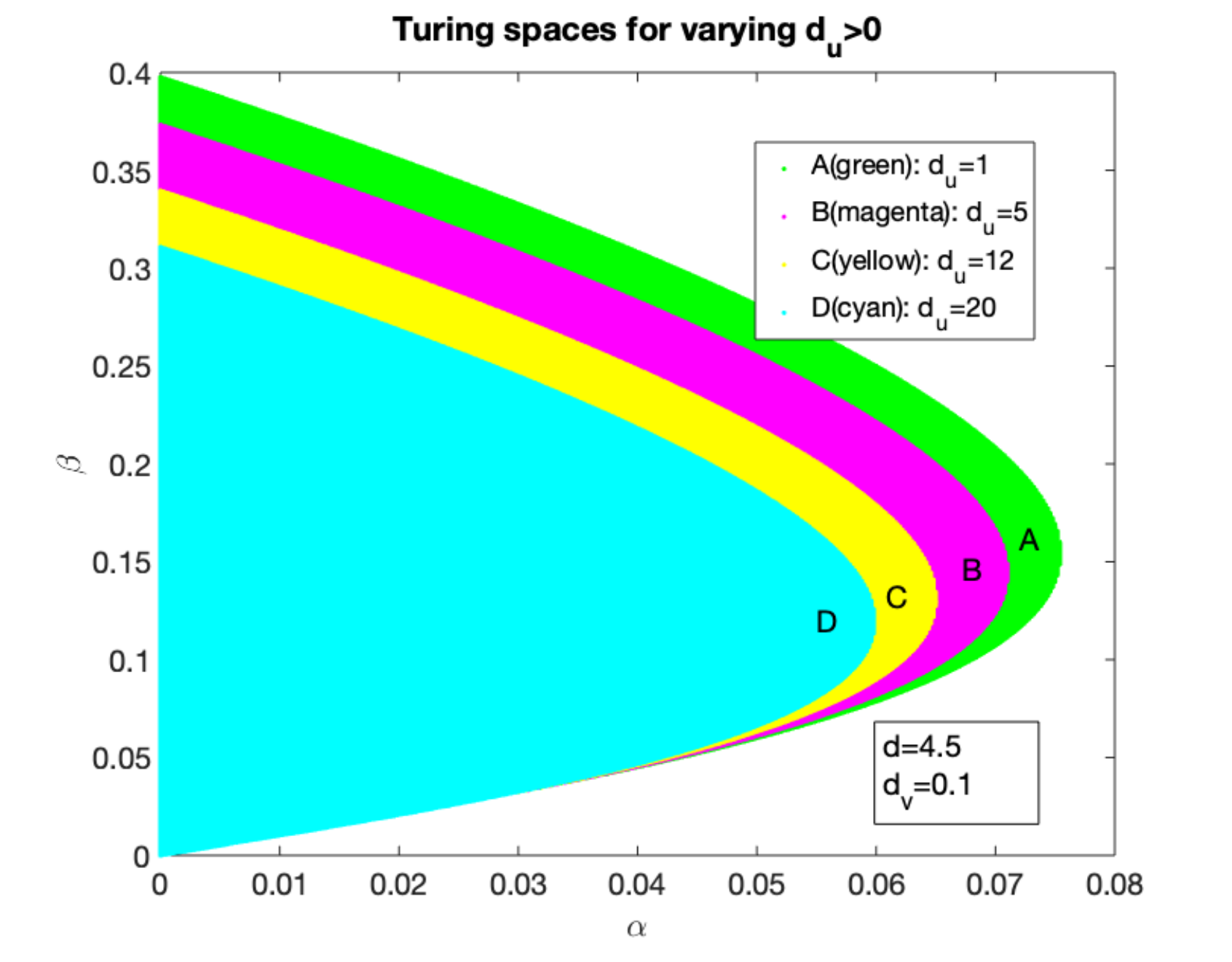} \\
{\small (c) Varying $d_u$ positively}
\end{tabular}
\begin{tabular}{cc}
\includegraphics[width=0.45\textwidth]{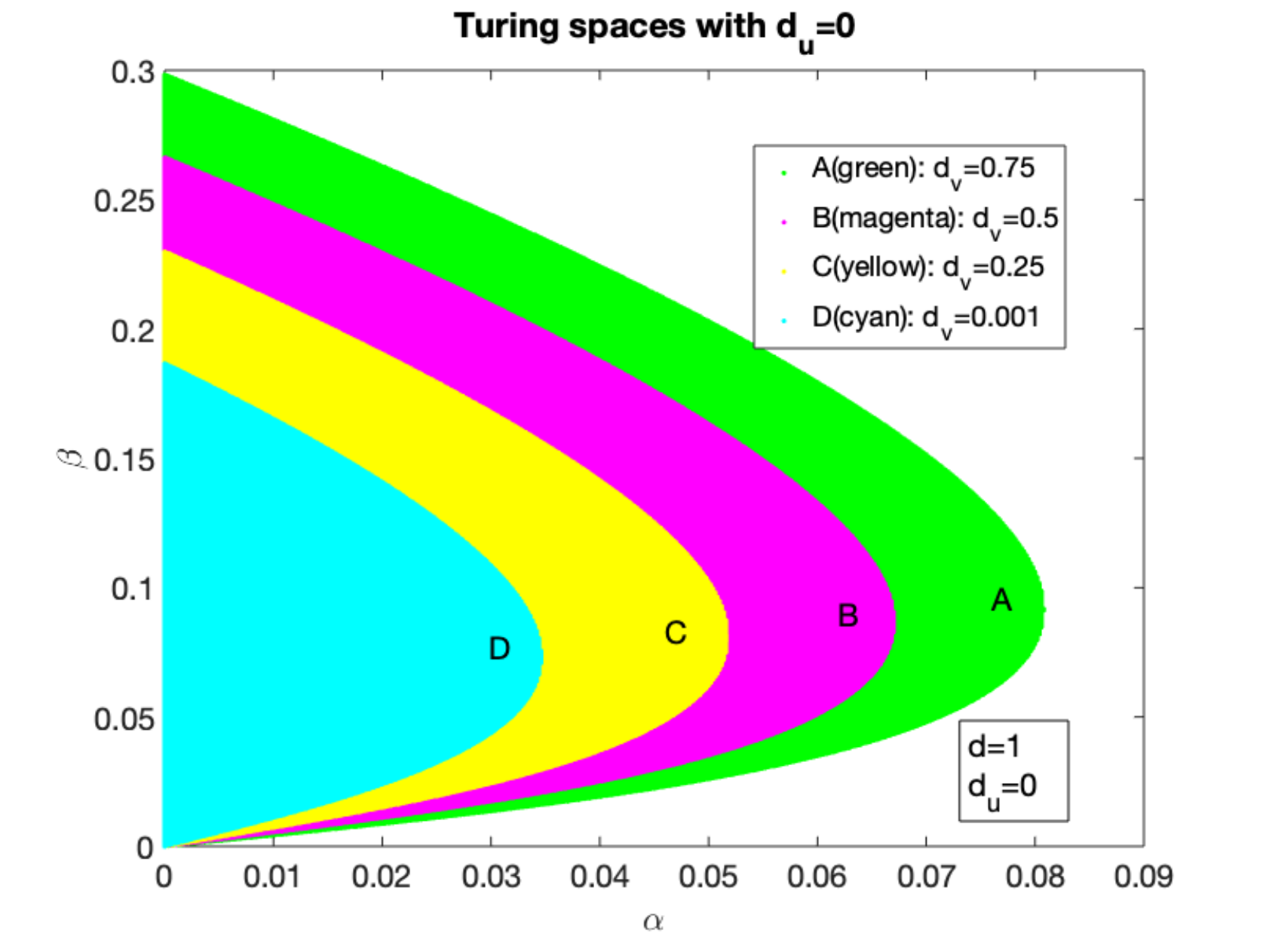} \\
{\small (d) Varying $d_v$ for $d_u=0$}
\end{tabular}
\begin{tabular}{cc}
\includegraphics[width=0.45\textwidth]{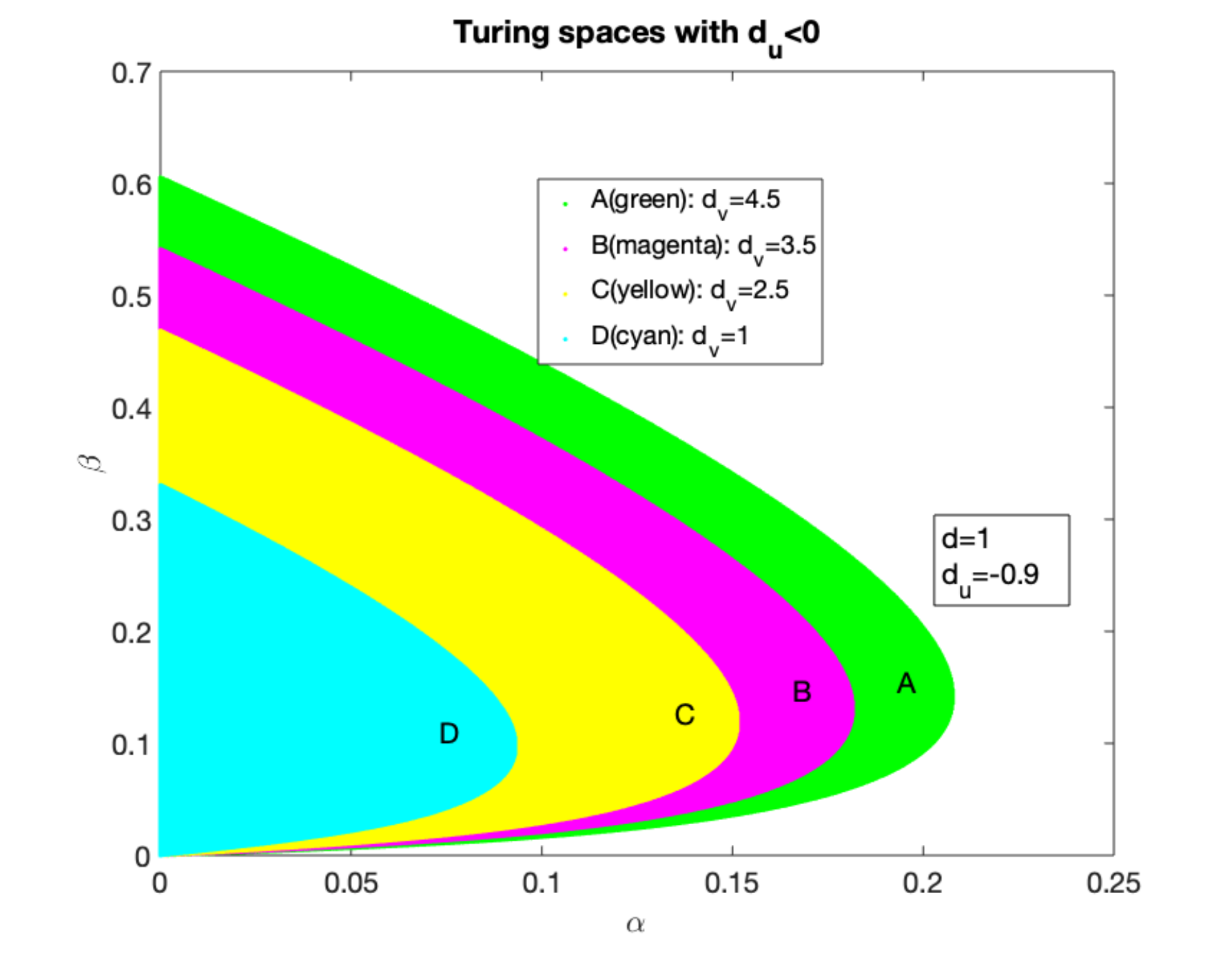} \\
{\small (e) Varying $d_v$ for $d_u<0$}
\end{tabular}
\begin{tabular}{cc}
\includegraphics[width=0.45\textwidth]{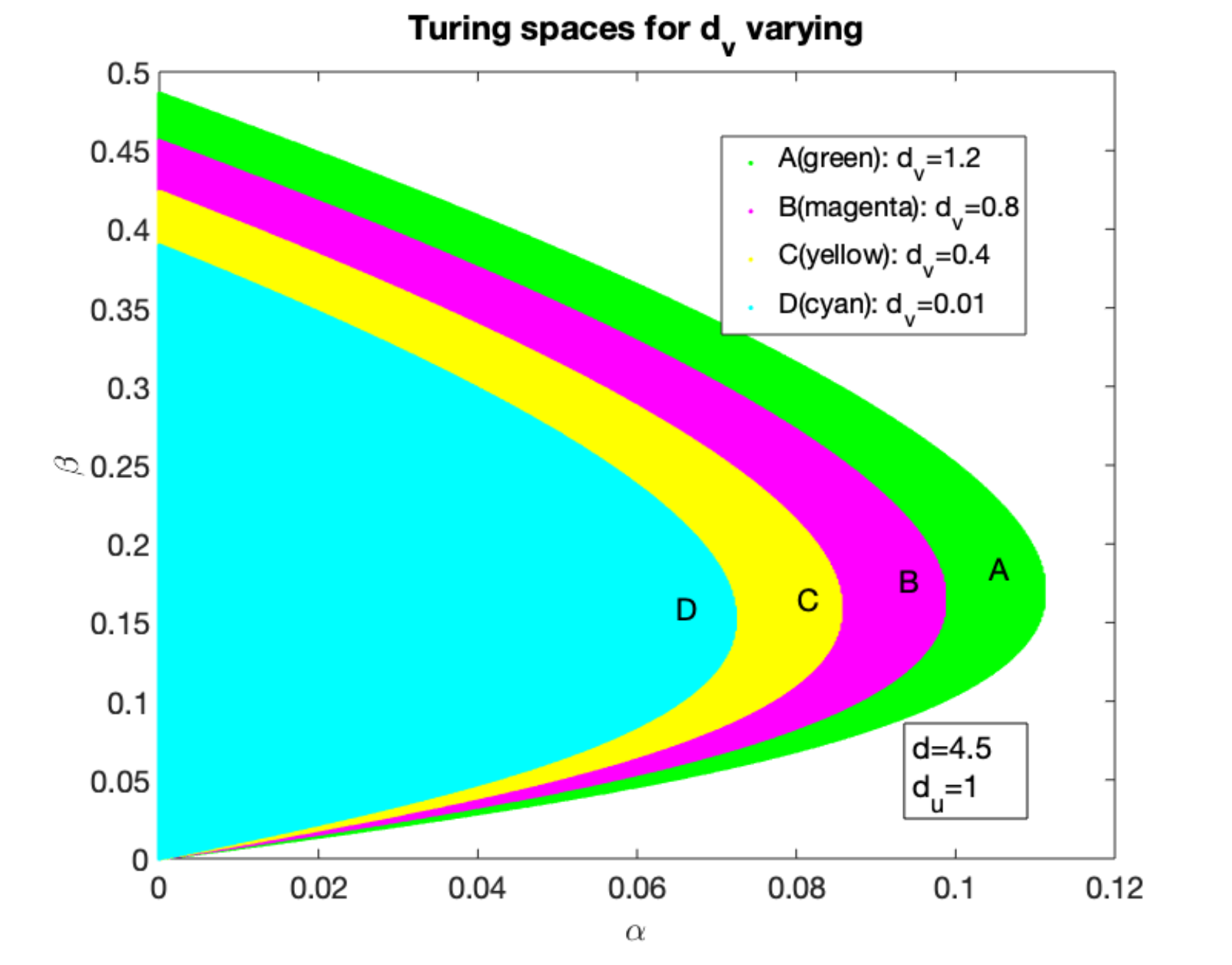} \\
{\small (f) Varying $d_v$ for $d_u=1$}
\end{tabular}
\caption{(a)-(f) Turing regions with domain-size $\rho$  restricted to satisfy conditions of Theorems \ref{theo1} and \ref{Maincond} for various system parameters. $n=0.1$, $m=1$, $\rho=14.8$, $\gamma=100$.}
\label{TuringTheo1}
\end{figure}

\begin{figure}[H]
\begin{tabular}{cc}
\includegraphics[width=0.45\textwidth]{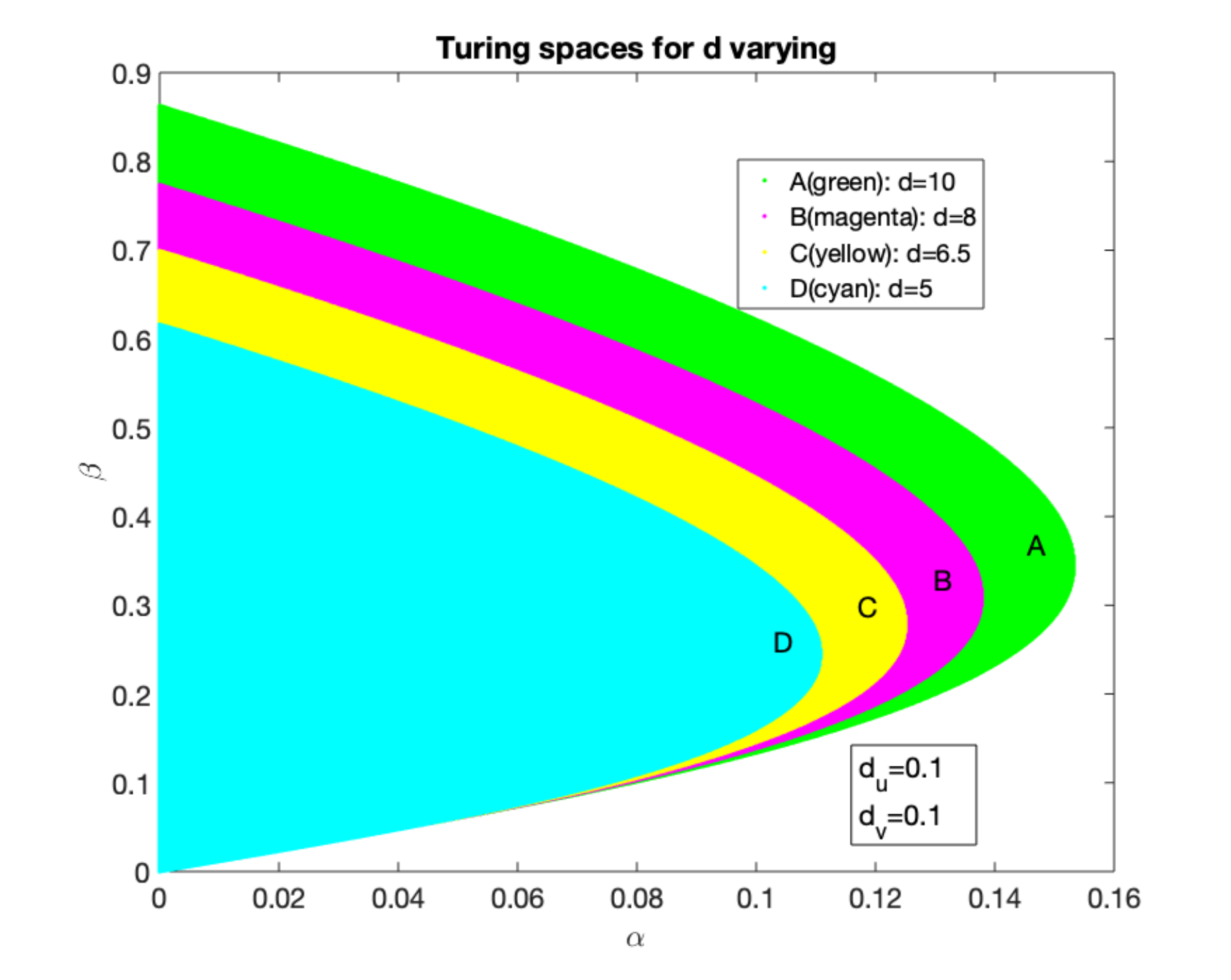} \\
{\small (a) Varying $d$}
\end{tabular}
\begin{tabular}{cc}
\includegraphics[width=0.45\textwidth]{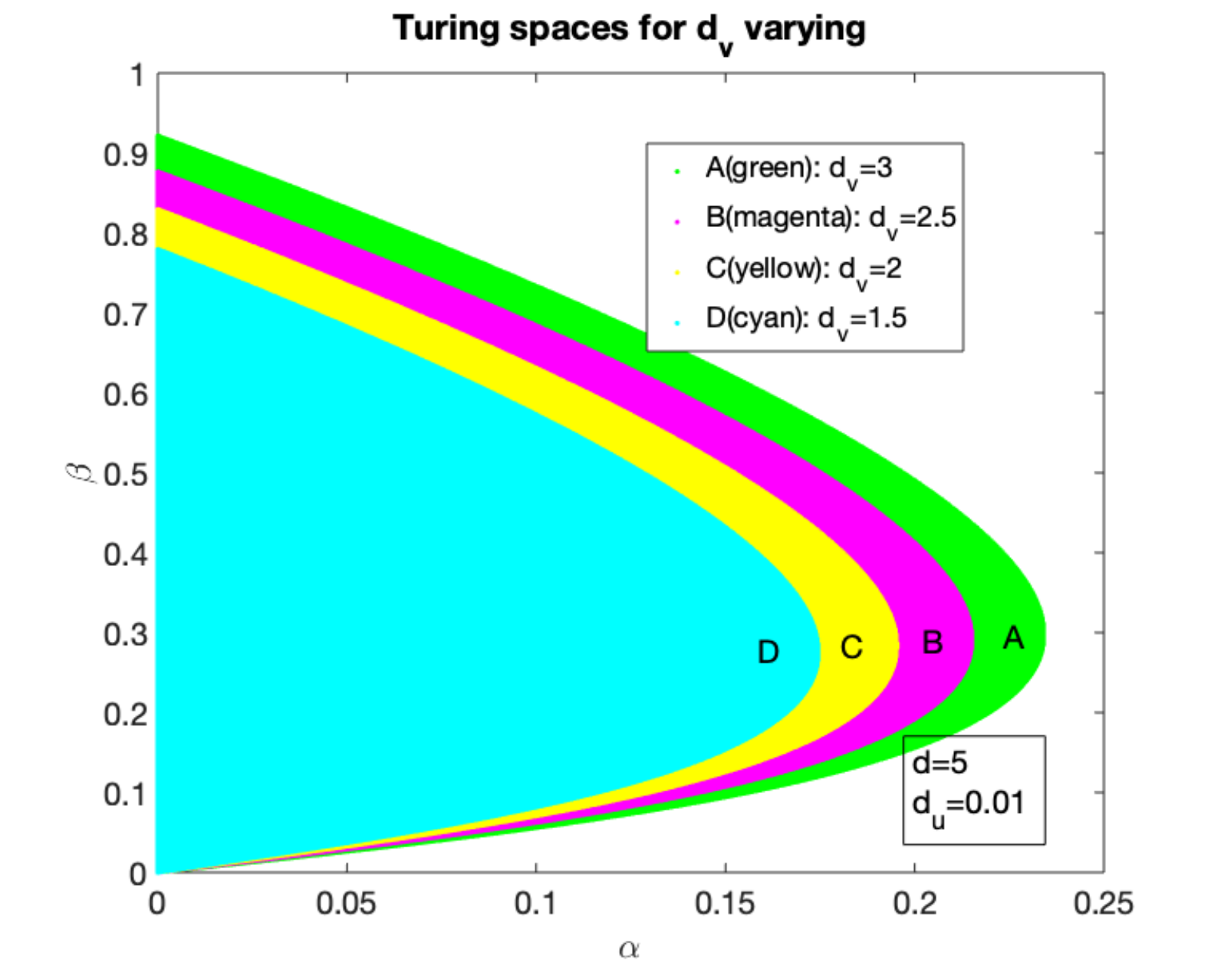} \\
{\small (b) Varying $d_v$}
\end{tabular}
\caption{(a)-(b) Turing regions with domain-size $\rho$  restricted to satisfy conditions of Theorems \ref{theo2} and \ref{Maincond} for various system parameters. $n=0.1$, $m=1$, $\rho=7.3$, $\gamma=100$.}
\label{TuringTheo2}
\end{figure}

\section{Finite element simulations on ring-shape domain}\label{sec:FemSim}

In this section, we employ the finite element method to obtain numerical simulations of the model system \eqref{r1}, aiming to validate the theoretical findings outlined in Section \ref{sec:linearstability}. The main strategy of the finite element method lies in the weak formulation which satisfies the weak variational form \cite{madzvamuse2007velocity,barreira2011surface,madzvamuse2014fully,lakkis2013implicit,tuncer2017projected,frittelli2021bulk,song2022efficient,frittelli2023bulk,frittelli2017lumped}. The weak variational form is then transformed into a semi-discrete weak variational form where the domain has been triangulated into a conforming finite element formulation \cite{evans2022partial,brenner2008mathematical}. This process results in a system of ordinary differential equations, then the desired finite element numerical solution is obtained by using a suitable time-stepping scheme, resulting in a fully-discrete weak formulation.  Details of time-stepping schemes for reaction-diffusion systems can be found in \cite{madzvamuse2006time1, ruuth1995implicit}.

In all our simulations presented in the preceding section, model parameters are given in the figures and their captions. The reaction-kinetic parameters $\alpha$ and $\beta$ are selected from the parameter spaces generated, while the self- and cross-diffusion coefficients are chosen to ensure that the necessary and sufficient conditions for the generation of the parameter spaces are fulfilled. Furthermore, we choose the scaling parameter appropriately.  It is essential to note that the shape of the patterns and convergence rates remained unchanged when we used more and more refined time and space steps. Since we are working on a non-convex domain, we use an open source package for the domain generation called $Gmsh$ \cite{geuzaine2009gmsh}. We use 7936 elements and 8292 degrees of freedom in all our simulations on the discretized domain. 

In all our simulations, initial conditions are taken as small random perturbations around the uniform steady state \cite{madzvamuse2015stability,sarfaraz2017classification, sarfaraz2018domain,sarfaraz2020stability}.  We have used two different forms for the initial conditions to demonstrate the applicability of the numerical solver. First, initial conditions were taken to be small random perturbations near the uniform steady-state given by
\begin{eqnarray*}\label{inits} 
u_0(x,y)= \alpha+\beta+ \epsilon_1 \cos (2\pi(x+y)) + \epsilon_2 \sum\limits_{n=0}^{8} \cos{(n\pi x)} ,\\
v_0(x,y)= \dfrac{\beta}{(\alpha+\beta)^2} + \epsilon_1 \cos(2\pi(x+y))
+ \epsilon_2  \sum\limits_{n=0}^{8} \cos{(n\pi x)},
\end{eqnarray*}
with $\epsilon_1=0.0016$ and $\epsilon_2 = 0.001$. 
Second, we also considered initial conditions prescribed by purely small random perturbations
\begin{eqnarray*}\label{initsnew} 
u_0(x,y)= \alpha+\beta+  \epsilon \, rand , \\
v_0(x,y)= \dfrac{\beta}{(\alpha+\beta)^2} +\epsilon \, rand,
\end{eqnarray*}
where $rand$ is a randomly generated number from the uniform distribution with mean one and variance zero. The first approach allows to validate the numerical solver in the case that parameter values are selected to excite specific modes $m$ and $n$, for example. The second, assumes no specific knowledge of the behaviour of the solution of the reaction-diffusion system close to bifurcation points. For mode selection and the role of initial conditions close to bifurcation points, the interested reader is referred to consult \cite{madzvamuse2015stability,madzvamuse2000numerical}.

We start our numerical experiments to understand the effects of parameter choice in the spatial pattern formation using conditions on the value representing the thickness of the annulus ($\Omega$), which is denoted by $\rho$. Fig. \ref{Turing1} shows the Turing type of pattern formation as analytically predicted when system parameters are chosen in adherence to conditions stated in Theorems \ref{theo1} and \ref{Maincond}.  A timestep value of   $\Delta t =0.0025$ and model parameter values $d=1$, $\gamma=730$, $d_u=0.001$, $d_v=0.46$, $\alpha=0.09$ and $\beta=0.2$ are selected for the finite element simulation of System \eqref{r1}. We observe  the convergence of the numerical solution to a spatially inhomogeneous time-independent solution.   During this evolution to the spatially inhomogeneous steady state, we observe patterns evolving from stripes to spots in the early stages. The shape of the spot type pattern remains unchanged in later stages of its evolution. This is validated by analyzing the $L_2$ norm of the discrete time-derivatives of the numerical solutions $u$ and $v$ as shown in Fig.  \ref{Turing1} (h). 

Fig. \ref{TuringNegativeCD} presents the Turing type of pattern formation in light of conditions stated in Theorems \ref{theo1} and \ref{Maincond} using the parameters $\Delta t =0.0025$, $d=1$, $\gamma=720$, $d_u=-0.1$, $d_v=0.5$, $\alpha=0.09$ and $\beta=0.2$.  Here, we want to explore the impact of negative cross-diffusion on pattern formation. We observe the formation of a spatially inhomogeneous steady state which is time-independent in the later stages of its development. The important observation in Fig. \ref{TuringNegativeCD} is that the choice of negative cross-diffusion can give rise to Turing type of pattern formation. We observe that, once the patterns evolve from stripes to spots in the early stages, these remain unchanged as time progresses. This can be understood by considering the $L_2$ norm of the discrete time derivatives of the solutions $u$ and $v$ as shown in Fig.  \ref{TuringNegativeCD} (h). It is essential to note that,  in the absence of cross-diffusion Turing type pattern can not be obtained by the choice of self-diffusion coefficient $d=1$  alone, for this two-component reaction-diffusion system.  In the standard theory of diffusion-driven instability, one of the criticism of Turing's theory is that it requires diffusion coefficients to be of significantly different magnitudes. In particular, it requires that the inhibitor  diffuses a lot faster than the activator, resulting in what is known as the {\it long-range inhibition, short-range activation}. In our case, by adding cross-diffusion to the system, a two-component reaction-diffusion system with equal self-diffusion has the capacity to give rise to pattern formation as illustrated in the numerical simulations given by Fig. \ref{Turing1} and Fig. \ref{TuringNegativeCD}. We note that the main difference between the simulations of Fig. \ref{Turing1} and Fig. \ref{TuringNegativeCD} is that the choice of cross-diffusion coefficient is positive and negative. Both Figs. \ref{Turing1} and \ref{TuringNegativeCD} show the formation of spatially inhomogeneous steady state solutions which are time-independent at later stages.

Next, we want to demonstrate that if Theorems \ref{theo2} and \ref{Maincond} are simultaneously fulfilled, that is there exists model parameters and $m$ and $n$ such that the domain-size $\rho$ is bounded both from above and below, then hopf and/or transcritical bifurcations are completely forbidden. Only Turing patterns are allowed. In Fig. \ref{Turing2} we provide a numerical example illustrating this behaviour where we show the emergence of a Turing type pattern when model parameters are selected to fulfill conditions stated in Theorems \ref{theo2} and \ref{Maincond}. Model parameter values are selected as $\Delta t =0.0025$, $d=12$, $\gamma=270$, $d_u=1$, $d_v=1.7$, $\alpha=0.07$ and $\beta=0.45$ according to the conditions in Theorems \ref{theo2} and \ref{Maincond}. Comparing the results of Figs. \ref{Turing1} and \ref{TuringNegativeCD} with Fig. \ref{Turing2}, we observe that spot-type patterns are less abundant on the annular region. The evolution profile of the temporal stability in the dynamics is well-observed by visualising the $L_2$ norms of the discrete time-derivative of the numerical solutions of components $u$ and $v$.

As presented in the parameter regions given by Fig. \ref{HopfTransParameter}, when eigenvalues are a complex-conjugate pair with positive real parts, we expect the system to exhibit periodicity in the spatiotemporal pattern formation process.  This behavior is presented in Fig. \ref{Limitcycle1}, as a series of snapshots showing how the shape of the spot bifurcates, periodically,  in time. The alternating pattern behavior of radii of the spots, transitioning from smaller to larger, is well-captured in Fig. \ref{Limitcycle1}.  It is noted that the system parameters are selected to satisfy conditions stated in Theorems \ref{theo1} and \ref{Maincond} on $\rho$ with $\Delta t =0.0025$, $d=2.6$, $\gamma=375$, $d_u=1.6$, $d_v=0.5$, $\alpha=0.09$ and $\beta=0.1$. Spatiotemporal periodicity is shown by depicting the $L_2$ norm of the discrete time-derivative of the components $u$ and $v$. We observe the same temporal gap between the peaks of the discrete time-derivative of the solution as shown in the $L_2$ norm of the time derivatives given in Fig.  \ref{Limitcycle1} (k).  

Our next distinctive finite element simulation given in Fig. \ref{Limitcycle2} provides spatiotemporal {\it periodic} pattern formation with negative cross-diffusion as well as the special case of self-diffusion coefficient $d=1$ which demonstrates the effect of the cross-diffusion.  Fig. \ref{Limitcycle2} presents the finite element simulations when eigenvalues are a complex-conjugate pair with a positive real part as shown in Fig. \ref{HopfTransParameter}.  It is noted that the system parameters are selected to satisfy conditions of Theorems \ref{theo1} and \ref{Maincond} on $\rho$ with $\Delta t =0.0025$, $d=1$, $\gamma=250$, $d_u=-0.9$, $d_v=0.55$, $\alpha=0.085$ and $\beta=0.1$. The alternating pattern behavior of standing waves of spots type that transition from smaller to larger, is well-captured in Figs. \ref{Limitcycle1} and \ref{Limitcycle2}. 
Spatiotemporal periodicity is shown by the plot of the $L_2$ norm of the discrete time-derivative of the components $u$ and $v$.
We observe the same temporal gap with different amplitudes from the $L_2$ norm of the time derivatives given   in Fig.  \ref{Limitcycle2} (k).  In  Fig. \ref{Limitcycle3}, we provide the 3D views of the simulations presented in Fig. \ref{Limitcycle2}. Now,  with the 3D views of the simulations, we can capture how the system dynamics evolve both in space and time, resulting in what are known as "standing waves". Lastly, Fig. \ref{NegativeCD} illustrates the relationship between the plot of the $L_2$ norm for the discrete time-derivative and the selected simulations,  exhibiting the periodicity of the evolution of the spatiotemporal pattern. Note that Figs. \ref{Limitcycle2},  \ref{Limitcycle3} and \ref{NegativeCD} represent the same finite element simulation with the same parameter values but presented in different ways for better visualisation and interpretation.

\begin{figure}[H]
\begin{tabular}{cc}
\includegraphics[width=0.28\linewidth]{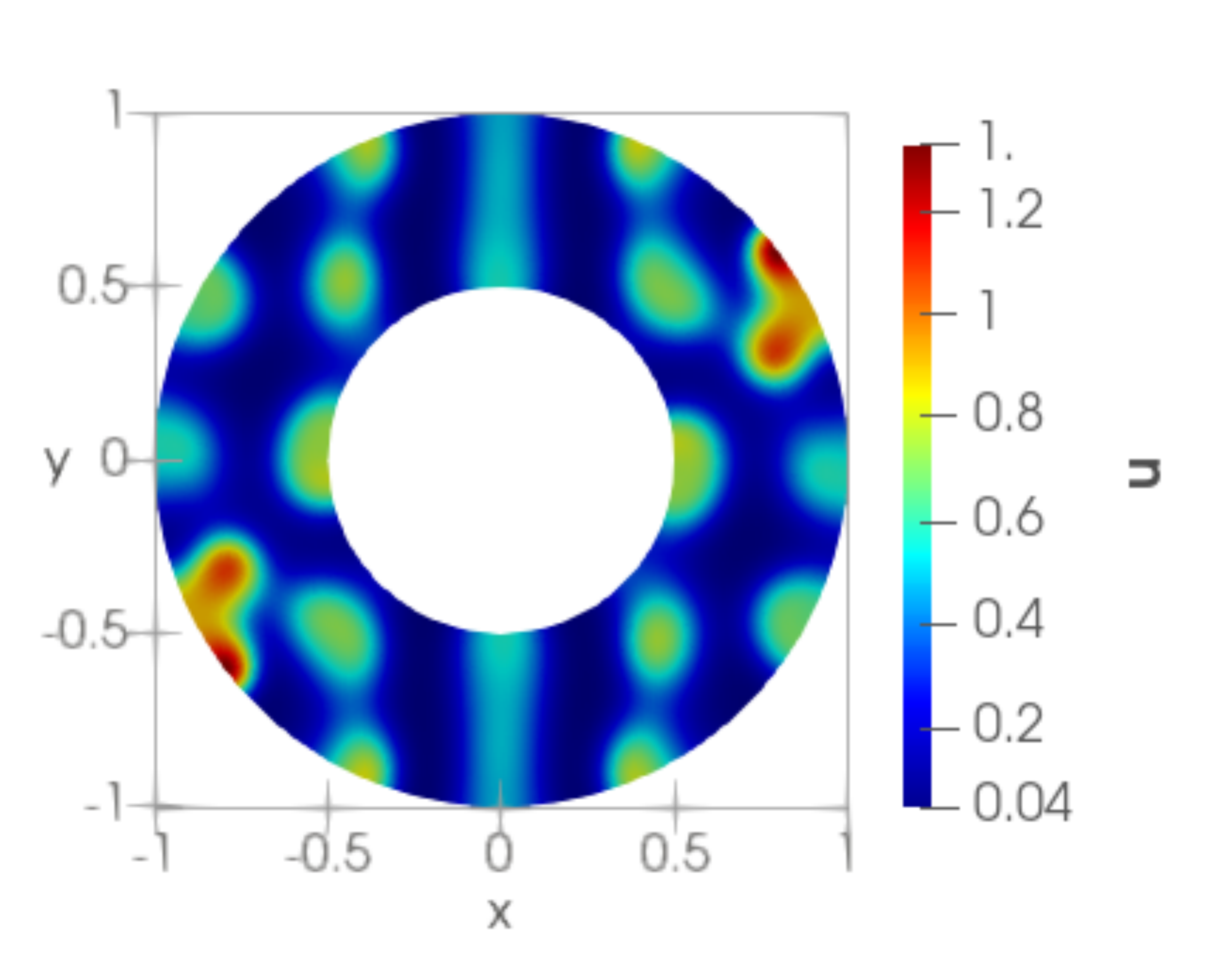} \\
{\small (a) t=0.04}
\end{tabular}
\begin{tabular}{cc}
\includegraphics[width=0.28\linewidth]{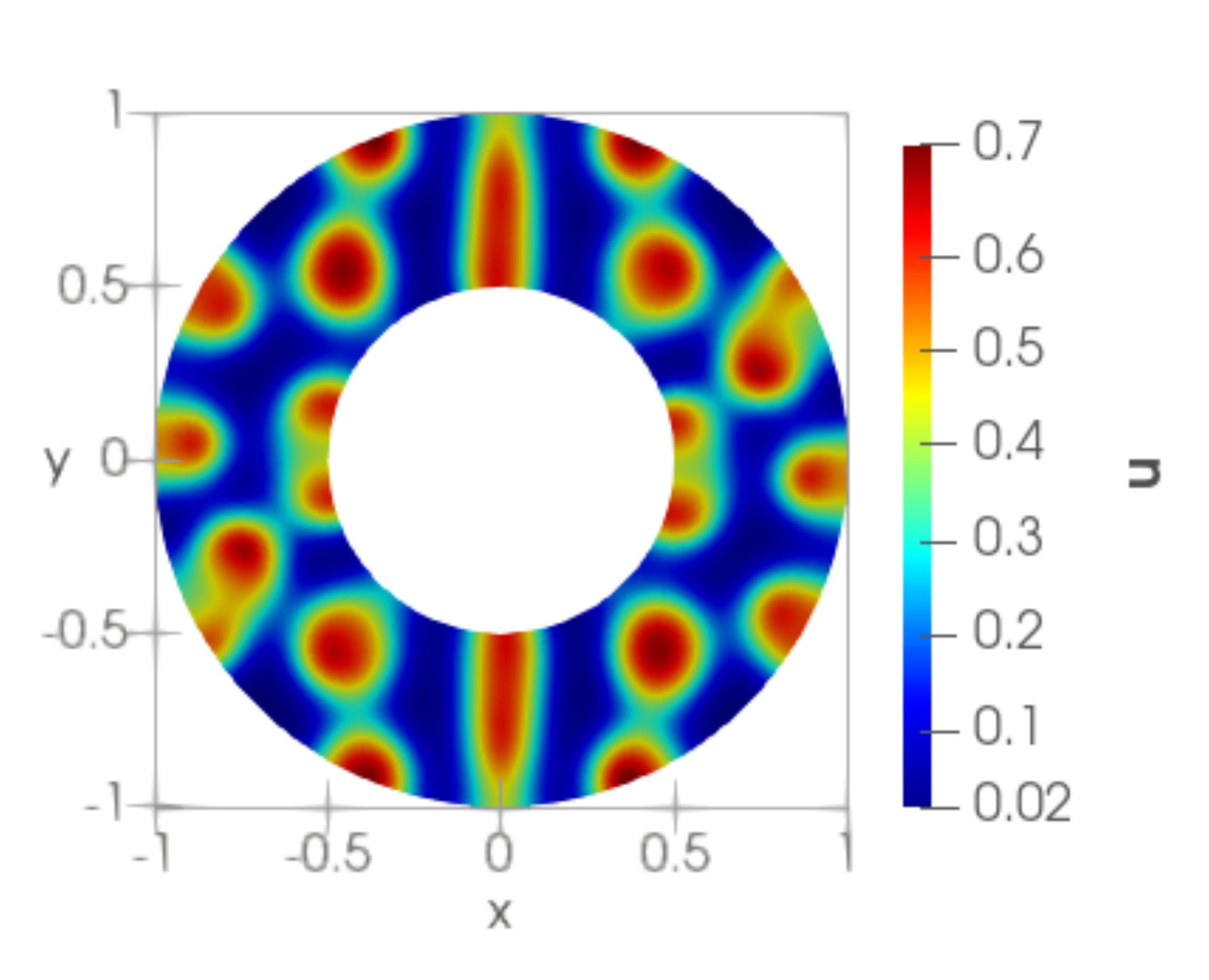} \\
{\small (b) t=0.075}
\end{tabular}
\begin{tabular}{cc}
\includegraphics[width=0.3\linewidth]{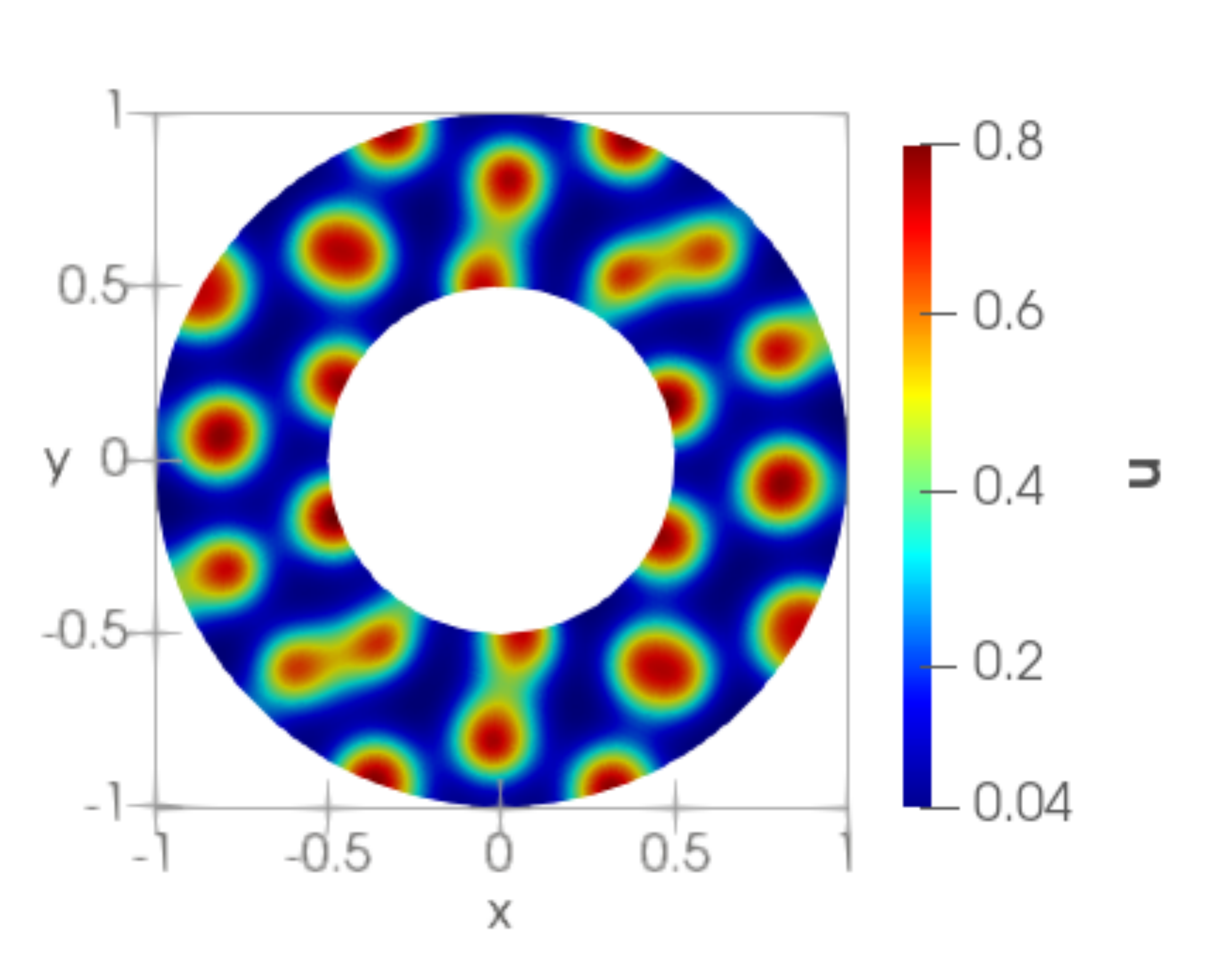} \\
{\small (c) t=0.25}
\end{tabular}
\begin{tabular}{cc}
\includegraphics[width=0.3\linewidth]{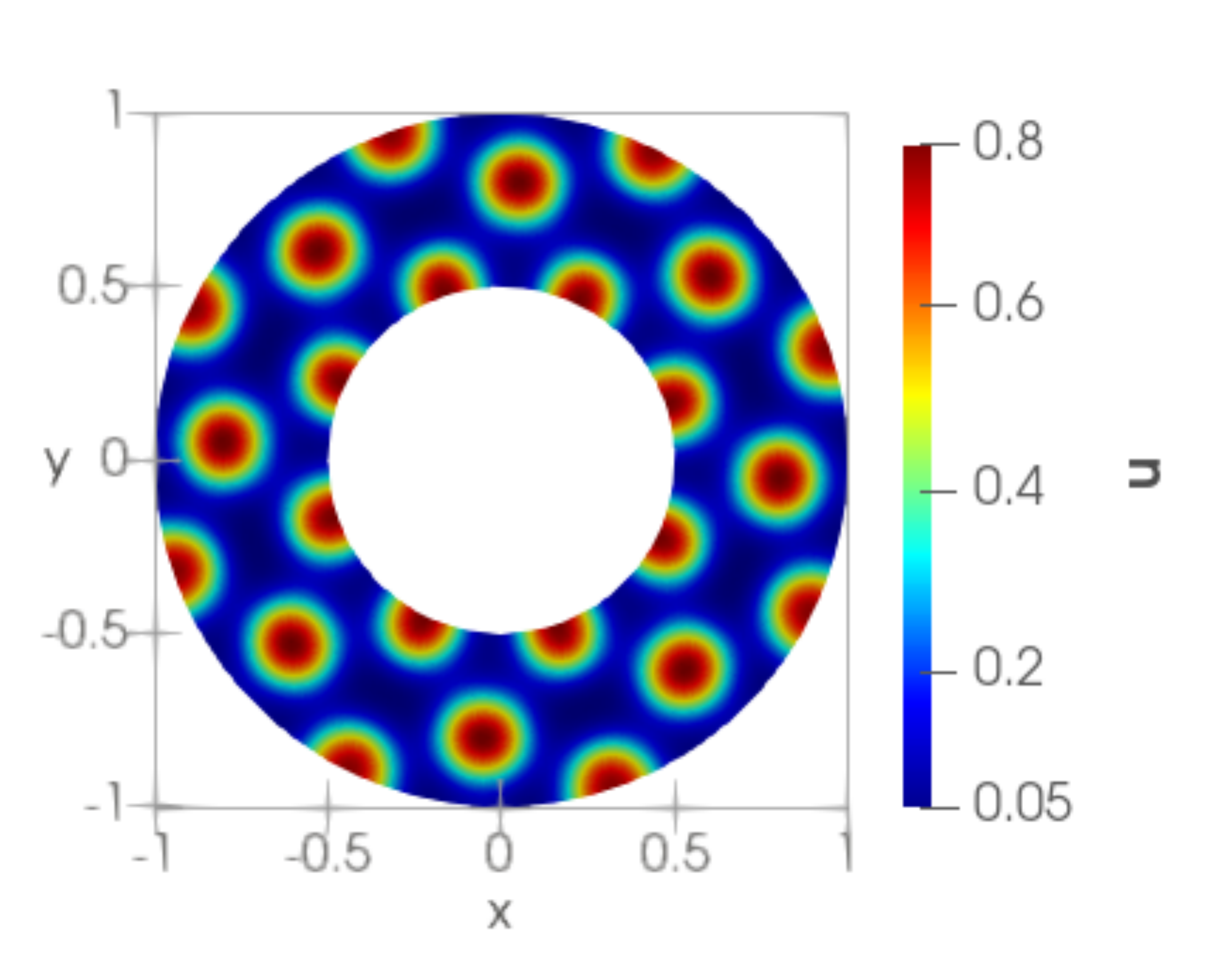} \\
{\small (d) t=0.8}
\end{tabular}
\begin{tabular}{cc}
\includegraphics[width=0.3\linewidth]{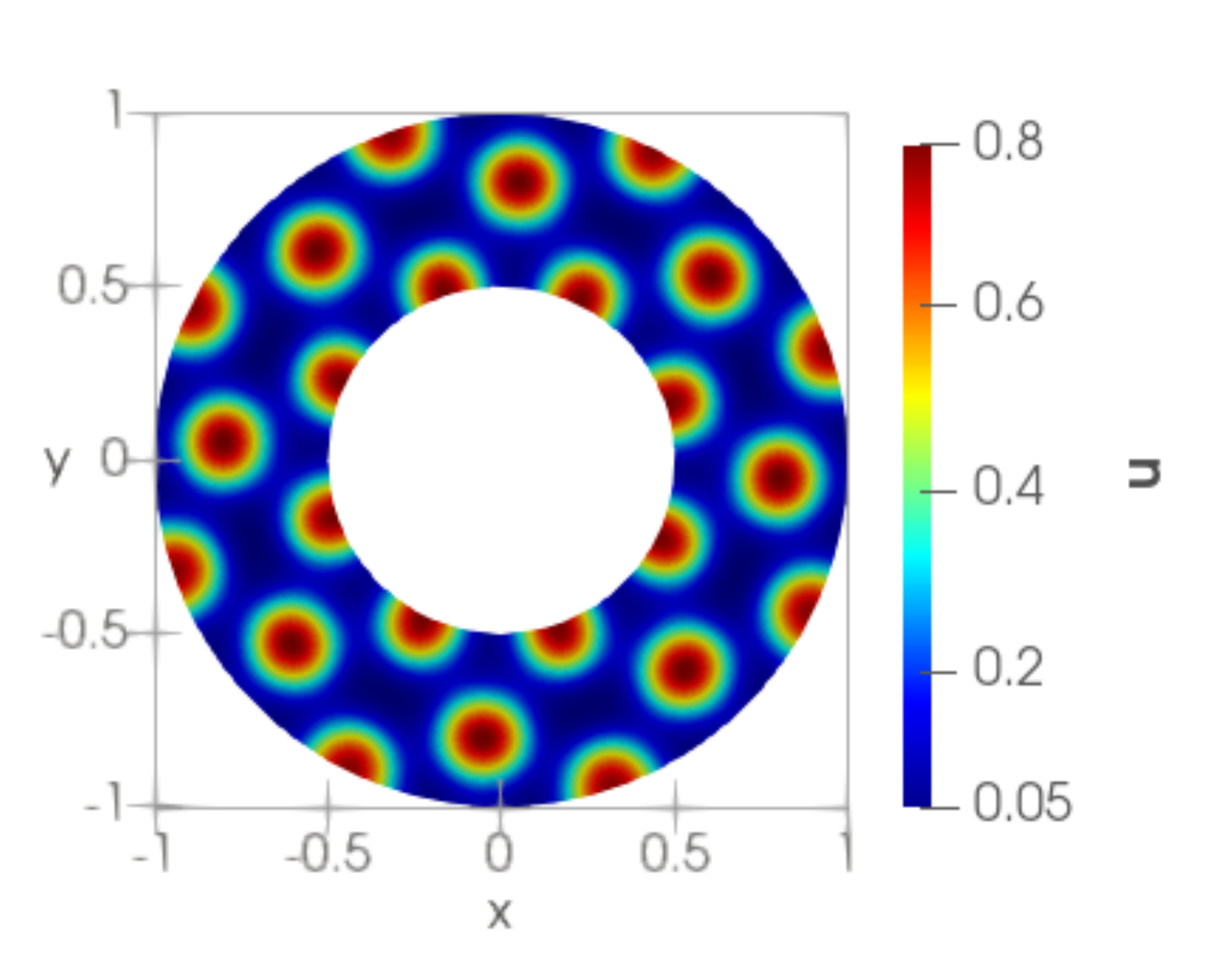} \\
{\small (e) t=1}
\end{tabular}
\begin{tabular}{cc}
\includegraphics[width=0.3\linewidth]{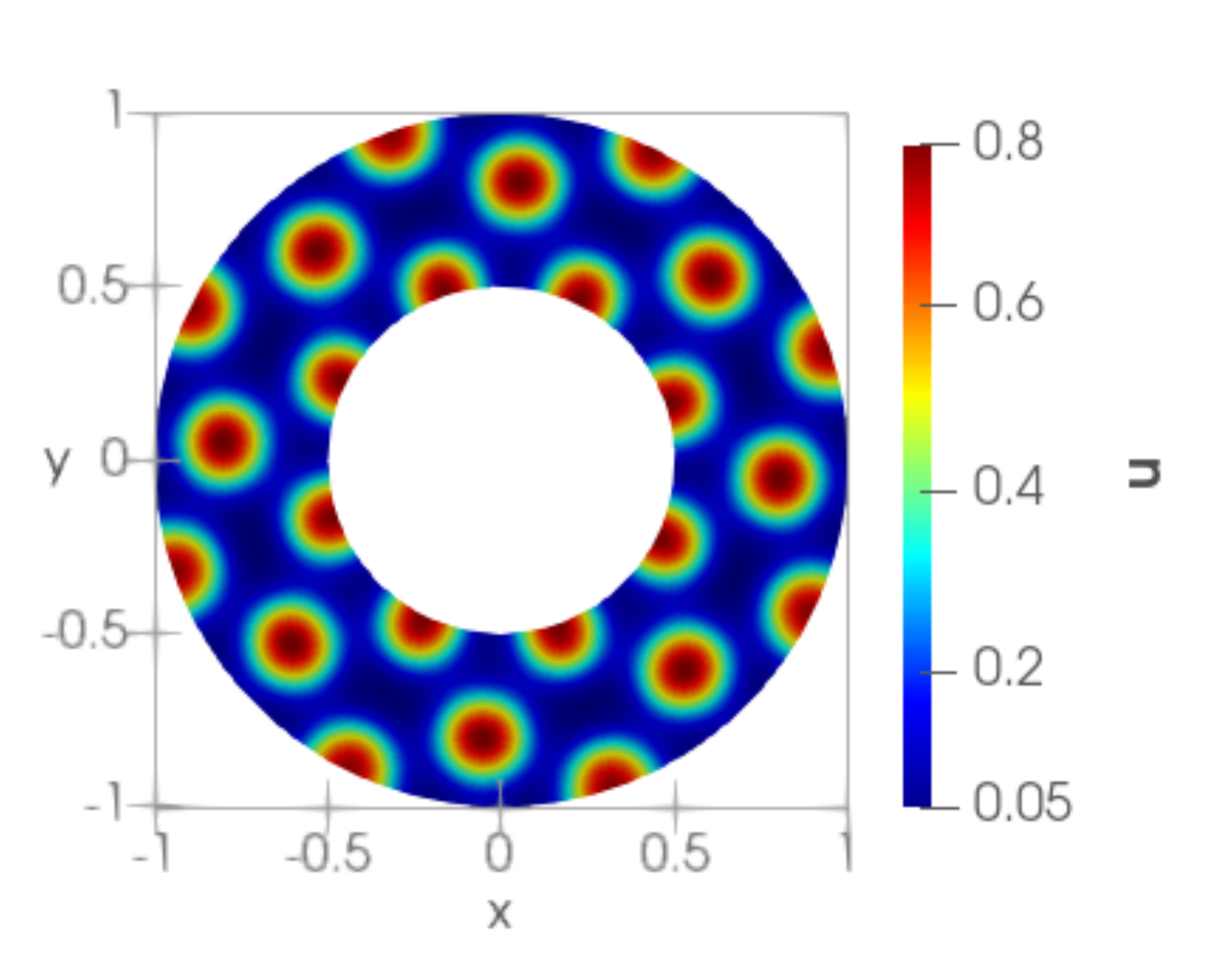} \\
{\small (f) t=1.5}
\end{tabular}
\begin{tabular}{cc}
\includegraphics[width=0.3\linewidth]{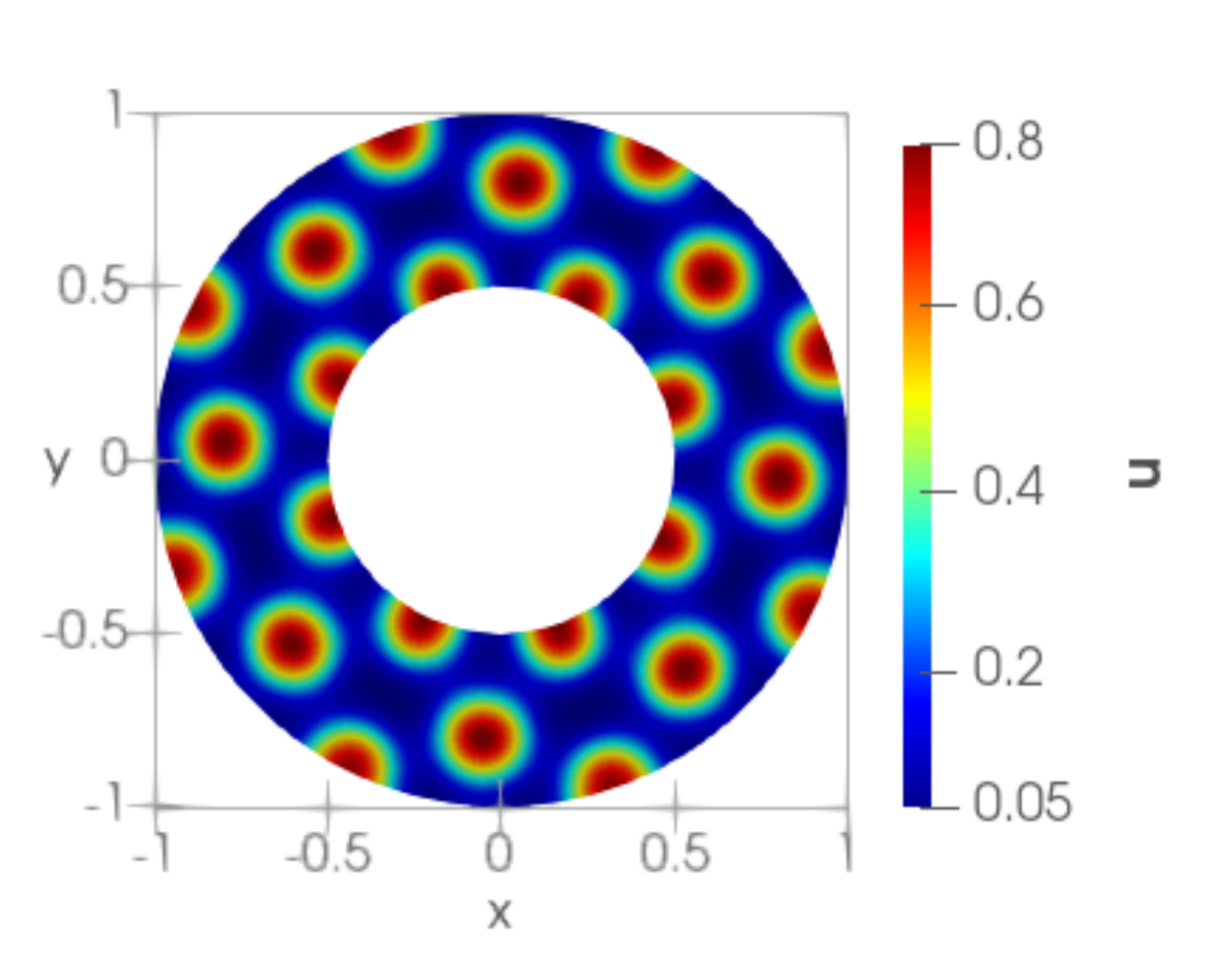} \\
{\small (g) t=2}
\end{tabular}
\begin{tabular}{cc}
\includegraphics[width=0.64\textwidth]{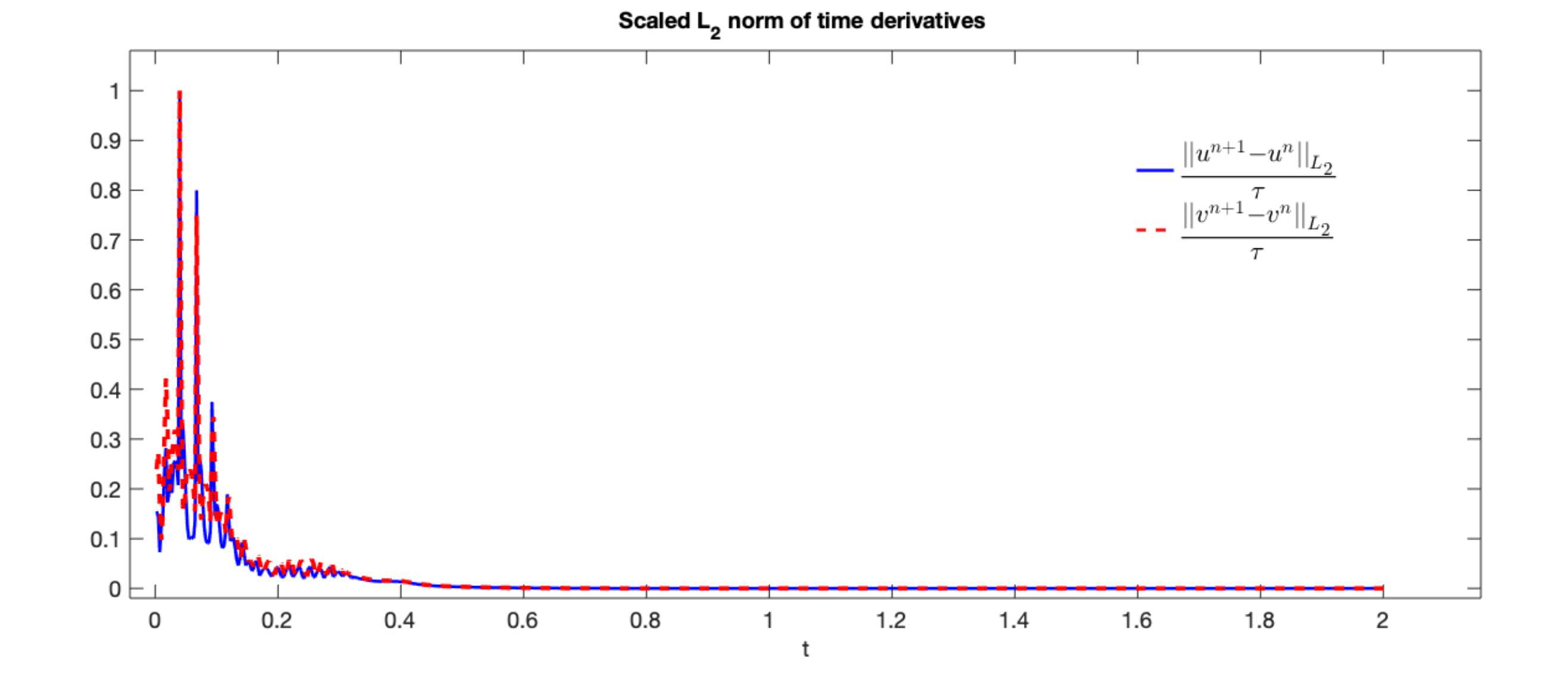} \\
{\small (h) Discrete time derivative of the solutions $u$ and $v$}
\end{tabular}
\caption{(a)-(g) Finite element simulations corresponding to the $u$-component of the cross-diffusive reaction-diffusion system on annulus showing the transient process to a spatially inhomogeneous, and time-independent pattern. Parameters are selected to satisfy conditions of Theorems \ref{theo1} and \ref{Maincond} with $\Delta t =0.0025$, $d=1$, $\gamma=730$, $d_u=0.001$, $d_v=0.46$, $\alpha=0.09$ and $\beta=0.2$, as shown in Fig. \ref{TuringTheo1} (h) Plot of the $L_2$ norms showing the convergence of the discrete time-derivative of the numerical solutions $u$ and $v$.}
\label{Turing1}
\end{figure}

\begin{figure}[H]
\begin{tabular}{cc}
\includegraphics[width=0.28\linewidth]{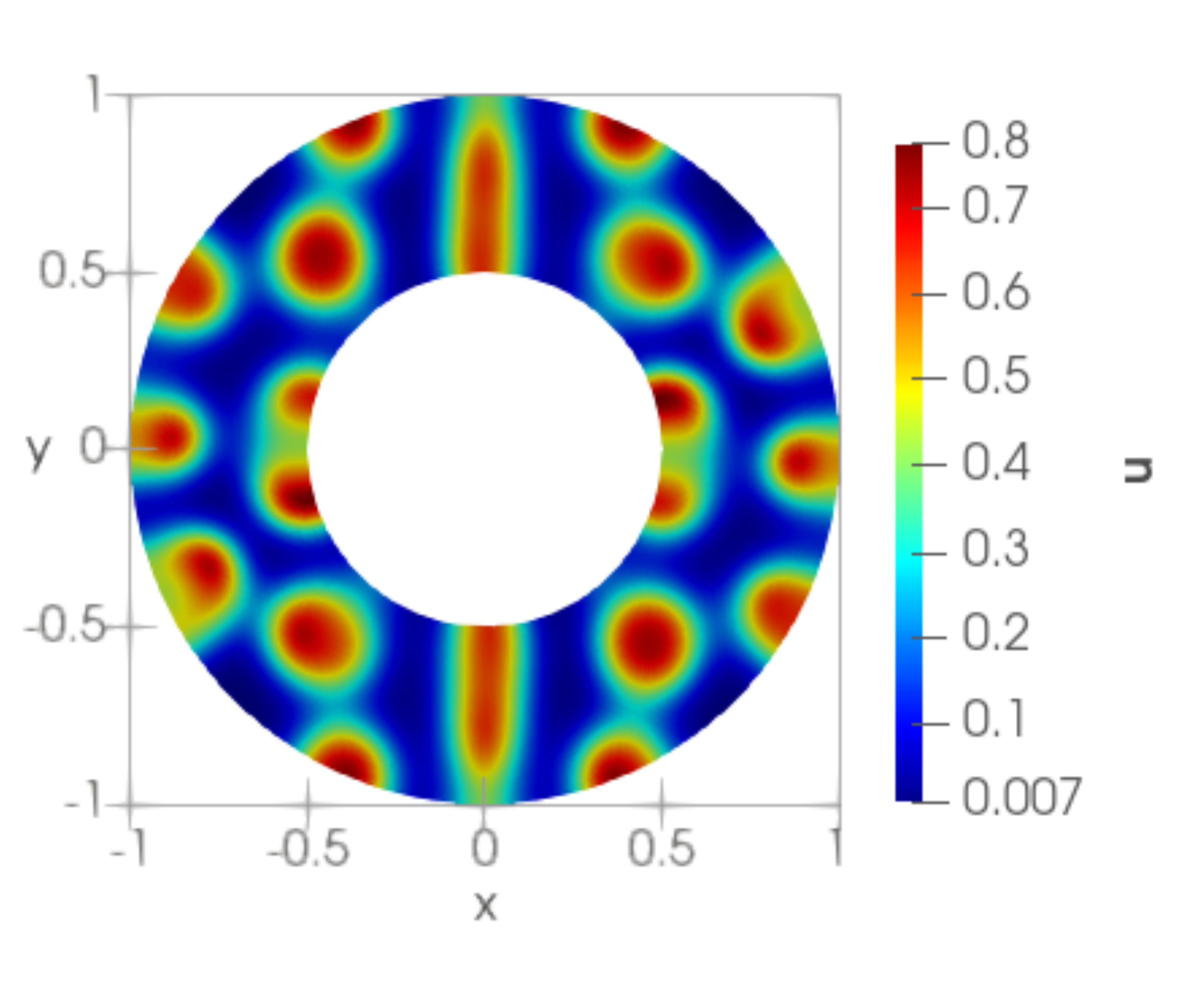} \\
{\small (a) t=0.075}
\end{tabular}
\begin{tabular}{cc}
\includegraphics[width=0.28\linewidth]{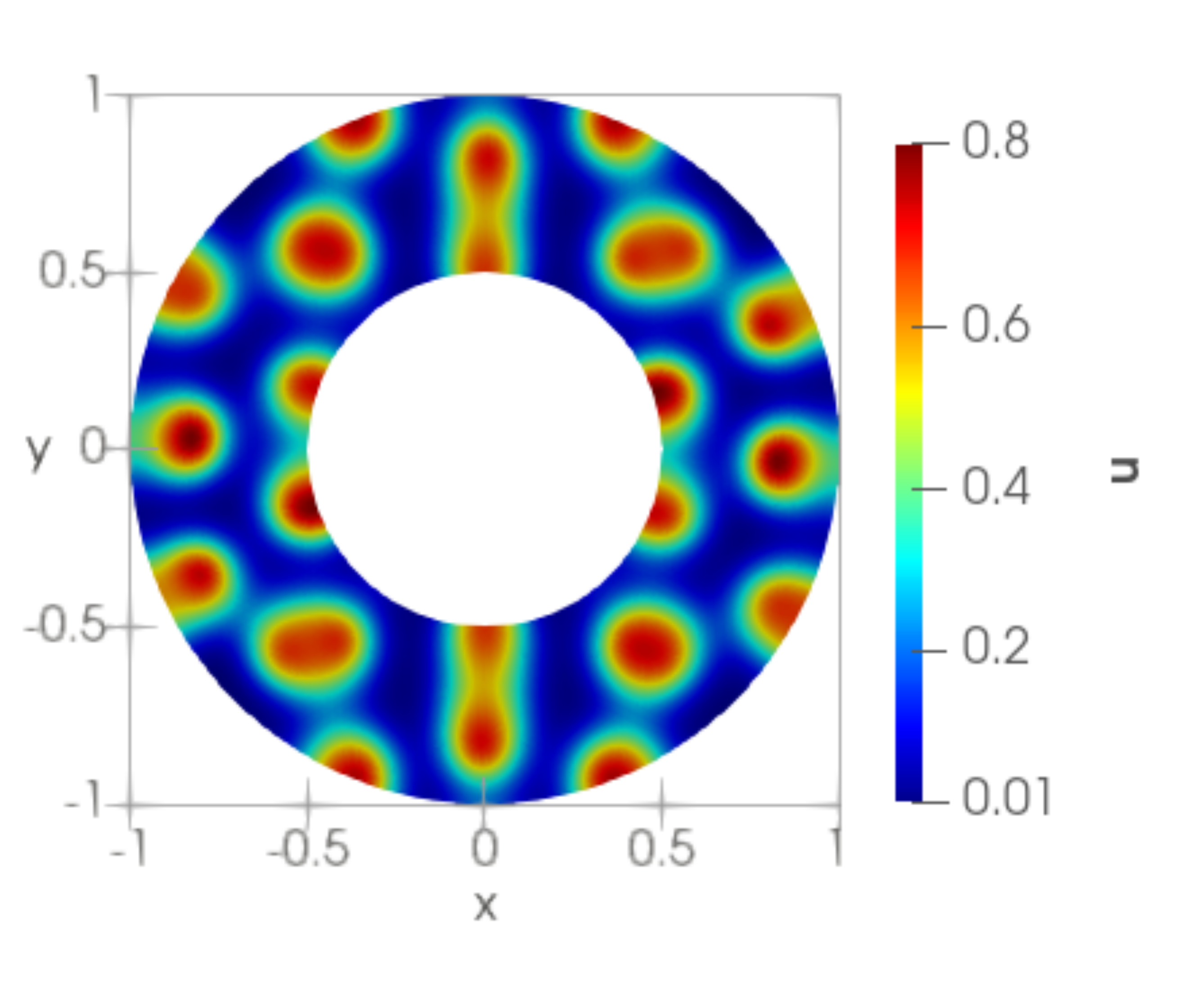} \\
{\small (b) t=0.1}
\end{tabular}
\begin{tabular}{cc}
\includegraphics[width=0.3\linewidth]{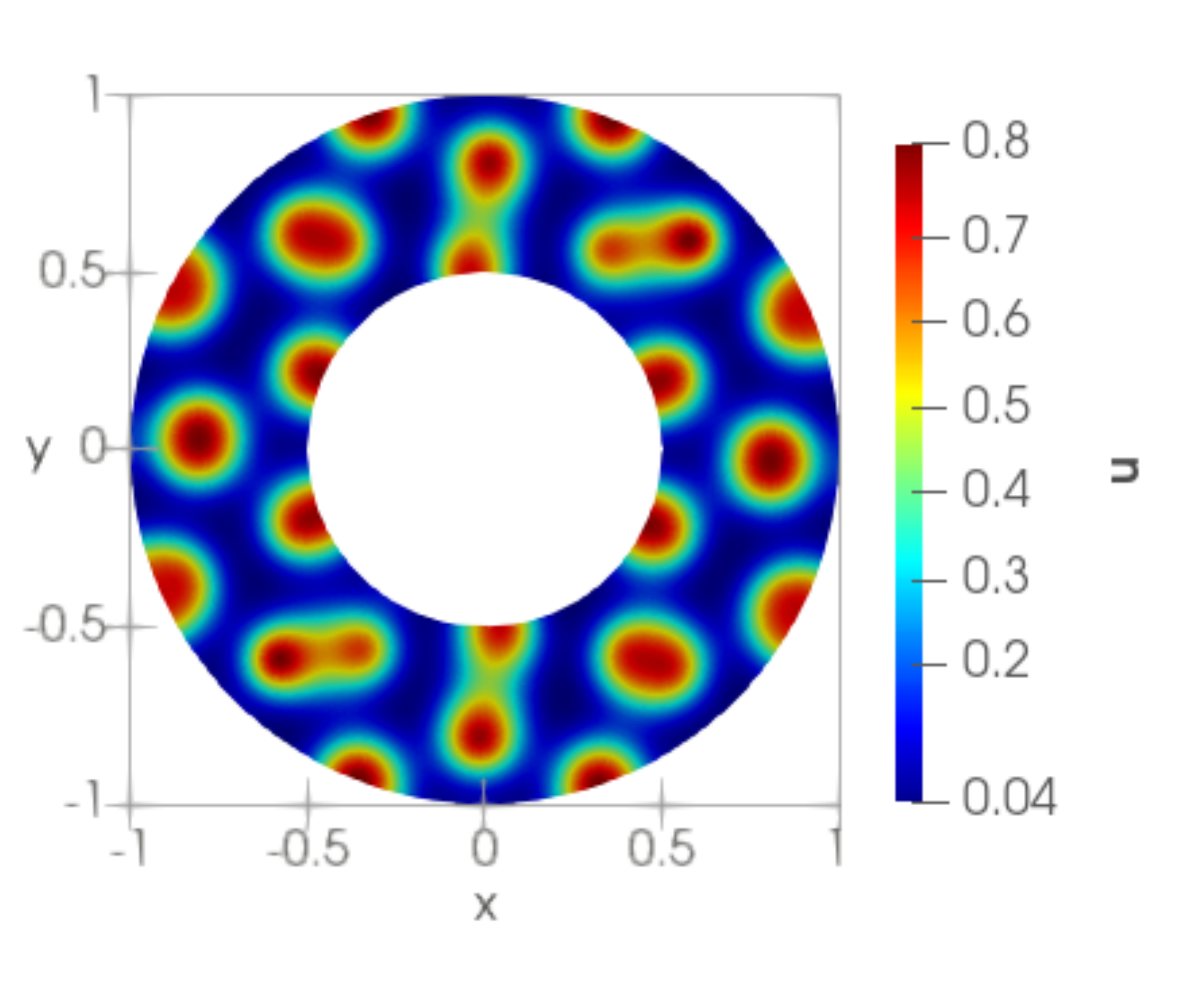} \\
{\small (c) t=0.25}
\end{tabular}
\begin{tabular}{cc}
\includegraphics[width=0.3\linewidth]{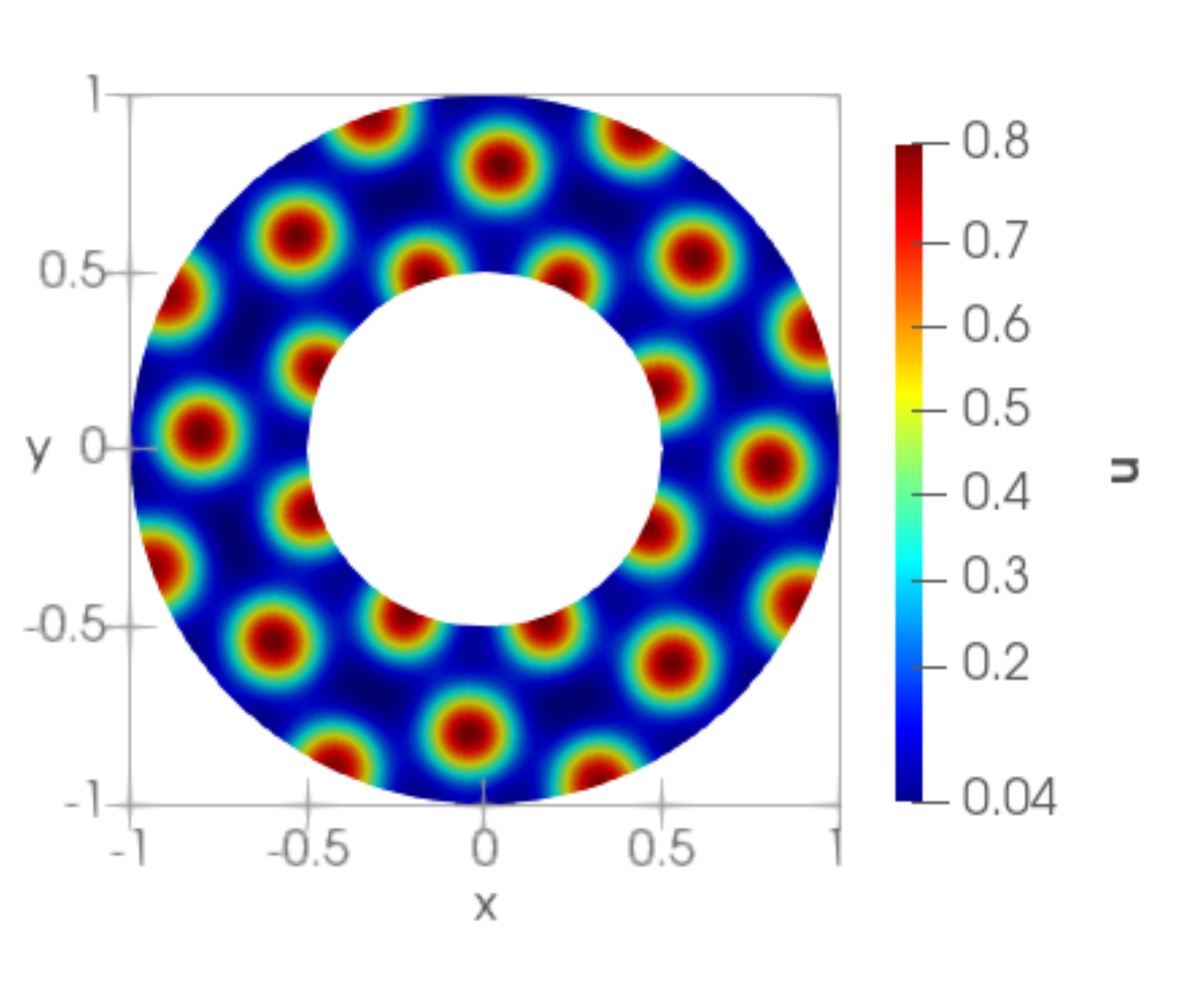} \\
{\small (d) t=0.5}
\end{tabular}
\begin{tabular}{cc}
\includegraphics[width=0.3\linewidth]{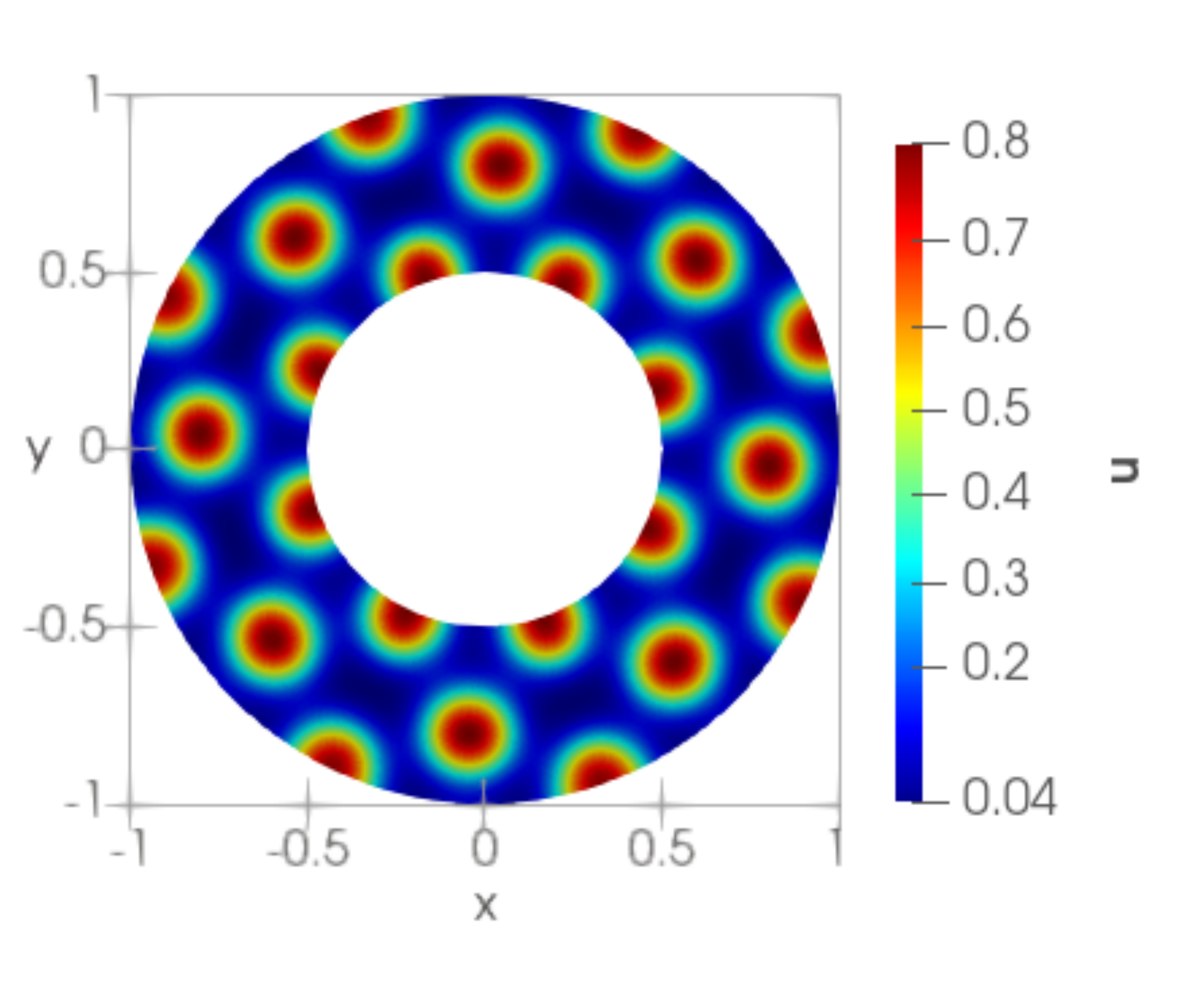} \\
{\small (e) t=1}
\end{tabular}
\begin{tabular}{cc}
\includegraphics[width=0.3\linewidth]{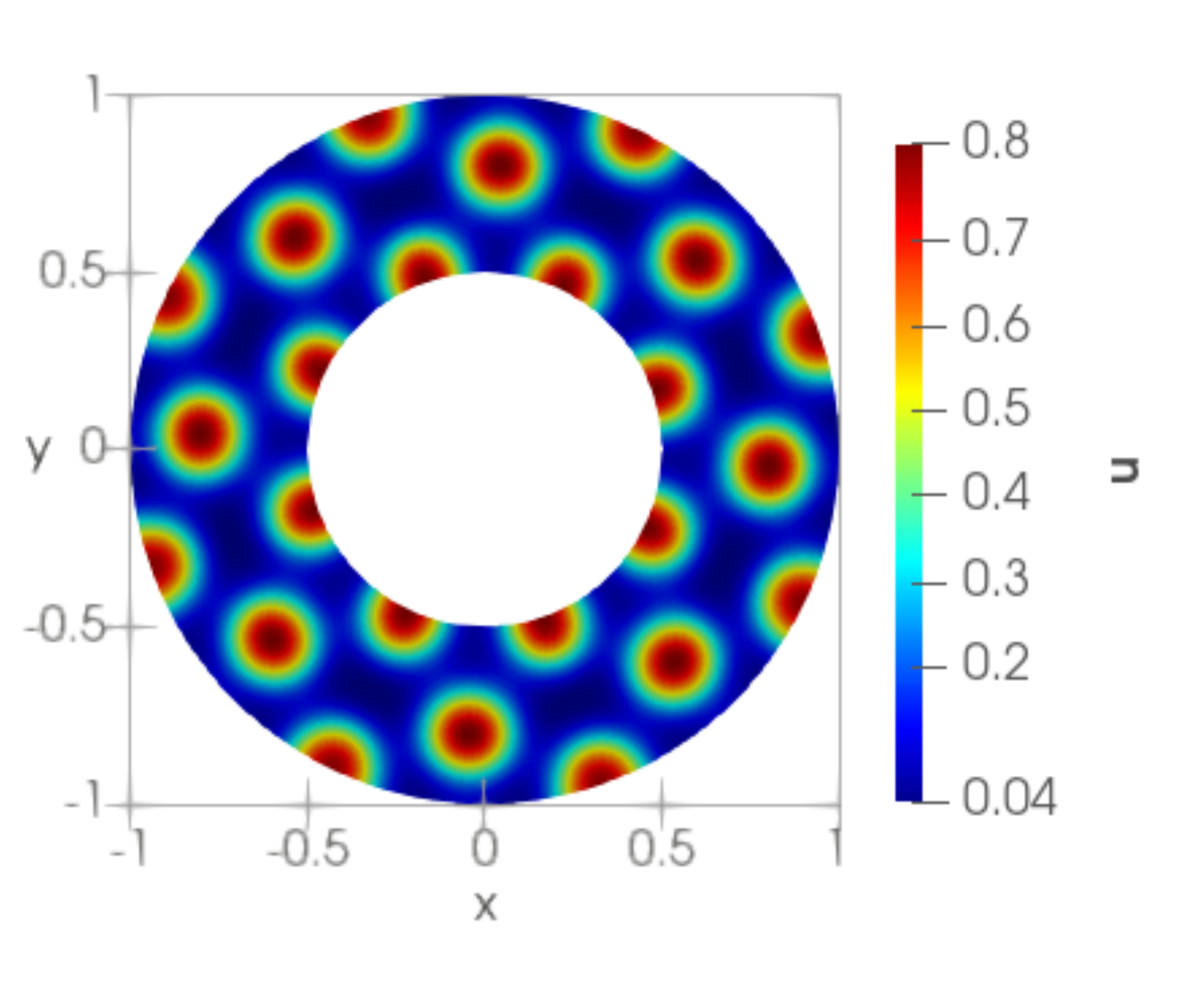} \\
{\small (f) t=1.5}
\end{tabular}
\begin{tabular}{cc}
\includegraphics[width=0.3\linewidth]{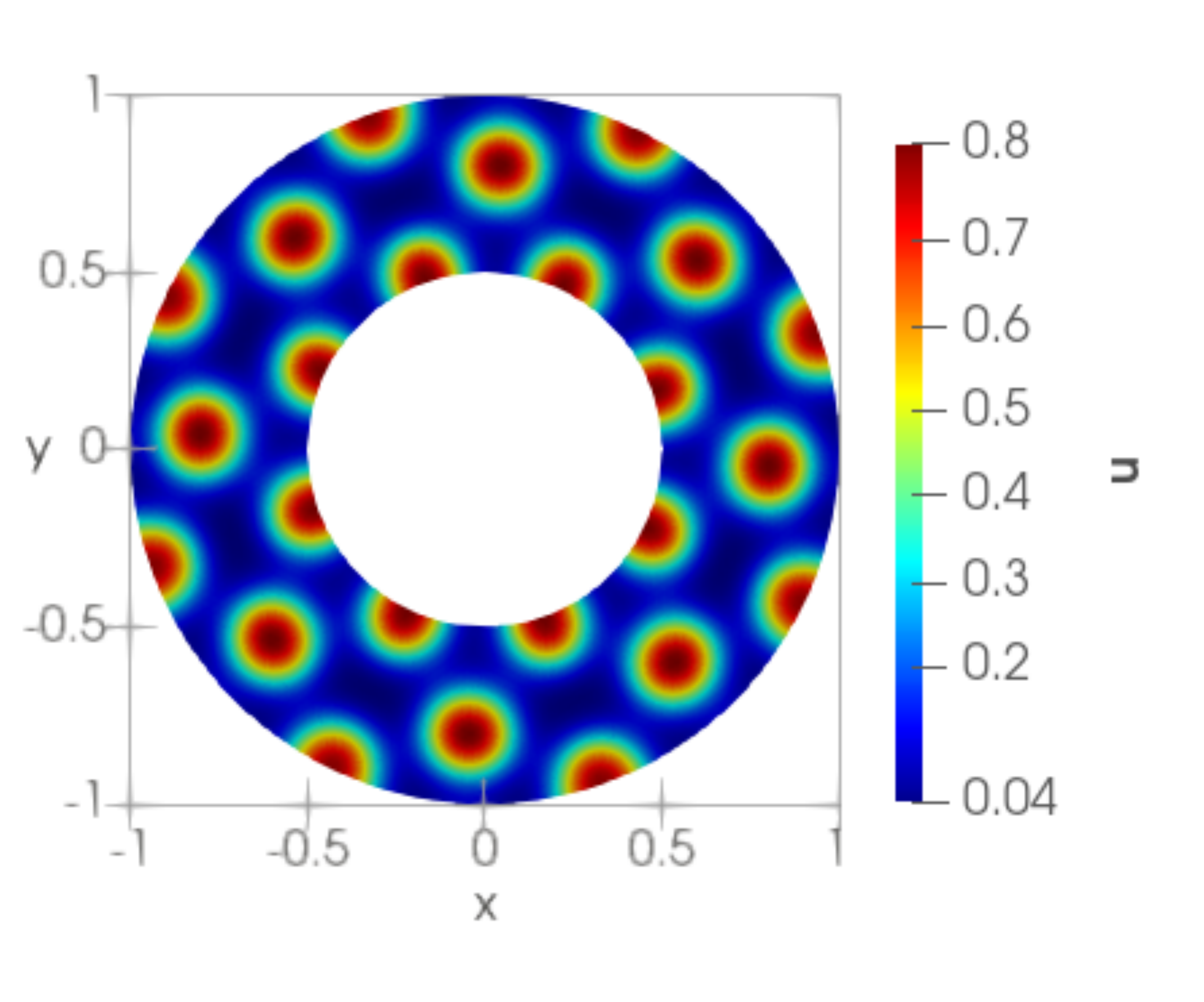} \\
{\small (g) t=2}
\end{tabular}
\begin{tabular}{cc}
\includegraphics[width=0.64\textwidth]{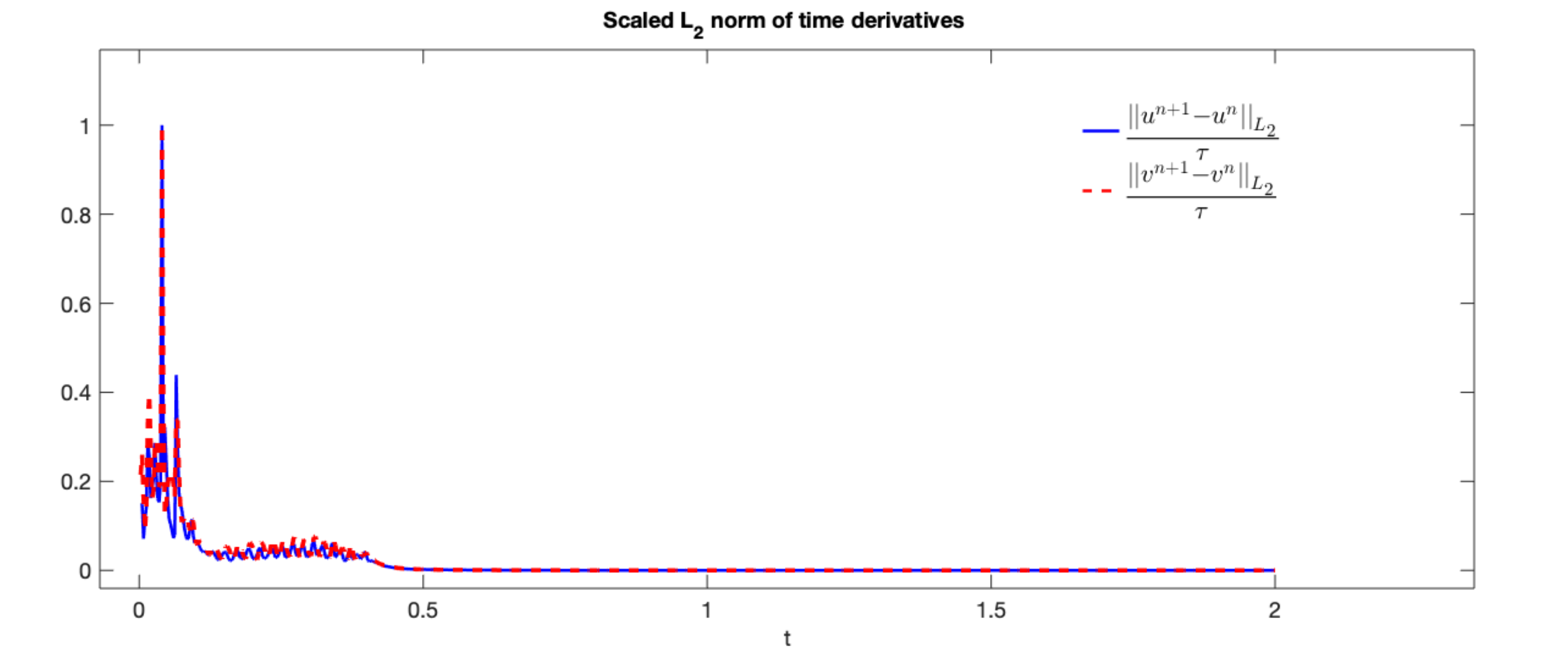} \\
{\small (h) Discrete time derivative of the solutions $u$ and $v$}
\end{tabular}
\caption{(a)-(g) Finite element simulations corresponding to the $u$-component of the cross-diffusive reaction-diffusion system on annulus domain exhibiting the transient process to a spatially inhomogeneous, and time-independent pattern. Parameters are selected to satisfy conditions of Theorems \ref{theo1} and \ref{Maincond} with $\Delta t =0.0025$, $d=1$, $\gamma=720$, $d_u=-0.1$, $d_v=0.5$, $\alpha=0.09$ and $\beta=0.2$, as shown in Fig. \ref{TuringTheo1} (h) Plot of the $L_2$ norms showing the convergence of the  discrete time-derivative of the numerical solutions $u$ and $v$.}
\label{TuringNegativeCD}
\end{figure}

\begin{figure}[H]
\begin{tabular}{cc}
\includegraphics[width=0.28\linewidth]{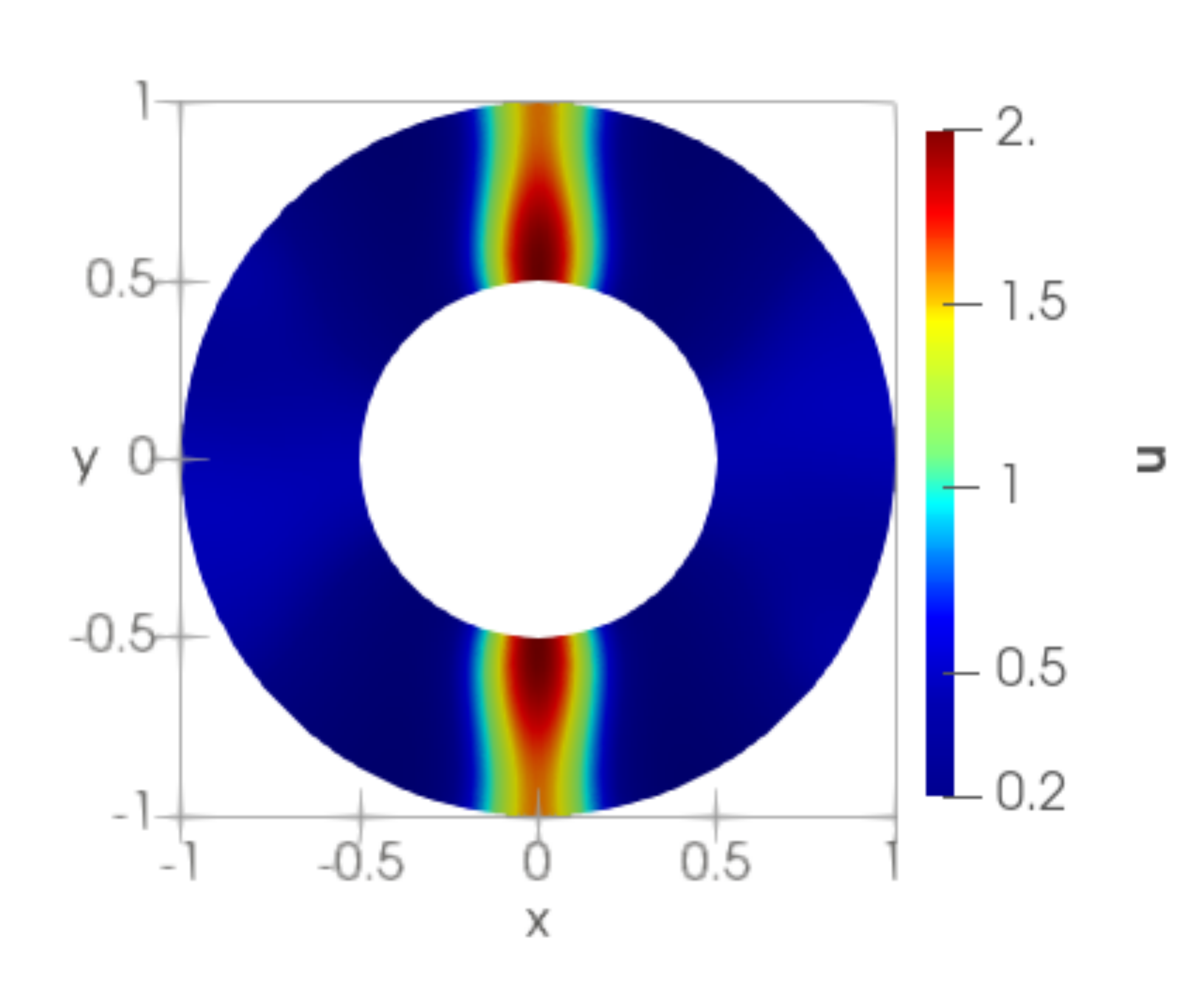} \\
{\small (a) t=0.02}
\end{tabular}
\begin{tabular}{cc}
\includegraphics[width=0.28\linewidth]{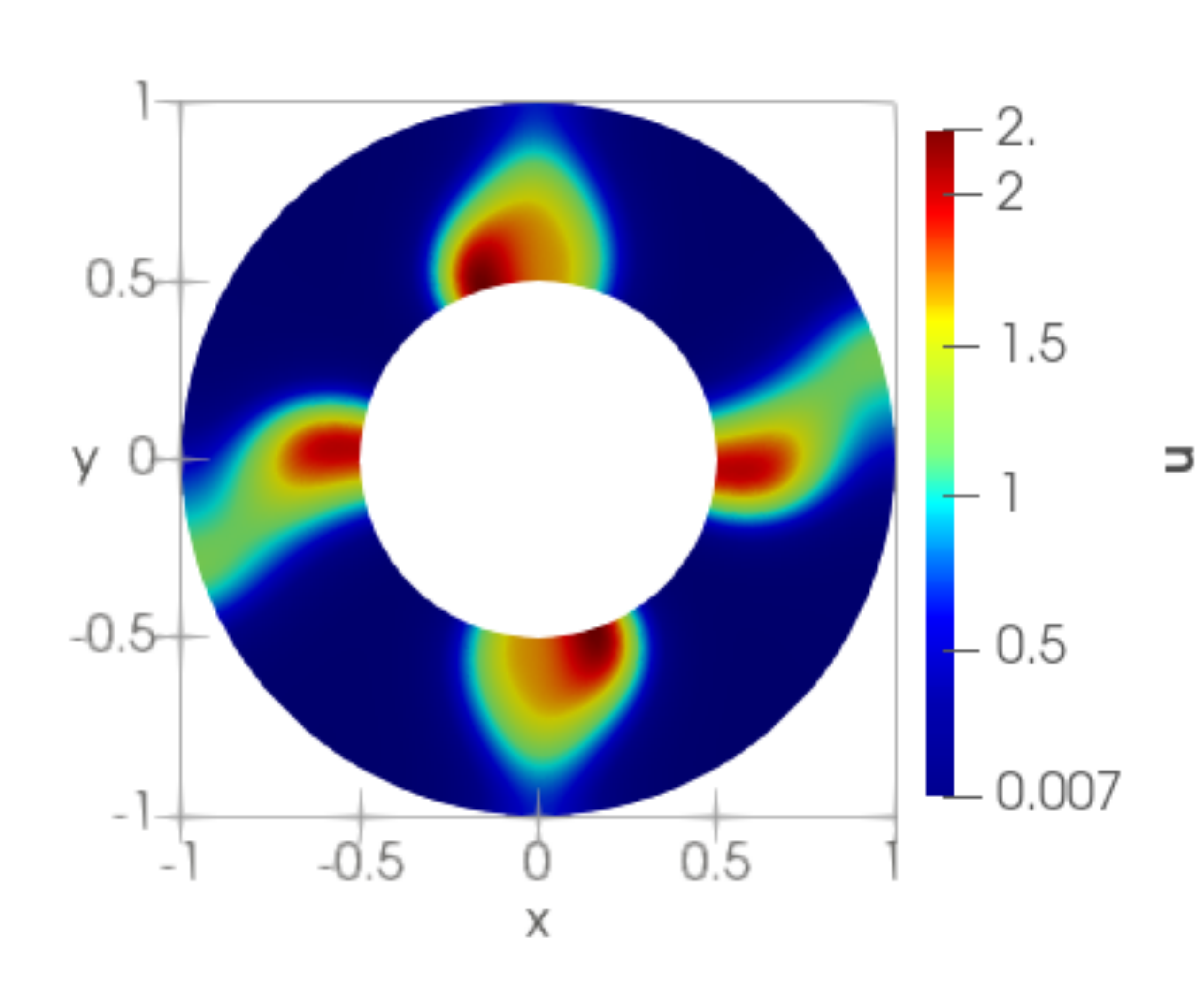} \\
{\small (b) t=0.05}
\end{tabular}
\begin{tabular}{cc}
\includegraphics[width=0.3\linewidth]{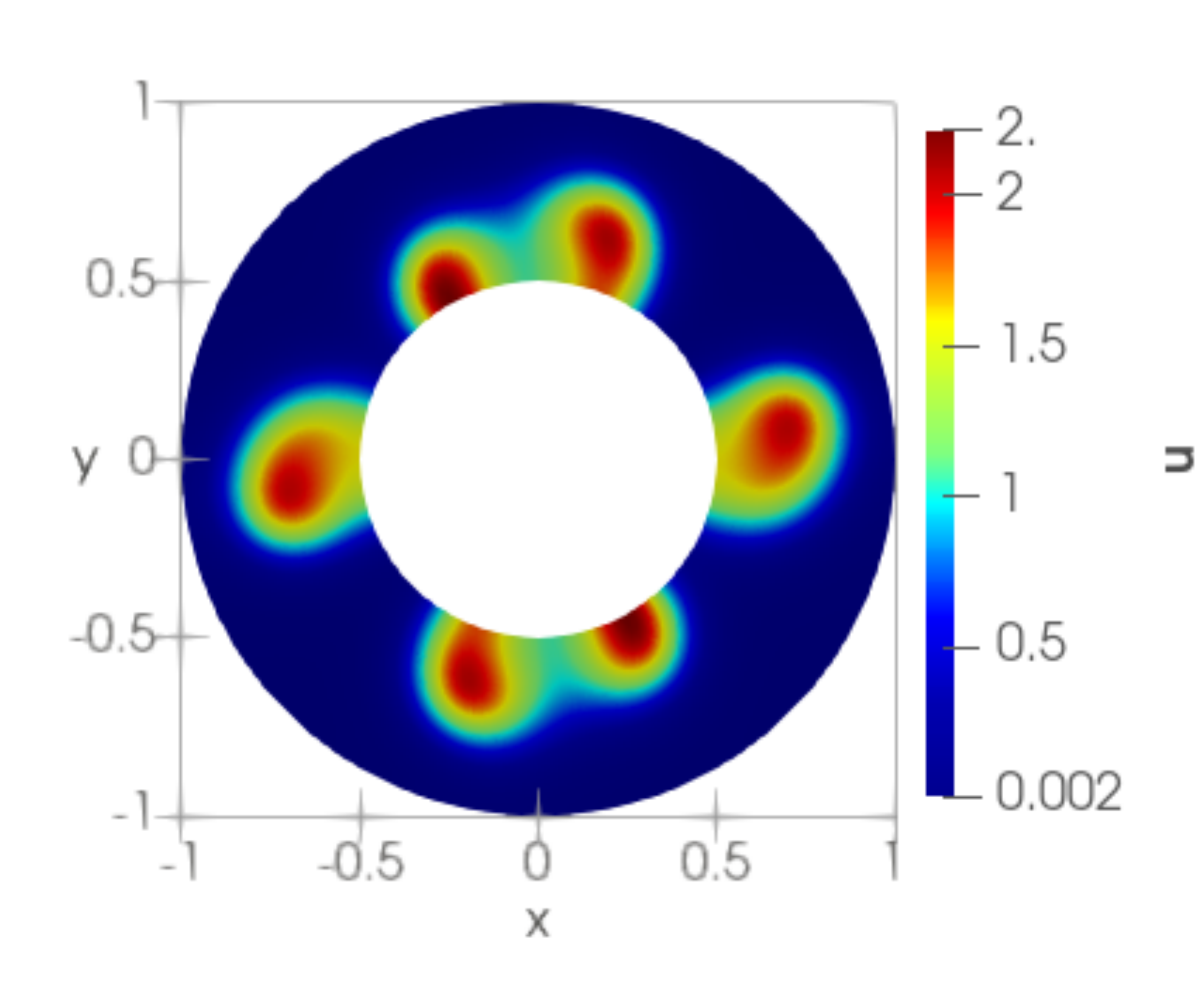} \\
{\small (c) t=0.1}
\end{tabular}
\begin{tabular}{cc}
\includegraphics[width=0.3\linewidth]{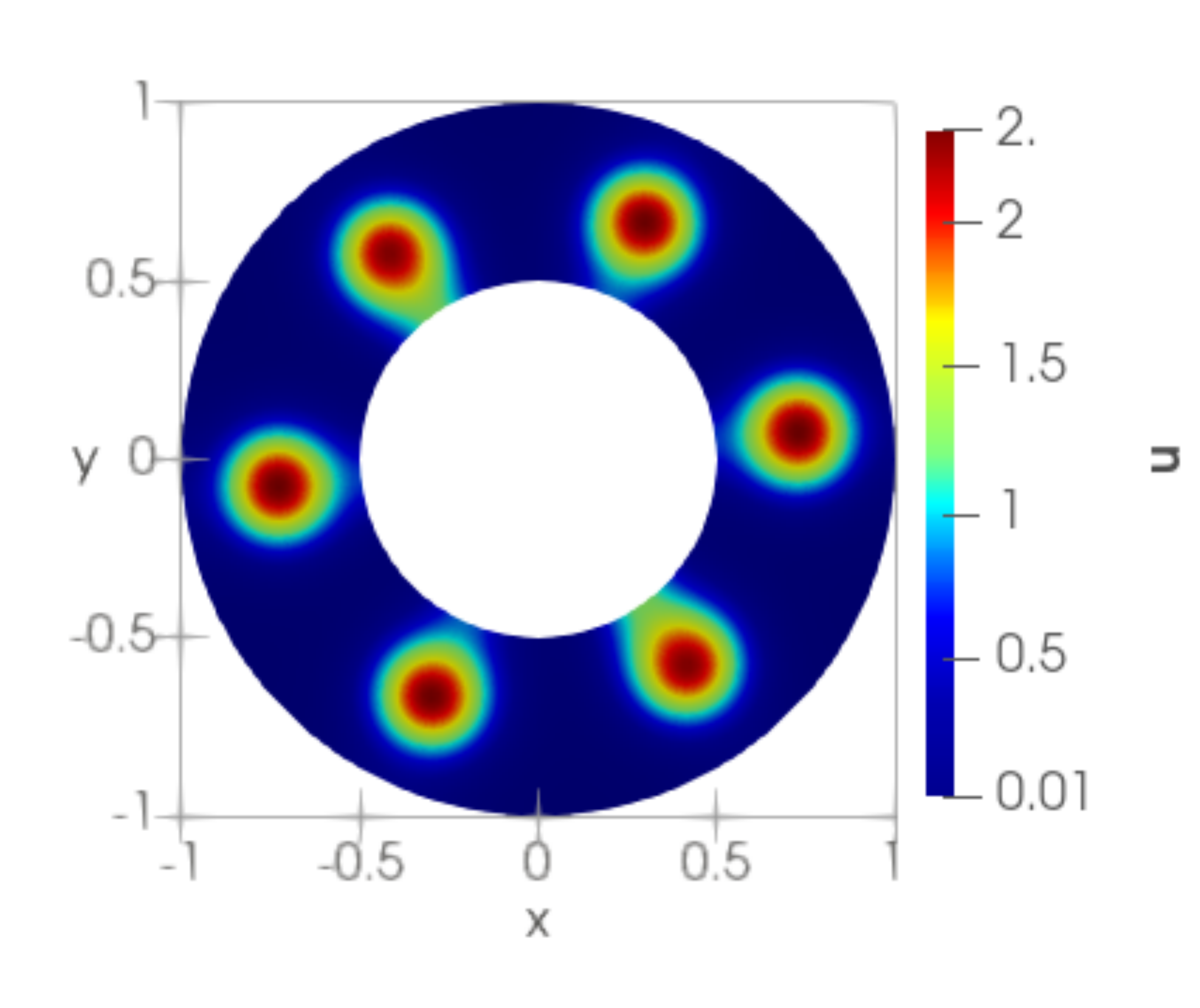} \\
{\small (d) t=0.25}
\end{tabular}
\begin{tabular}{cc}
\includegraphics[width=0.3\linewidth]{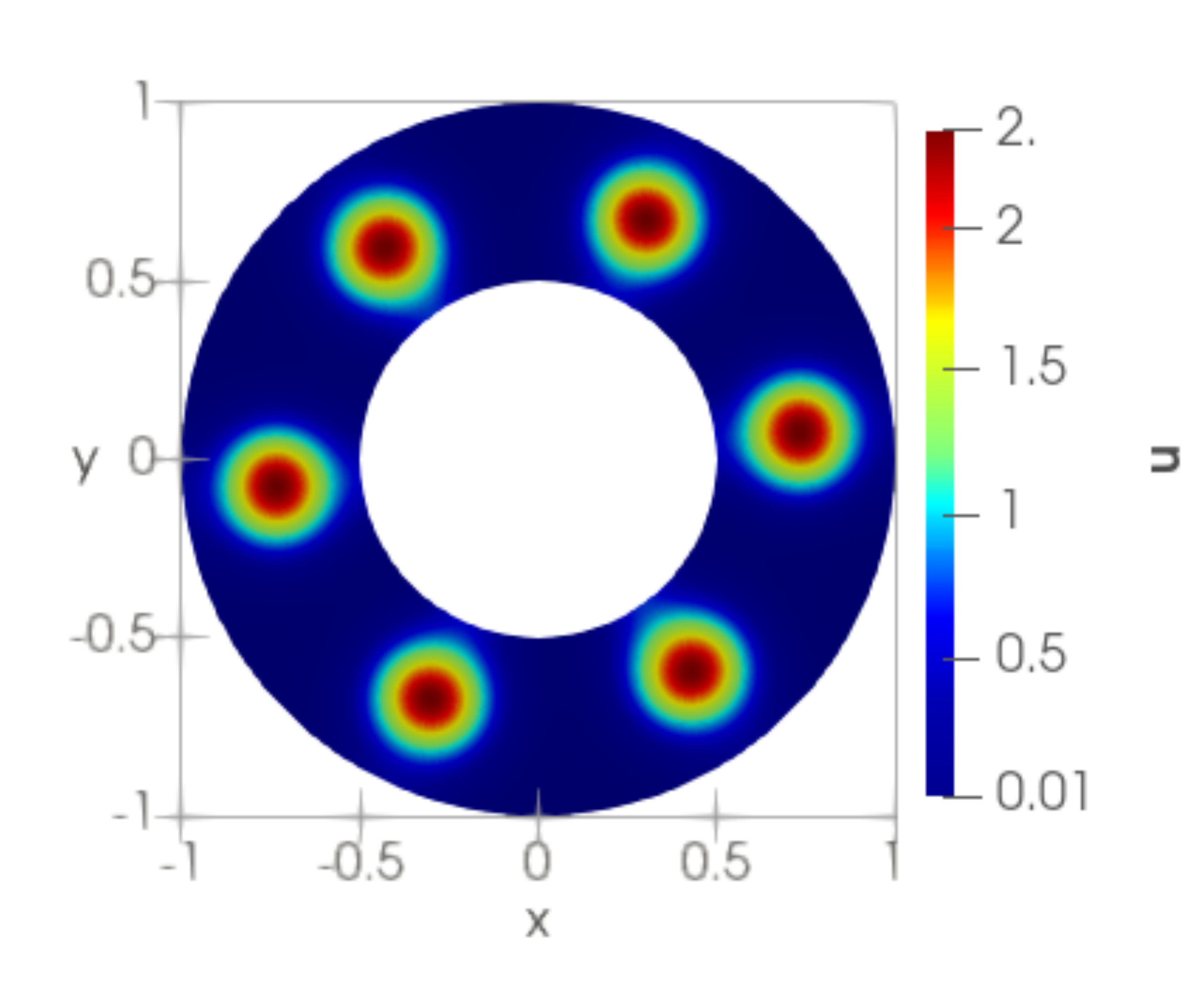} \\
{\small (e) t=0.5}
\end{tabular}
\begin{tabular}{cc}
\includegraphics[width=0.3\linewidth]{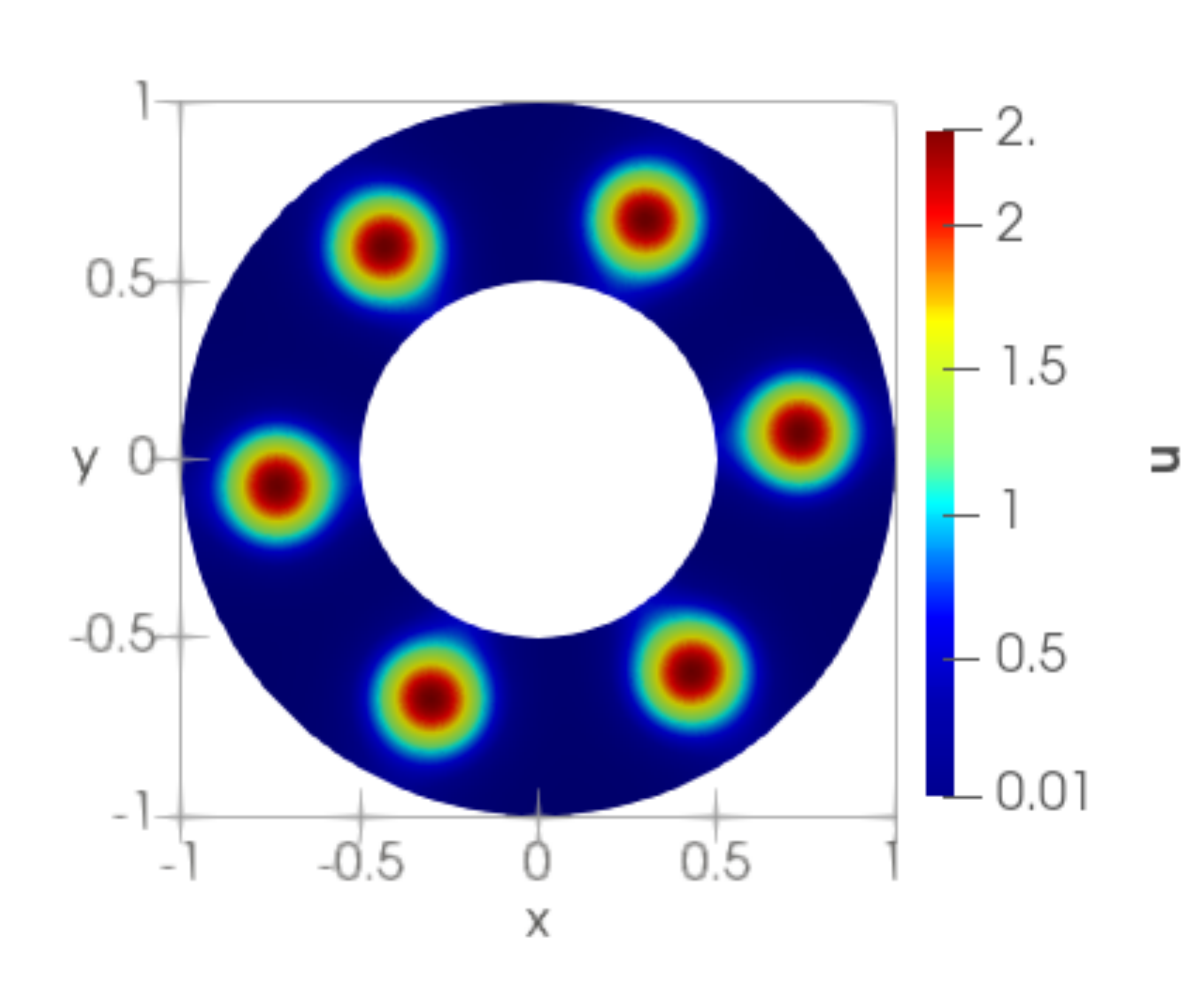} \\
{\small (f) t=1}
\end{tabular}
\begin{tabular}{cc}
\includegraphics[width=0.3\linewidth]{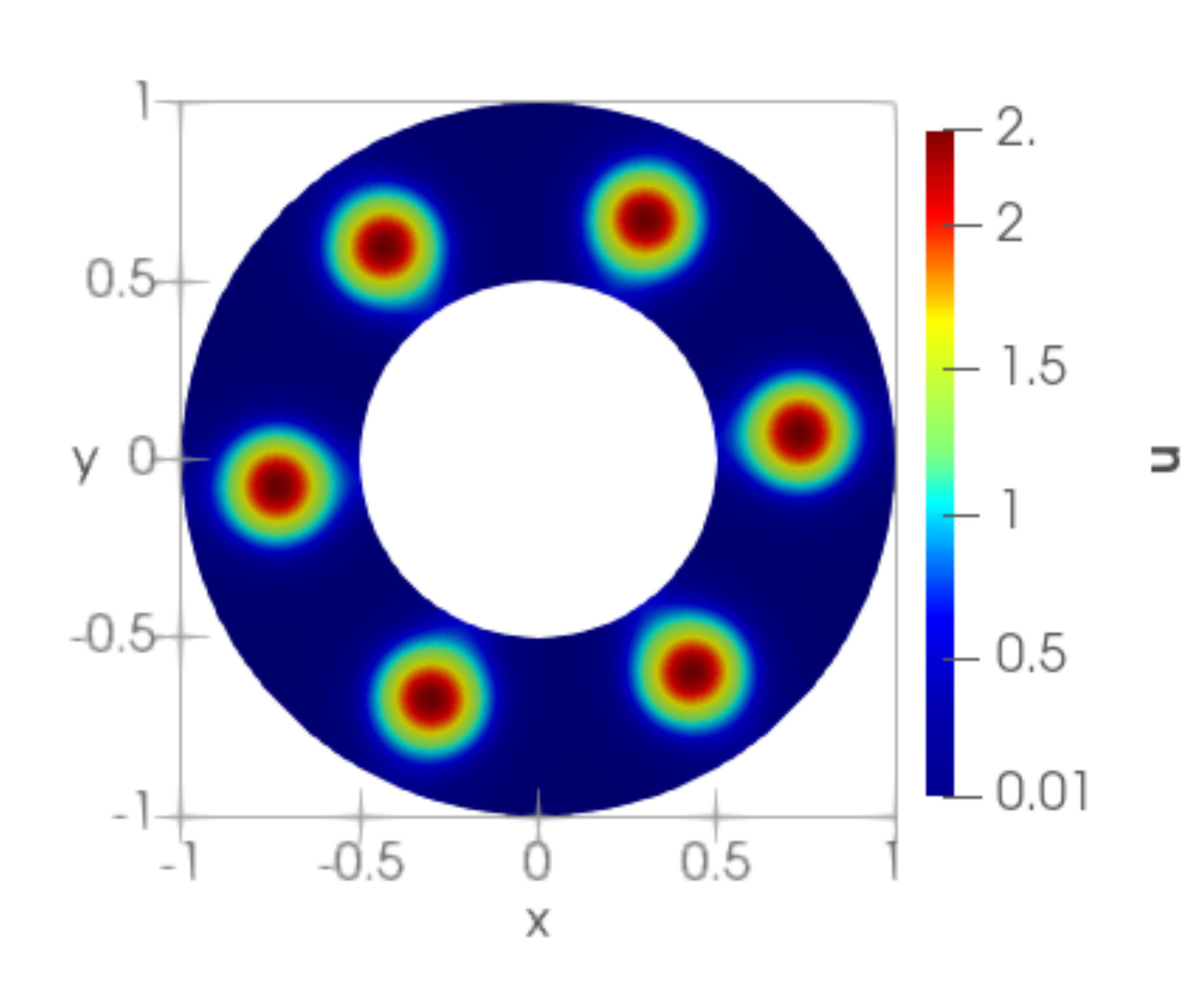} \\
{\small (g) t=2}
\end{tabular}
\begin{tabular}{cc}
\includegraphics[width=0.64\textwidth]{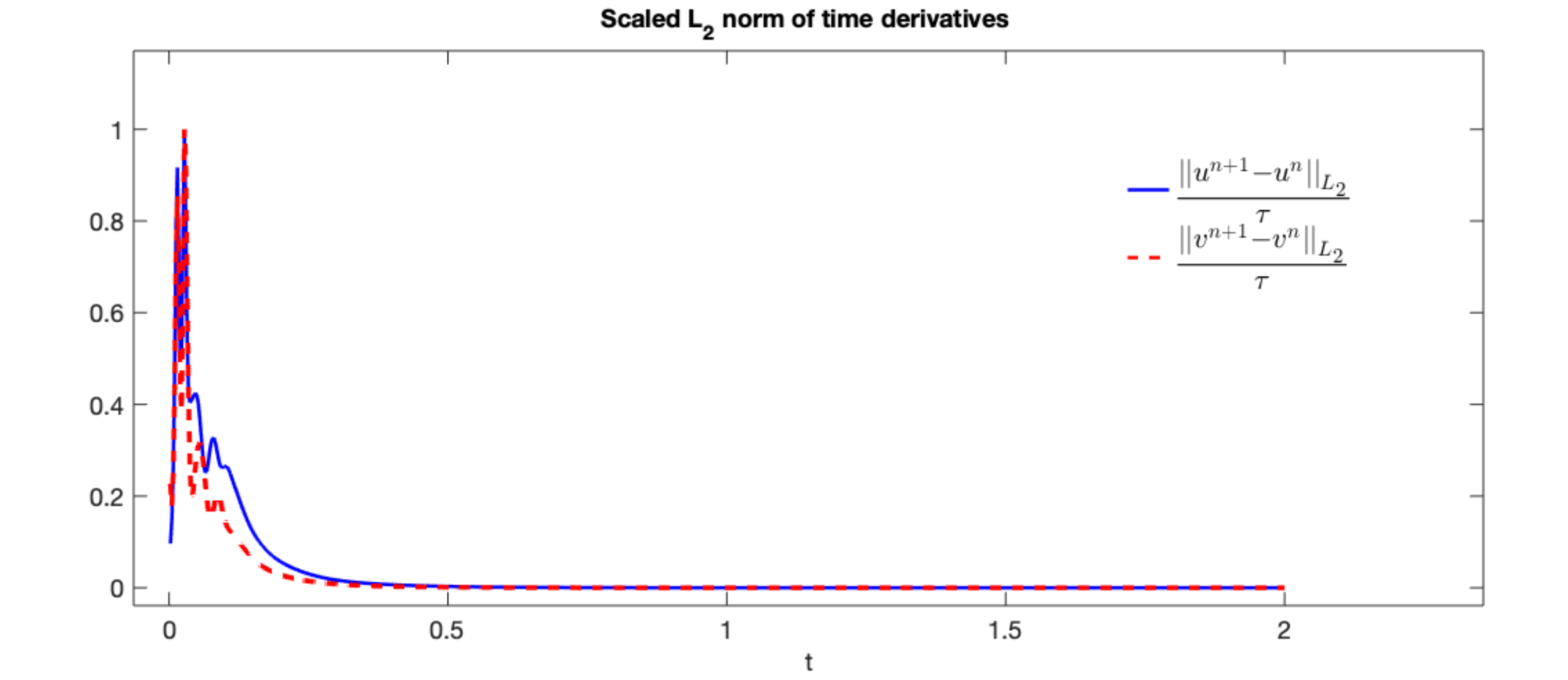} \\
{\small (h) Discrete time derivative of the solutions $u$ and $v$}
\end{tabular}
\caption{(a)-(g) Finite element simulations corresponding to the $u$-component of the cross-diffusive reaction-diffusion system on annulus domain illustrating the transient process to a spatially inhomogeneous, and time-independent pattern. Parameters are selected to satisfy conditions of Theorems \ref{theo2} and \ref{Maincond} with $\Delta t =0.0025$, $d=12$, $\gamma=270$, $d_u=1$, $d_v=1.7$, $\alpha=0.07$ and $\beta=0.45$, as shown in Fig. \ref{TuringTheo2}. (h) Plot of the $L_2$ norms showing the convergence of the  discrete time-derivative of the numerical solutions $u$ and $v$.}
\label{Turing2}
\end{figure}

\begin{figure}[H]
\begin{tabular}{cc}
\includegraphics[width=0.28\linewidth]{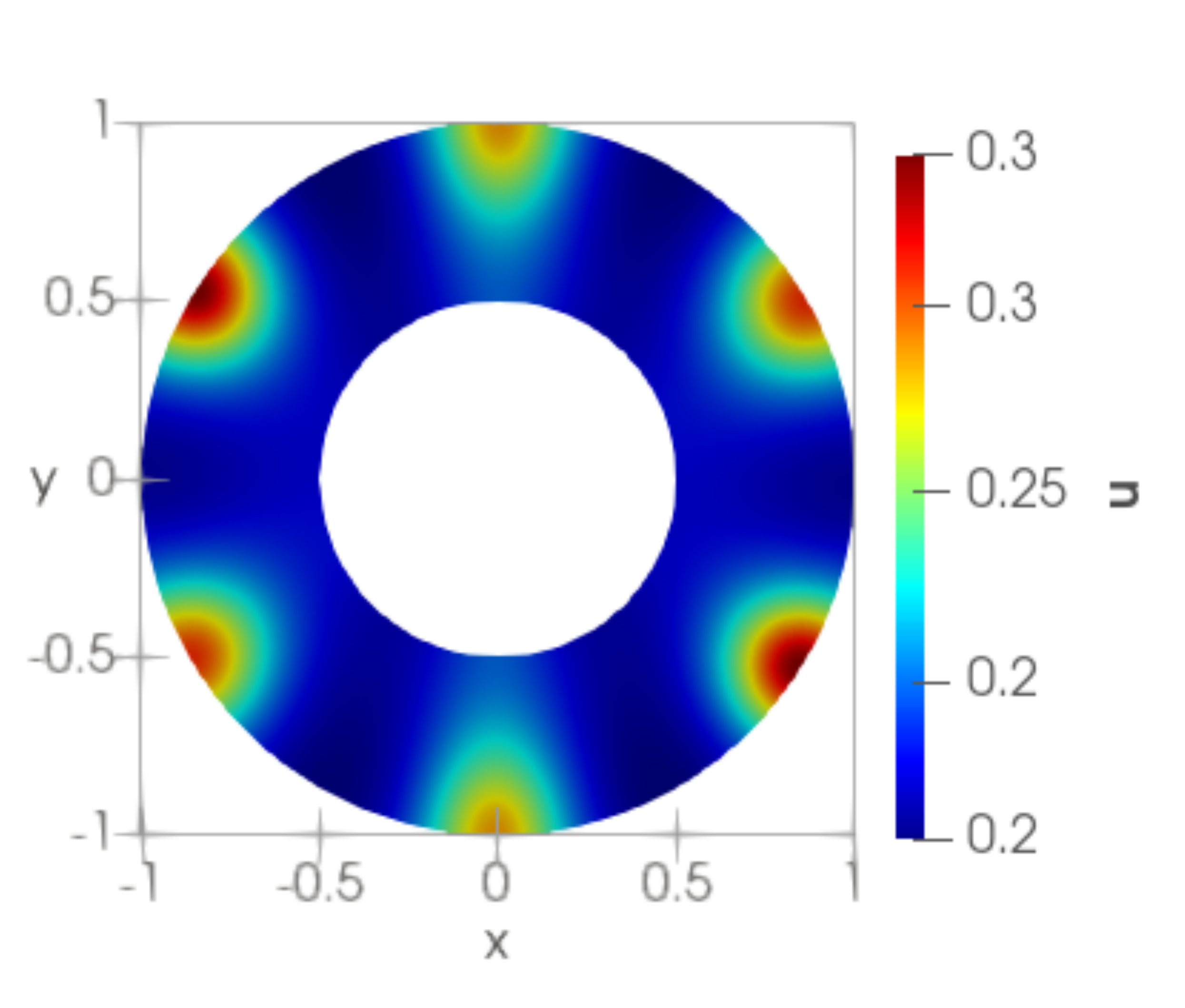} \\
{\small (a) t=0.1}
\end{tabular}
\begin{tabular}{cc}
\includegraphics[width=0.28\linewidth]{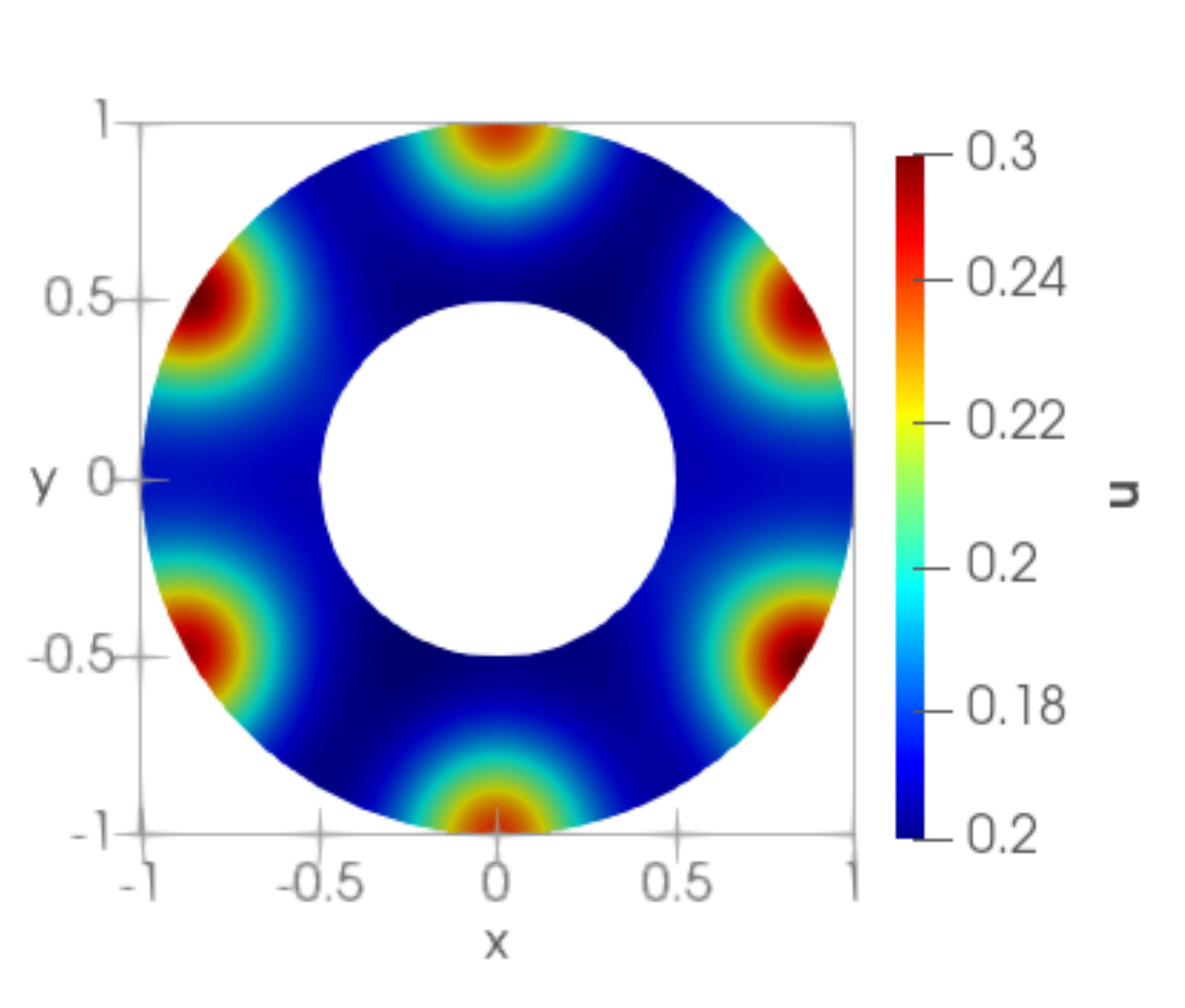} \\
{\small (b) t=0.3}
\end{tabular}
\begin{tabular}{cc}
\includegraphics[width=0.3\linewidth]{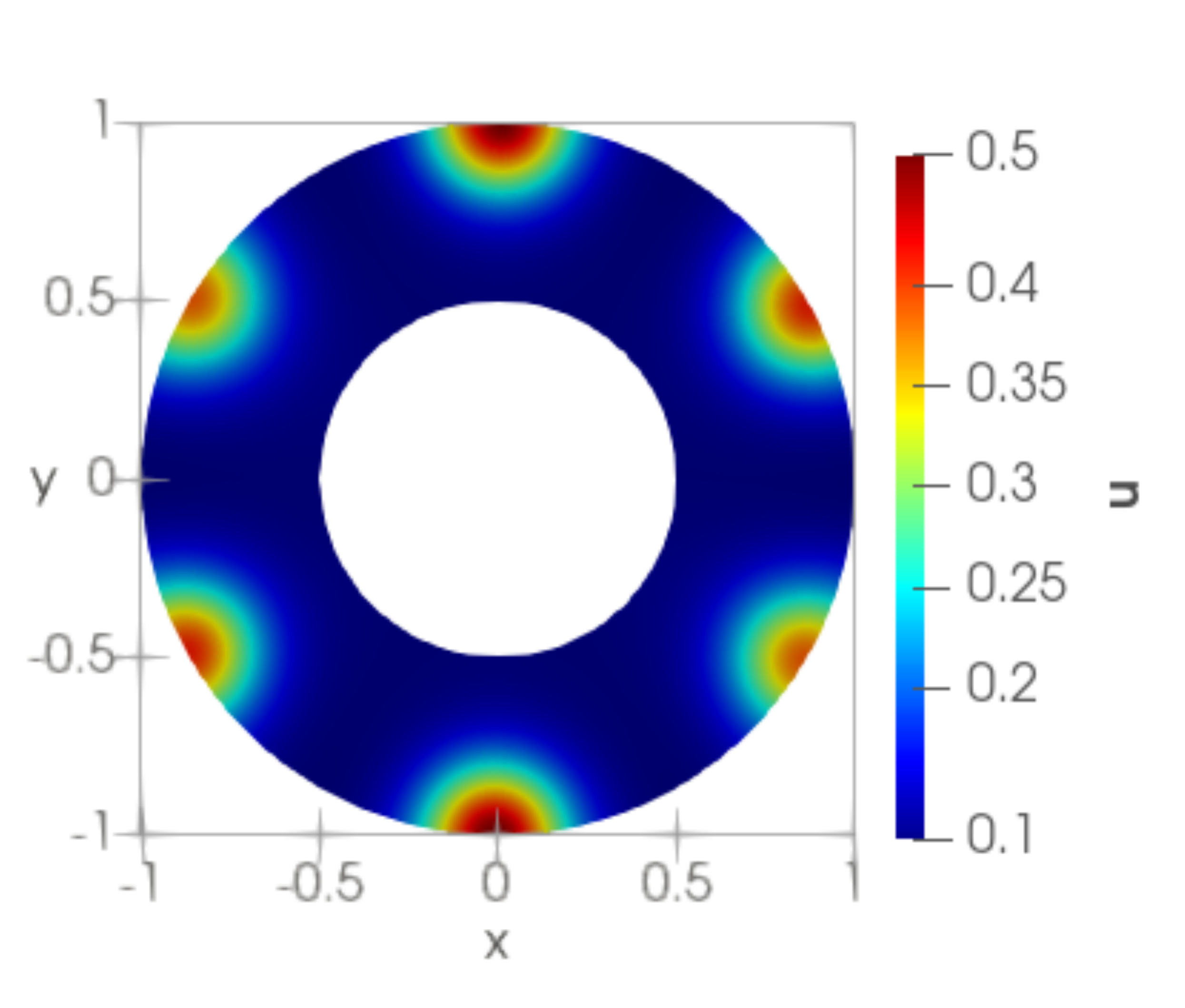} \\
{\small (c) t=0.45}
\end{tabular}
\begin{tabular}{cc}
\includegraphics[width=0.3\linewidth]{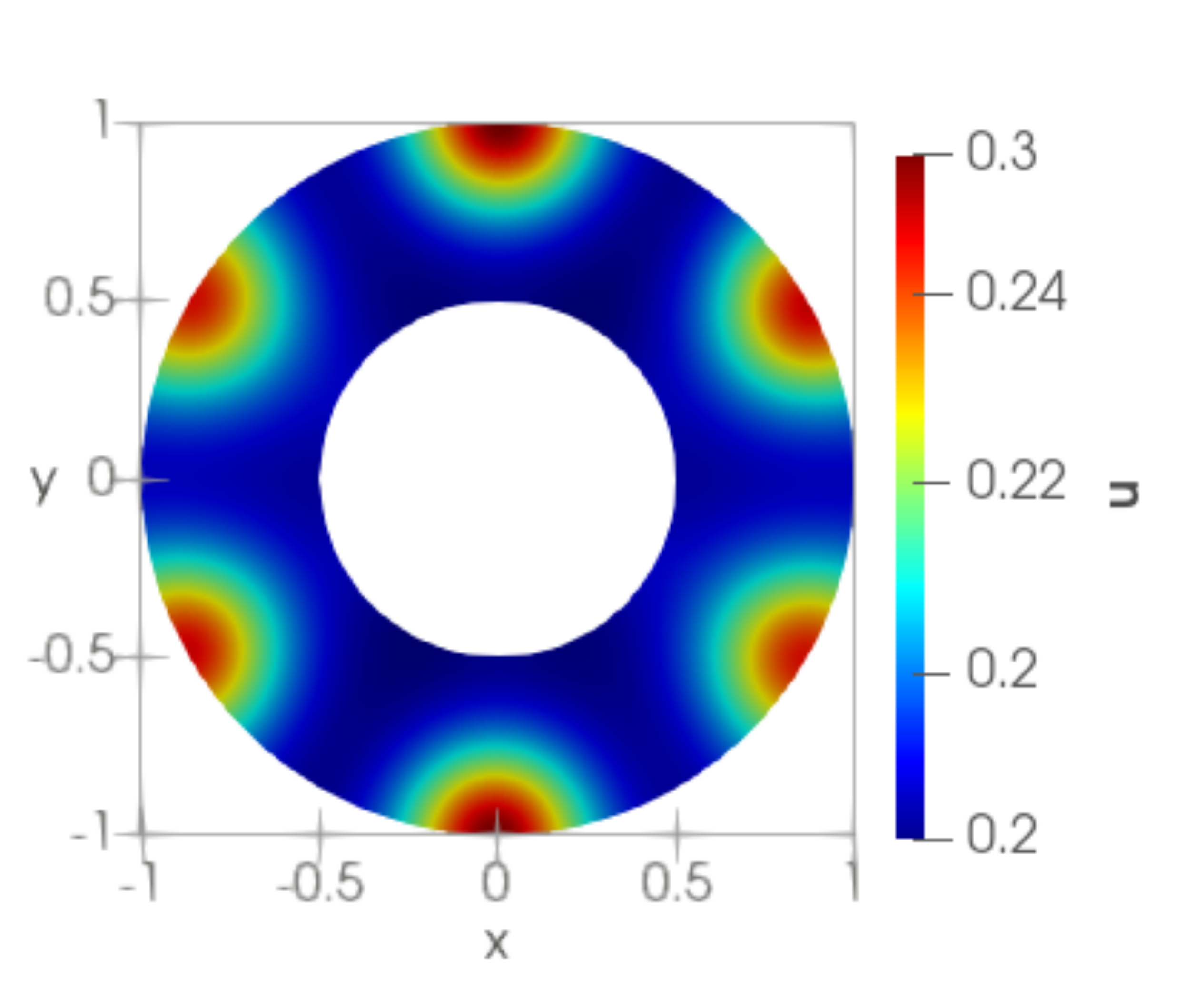} \\
{\small (d) t=0.50}
\end{tabular}
\begin{tabular}{cc}
\includegraphics[width=0.3\linewidth]{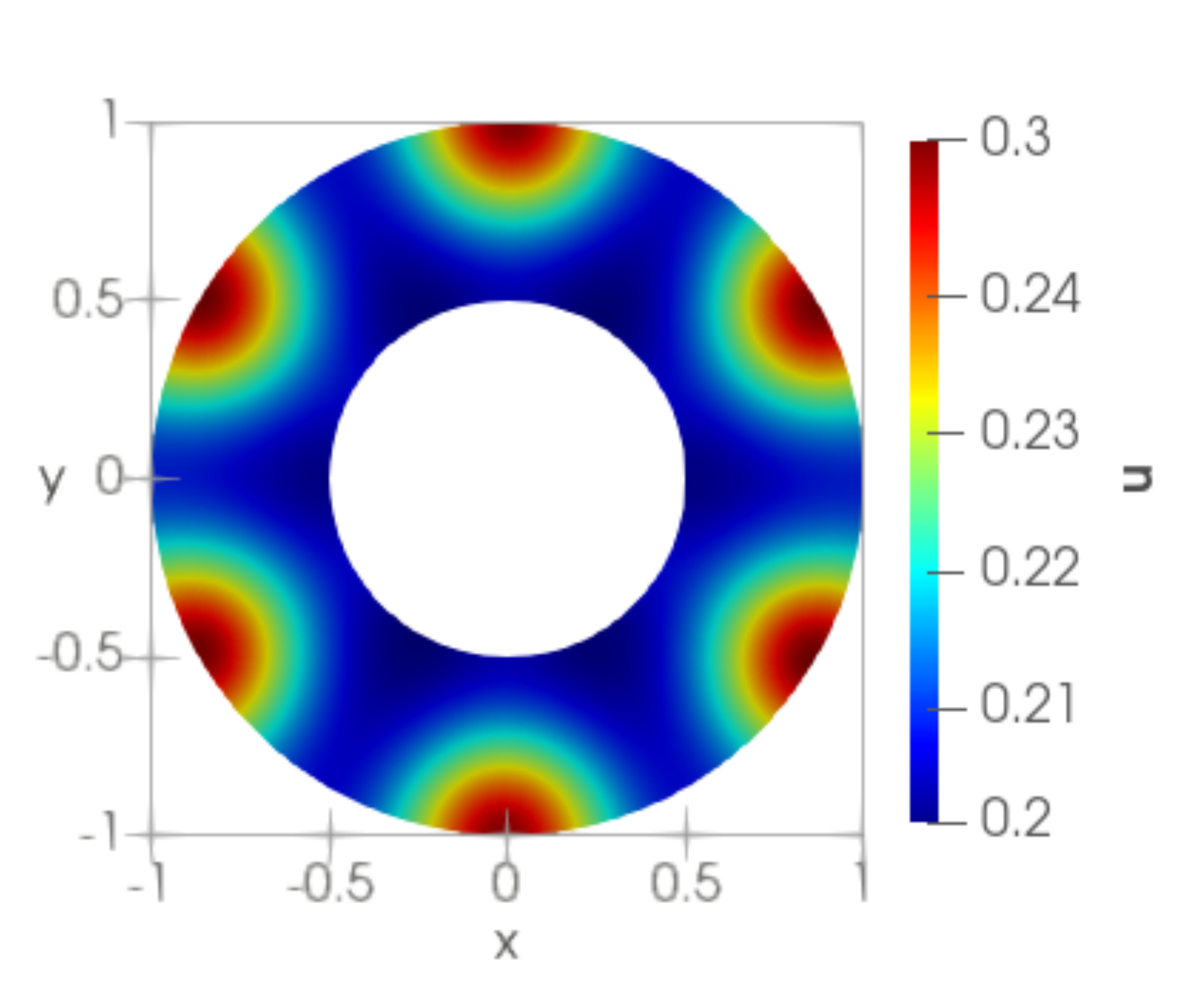} \\
{\small (e) t=0.825}
\end{tabular}
\begin{tabular}{cc}
\includegraphics[width=0.3\linewidth]{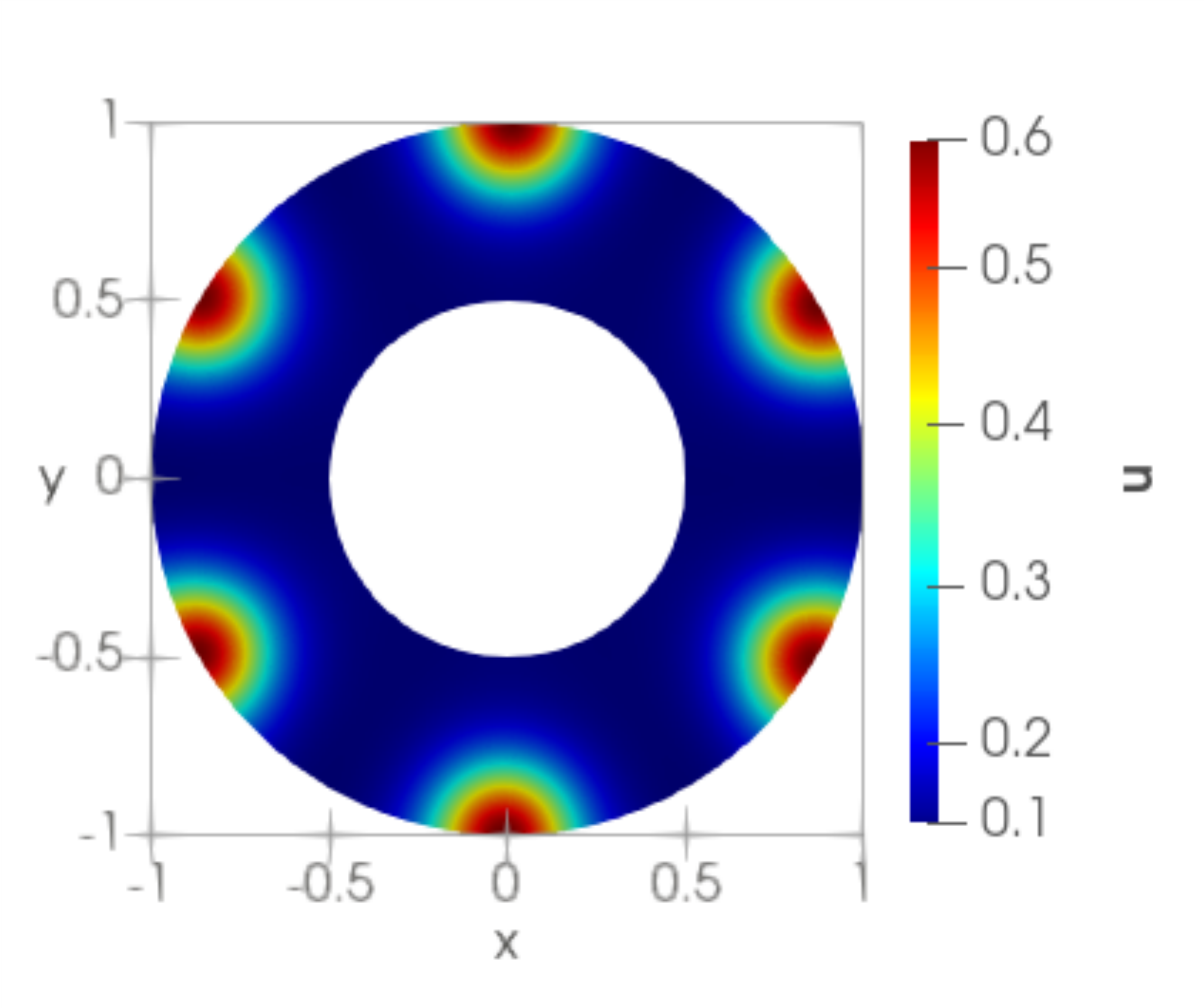} \\
{\small (f) t=1.27}
\end{tabular}
\begin{tabular}{cc}
\includegraphics[width=0.3\linewidth]{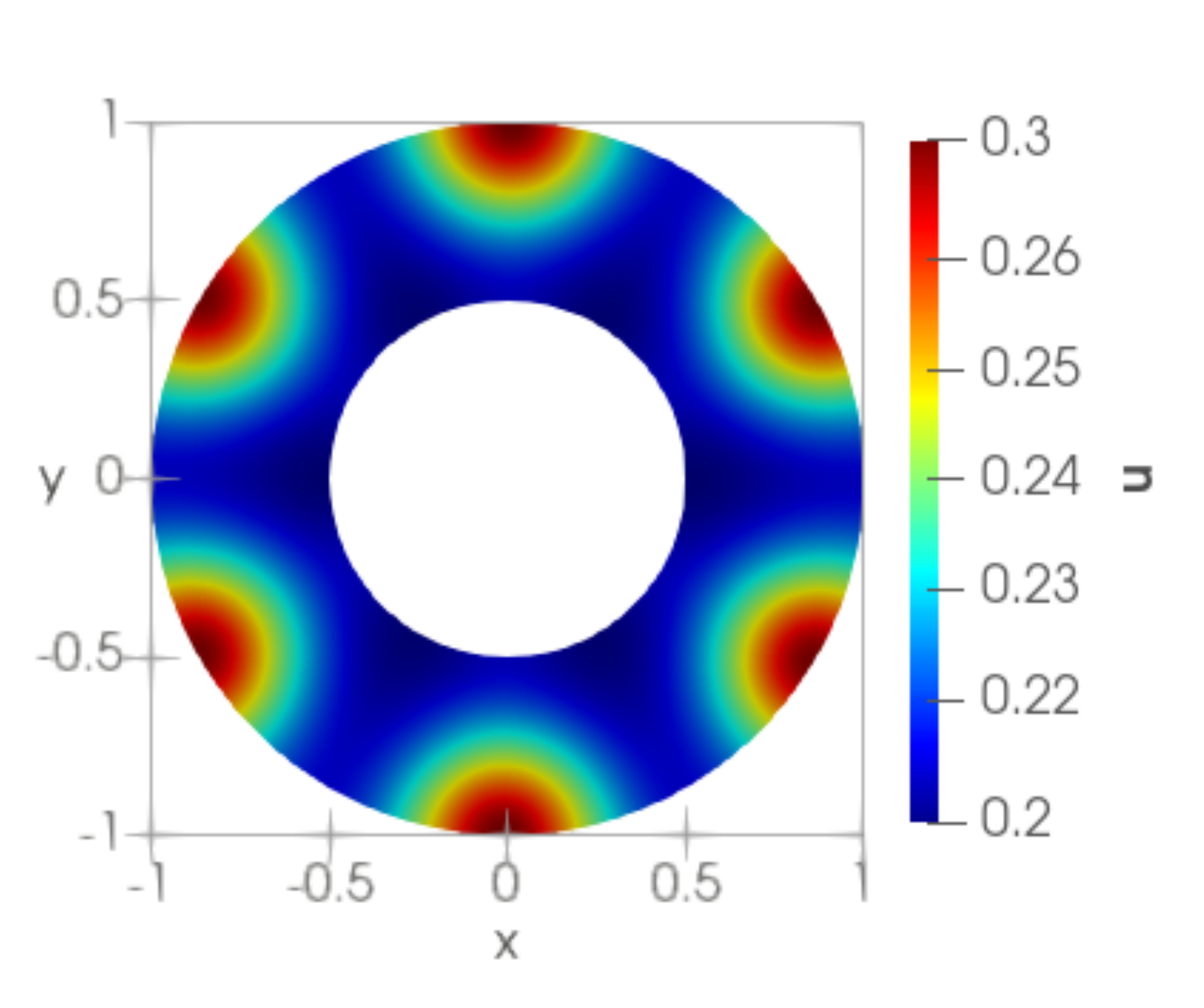} \\
{\small (g) t=1.45}
\end{tabular}
\begin{tabular}{cc}
\includegraphics[width=0.3\linewidth]{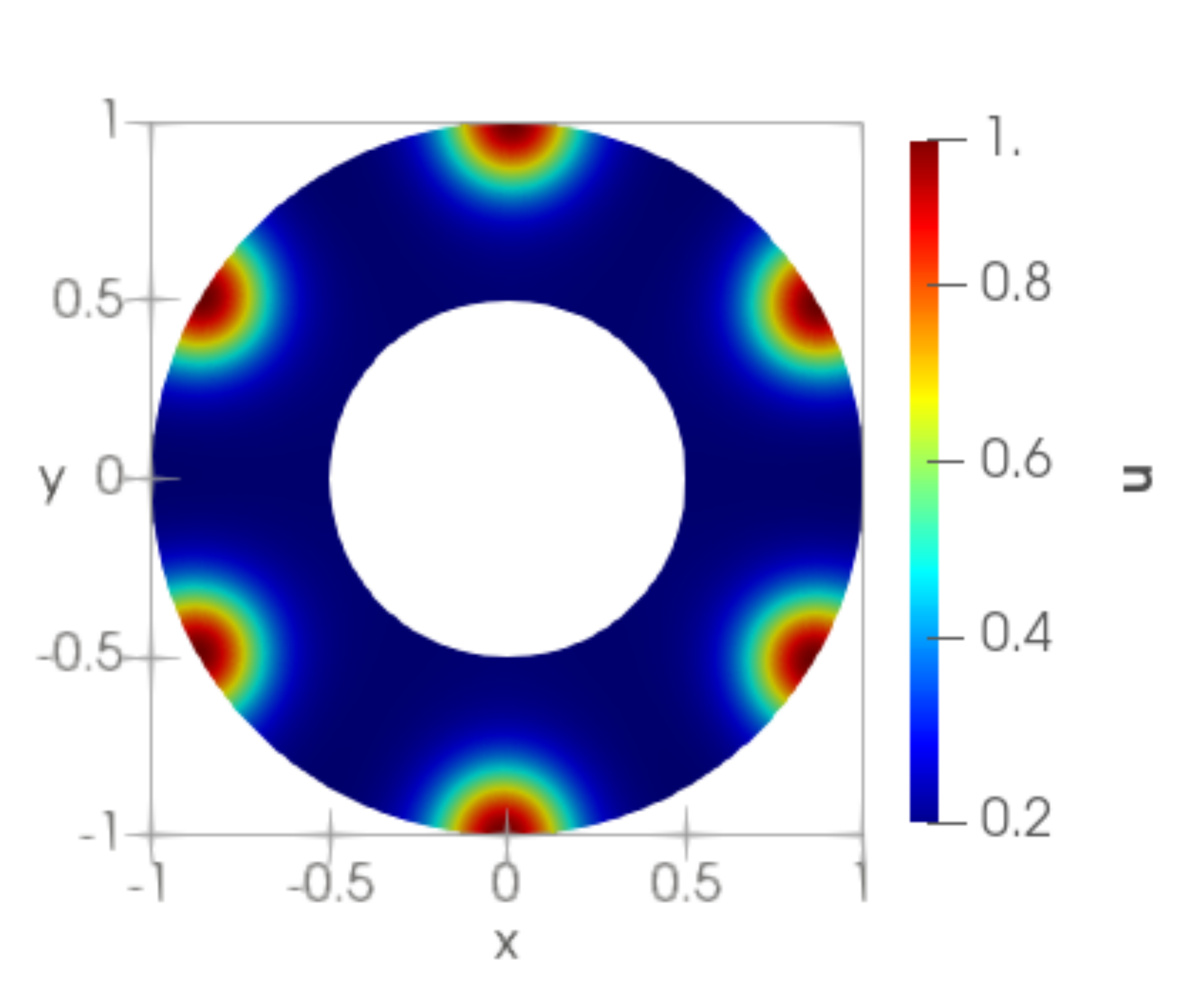} \\
{\small (h) t=1.982}
\end{tabular}
\begin{tabular}{cc}
\includegraphics[width=0.3\linewidth]{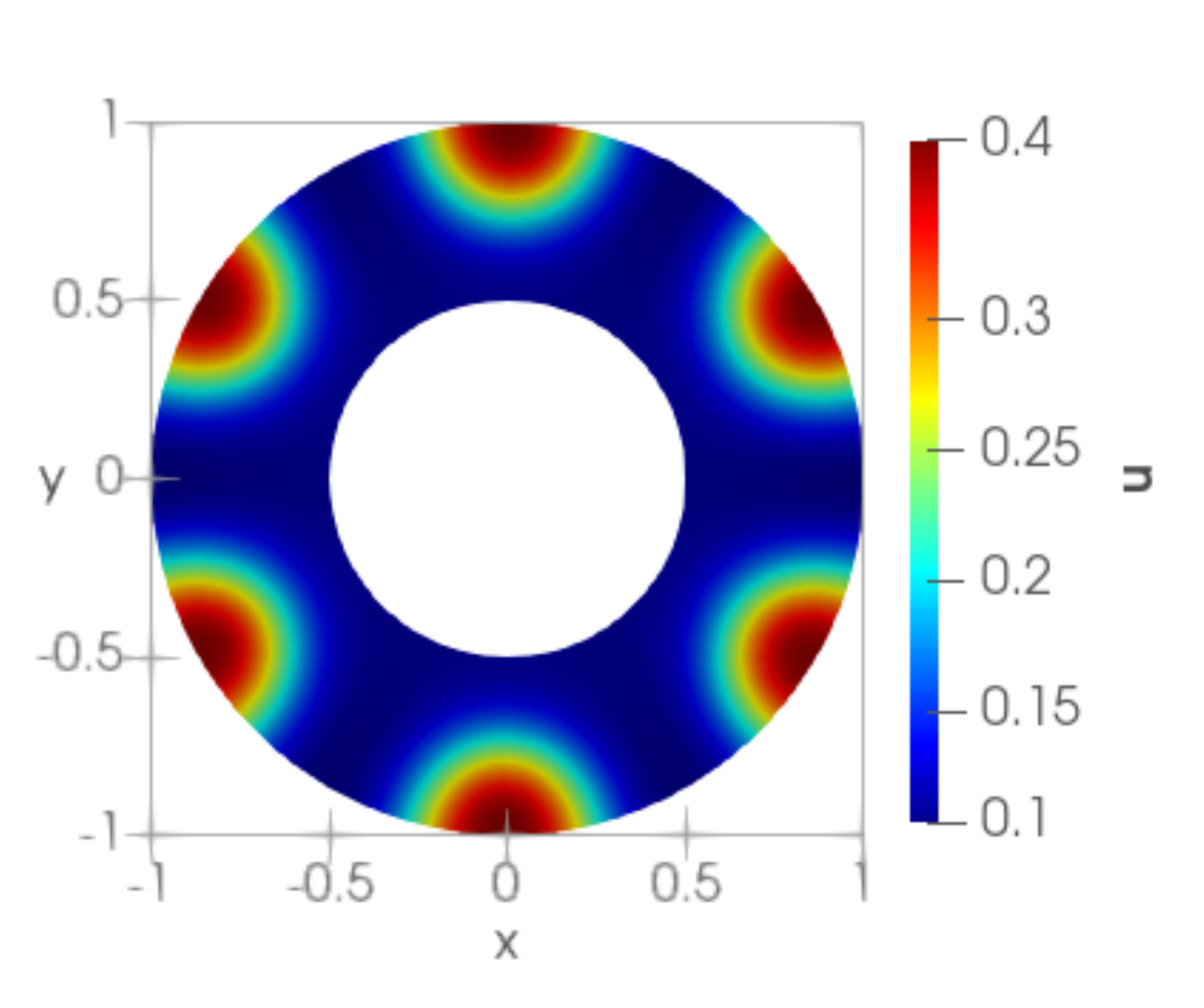} \\
{\small (i) t=2}
\end{tabular}
\begin{tabular}{cc}
\includegraphics[width=0.3\linewidth]{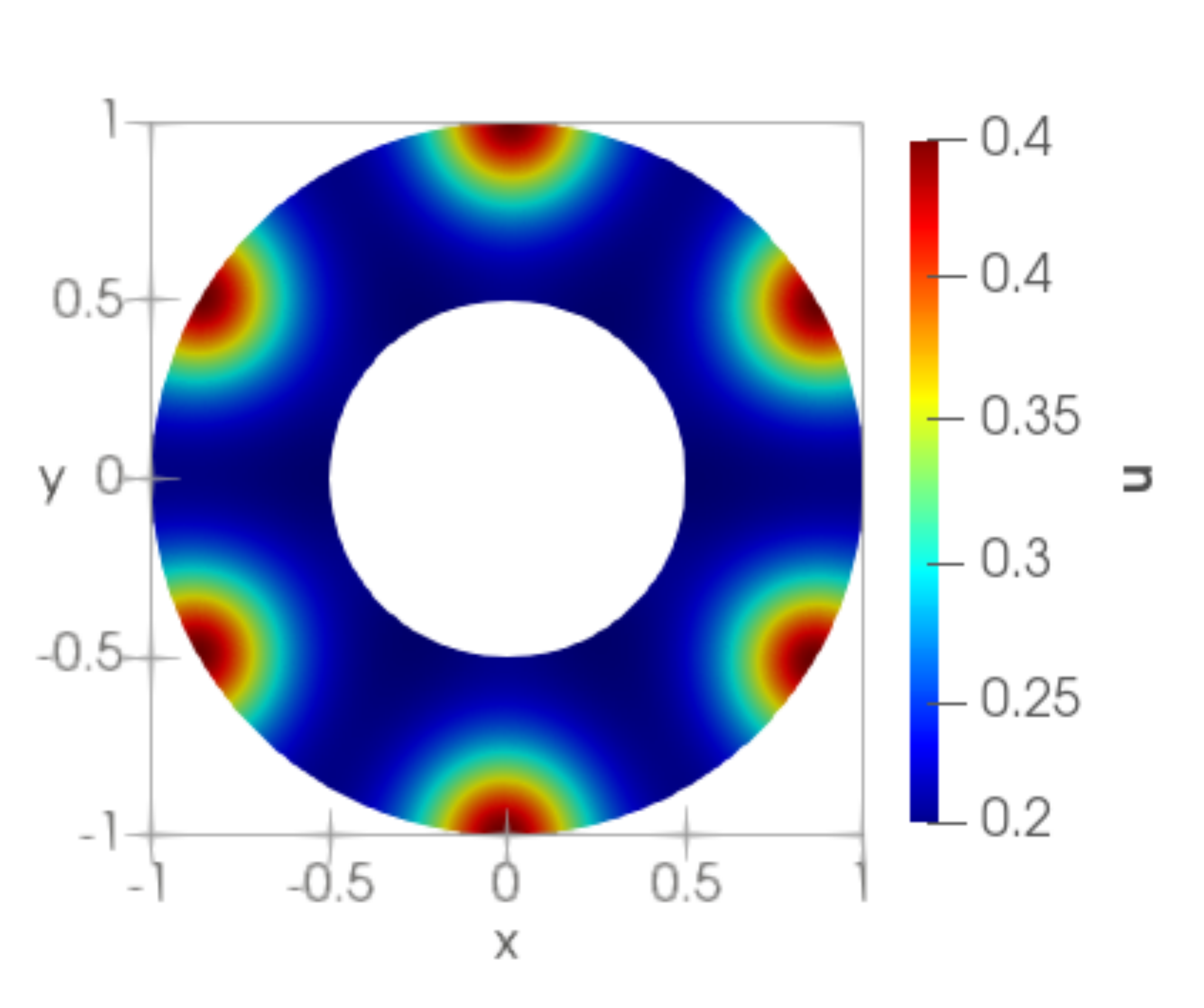} \\
{\small (j) t=2.5}
\end{tabular}
\begin{tabular}{cc}
\includegraphics[width=0.64\textwidth]{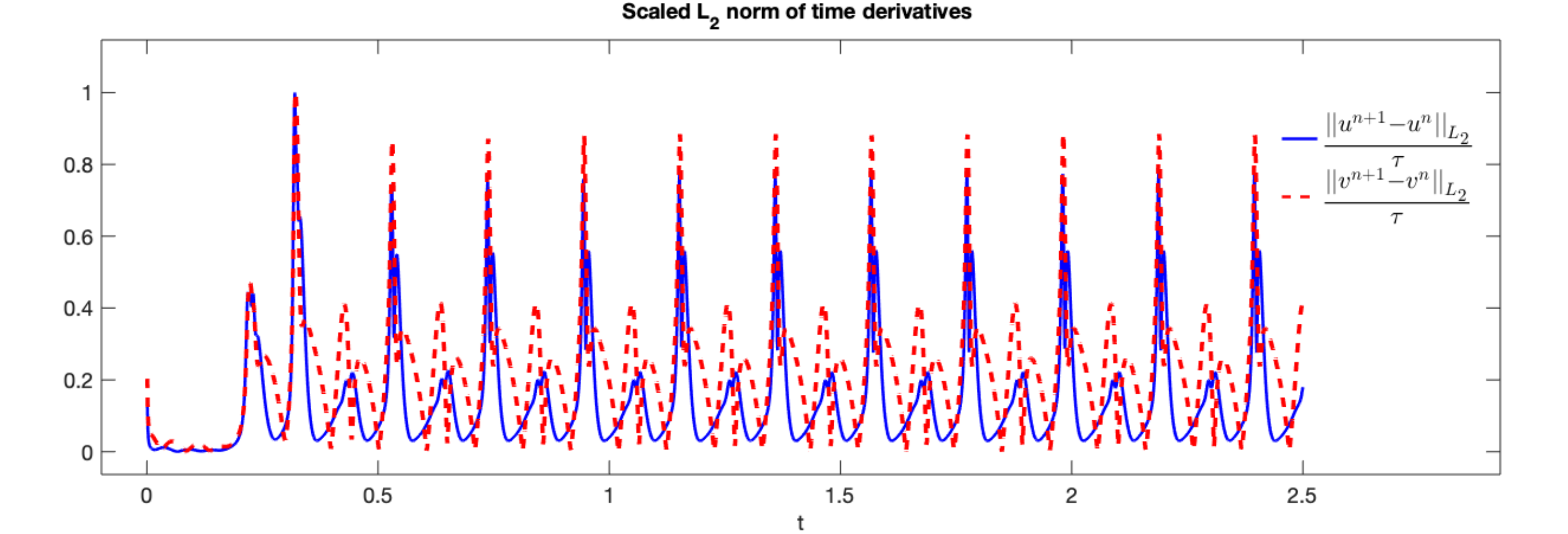} \\
{\small (k) Discrete time derivative of the solutions $u$ and $v$}
\end{tabular}
\caption{(a)-(j) Finite element simulations corresponding to the $u$-component of the cross-diffusive reaction-diffusion system on disc shape domain exhibiting {\it periodic} spatiotemporal time-periodic pattern formation, also known as {\it standing waves}. Parameters are selected to satisfy conditions of Theorems \ref{theo1} and \ref{Maincond} on $\rho$ with $\Delta t =0.0025$, $d=2.6$, $\gamma=375$, $d_u=1.6$, $d_v=0.5$, $\alpha=0.09$ and $\beta=0.1$, as shown in Fig. \ref{HopfTransParameter}. (k) Plot of the $L_2$ norms of the discrete time-derivatives showing the periodicity of the solutions $u$ and $v$.}
\label{Limitcycle1}
\end{figure}

\begin{figure}[H]
\begin{tabular}{cc}
\includegraphics[width=0.28\linewidth]{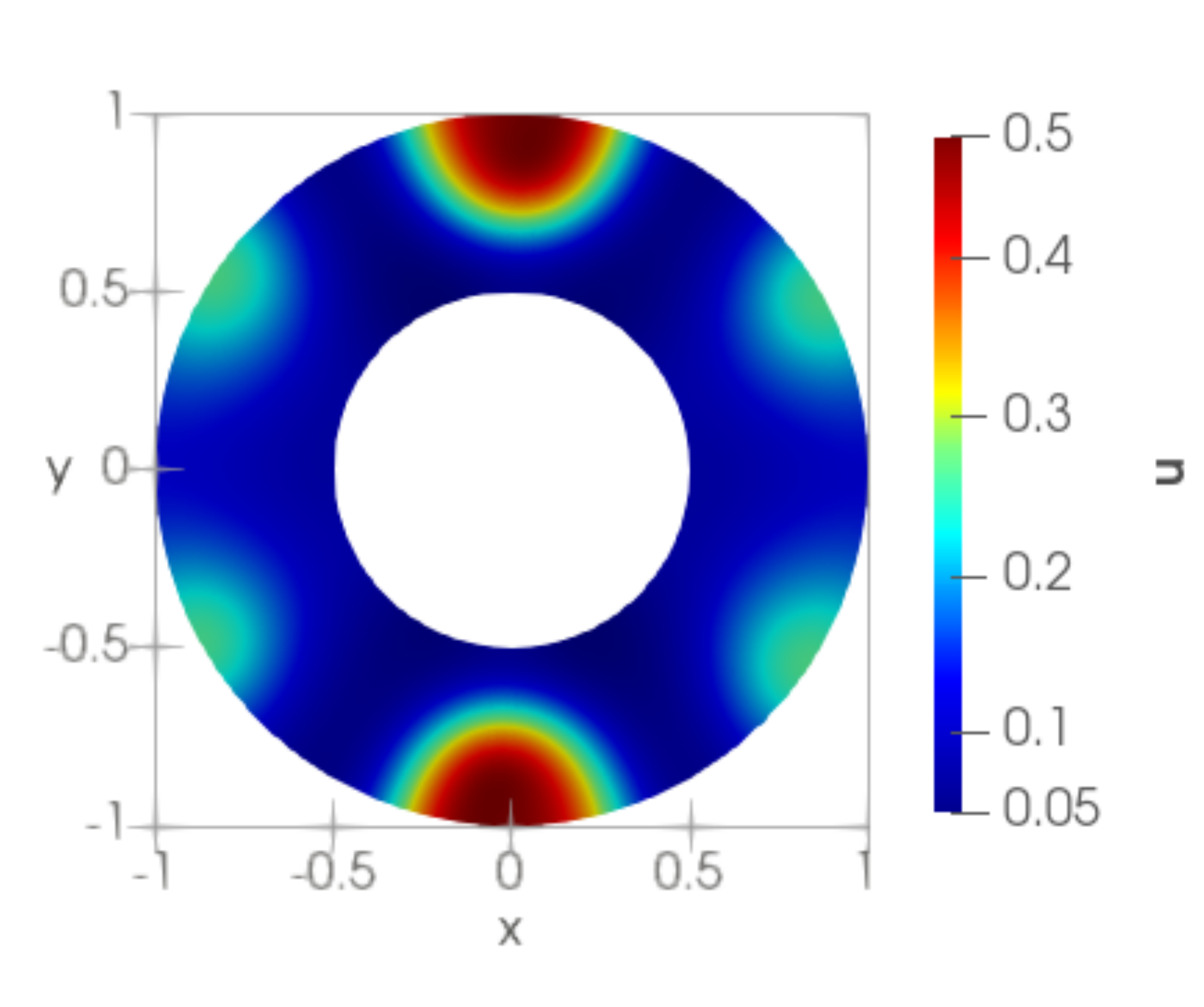} \\
{\small (a) t=0.235}
\end{tabular}
\begin{tabular}{cc}
\includegraphics[width=0.28\linewidth]{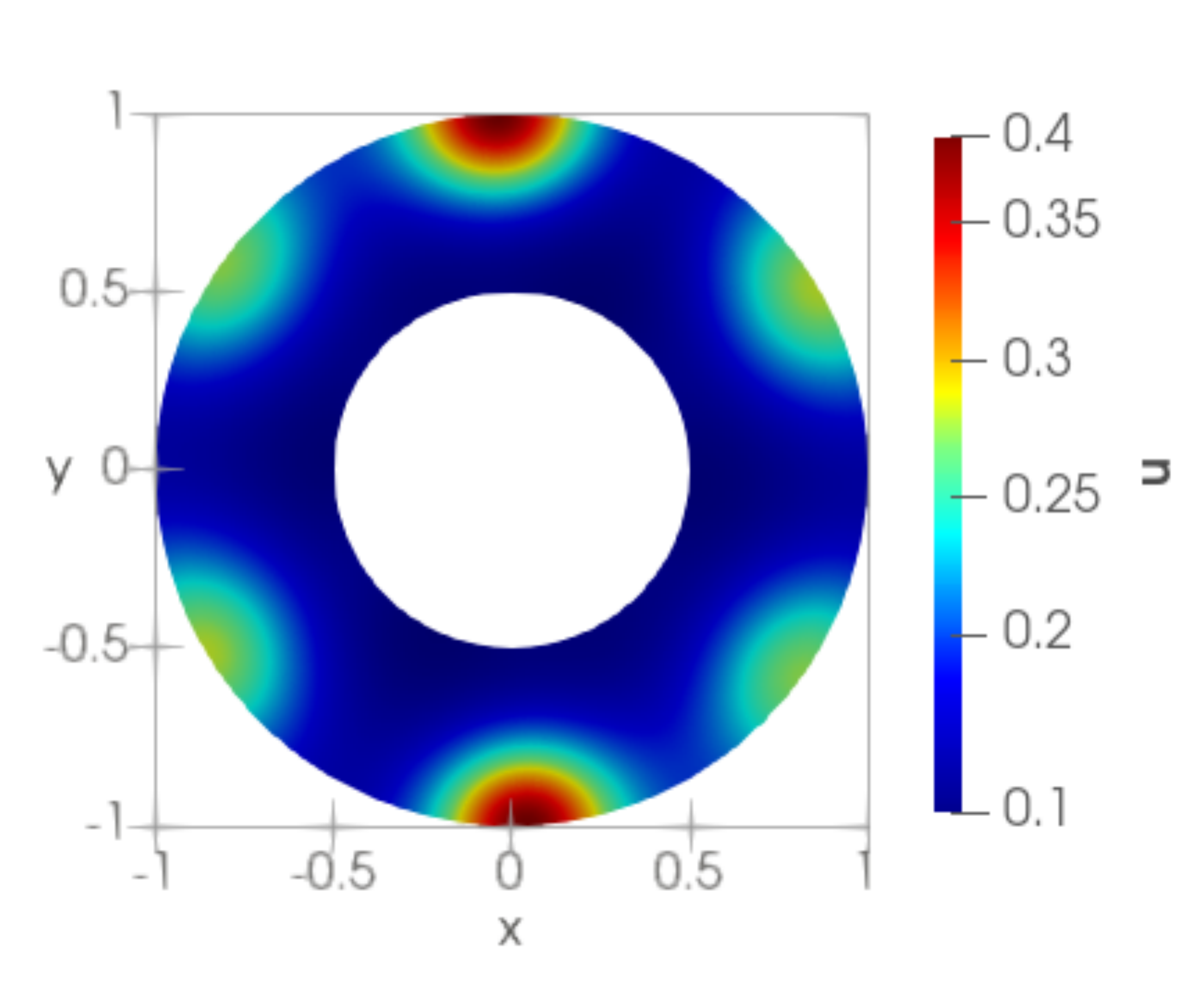} \\
{\small (b) t=0.7}
\end{tabular}
\begin{tabular}{cc}
\includegraphics[width=0.3\linewidth]{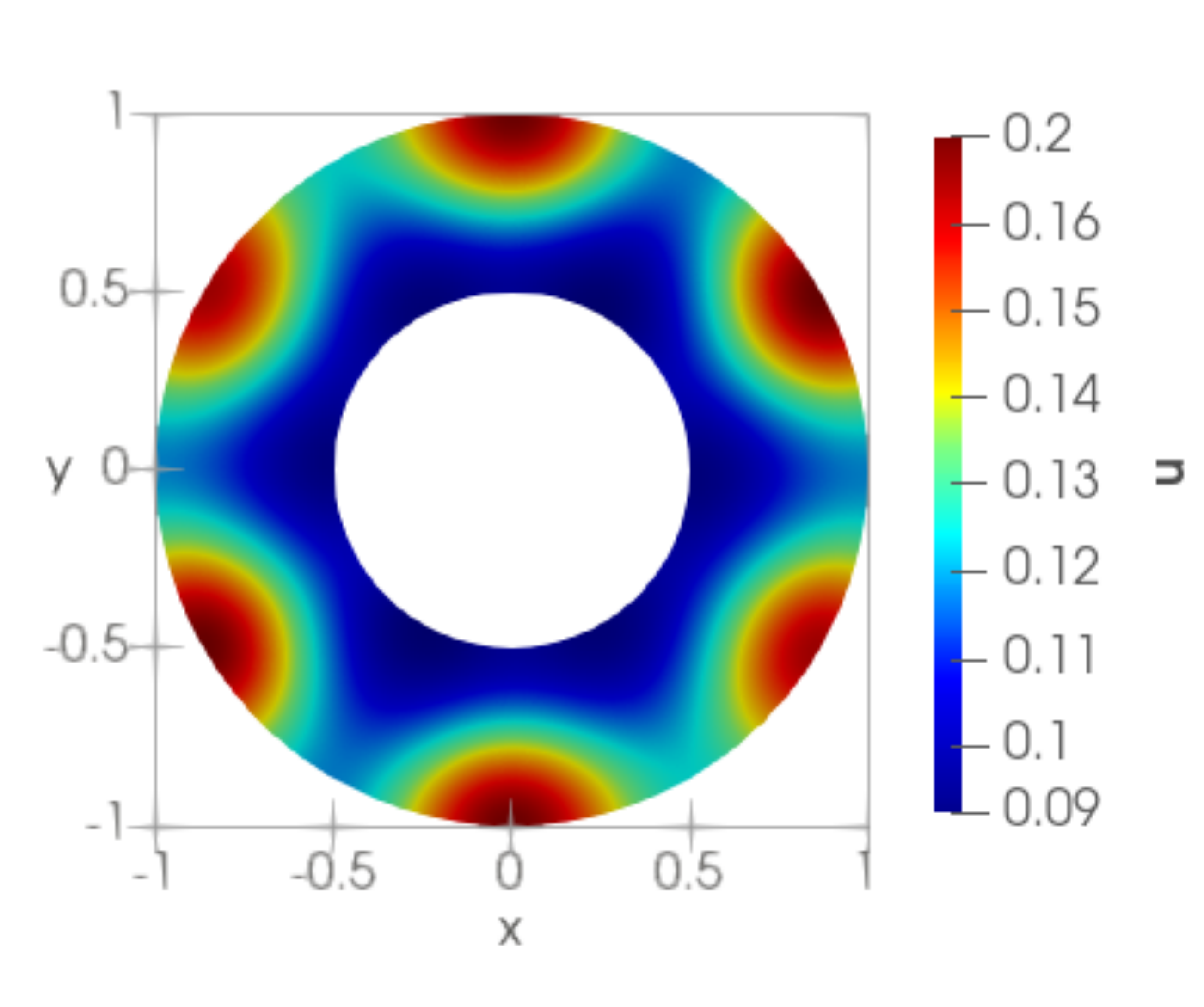} \\
{\small (c) t=0.945}
\end{tabular}
\begin{tabular}{cc}
\includegraphics[width=0.3\linewidth]{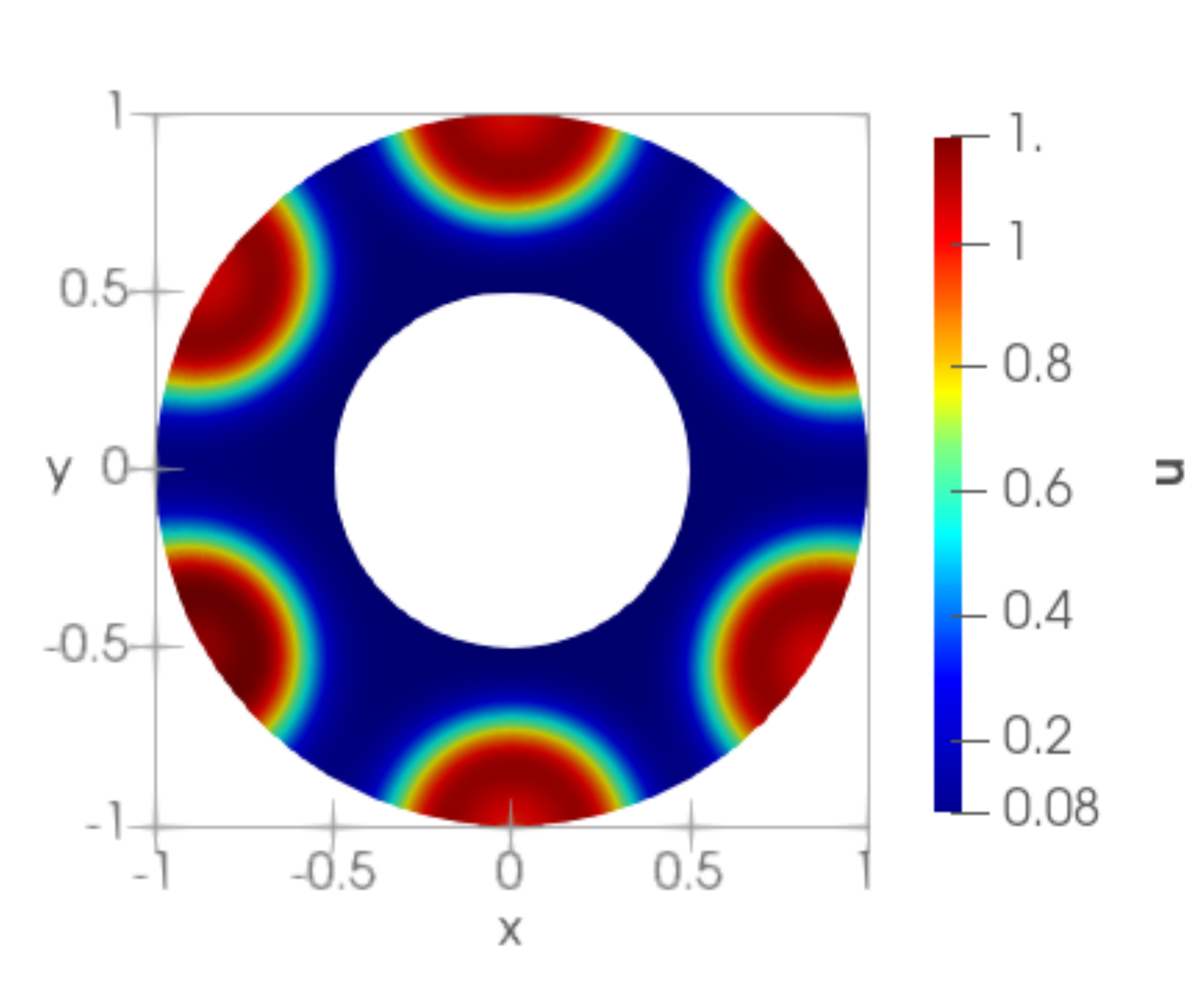} \\
{\small (d) t=1.495}
\end{tabular}
\begin{tabular}{cc}
\includegraphics[width=0.3\linewidth]{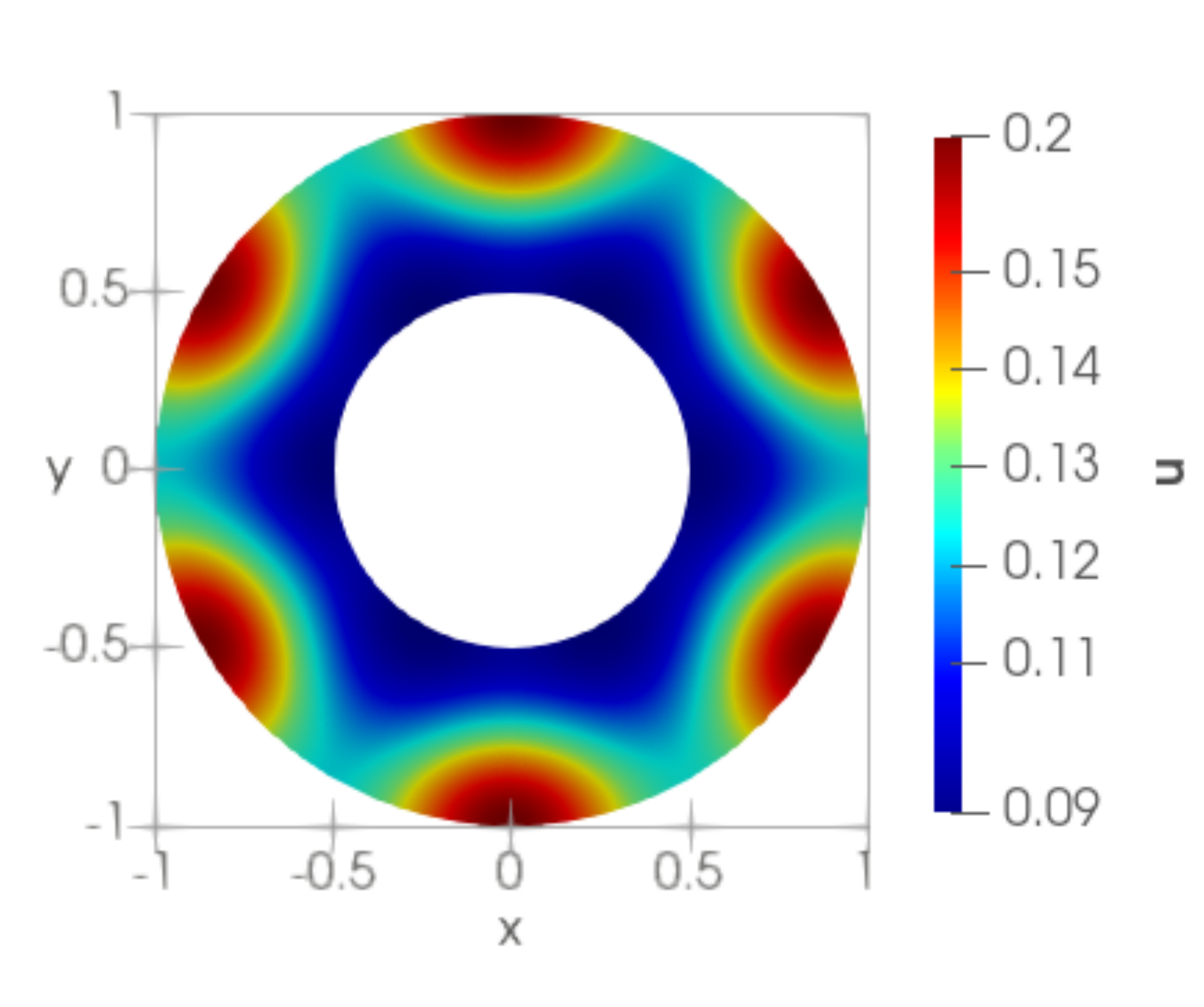} \\
{\small (e) t=1.725}
\end{tabular}
\begin{tabular}{cc}
\includegraphics[width=0.3\linewidth]{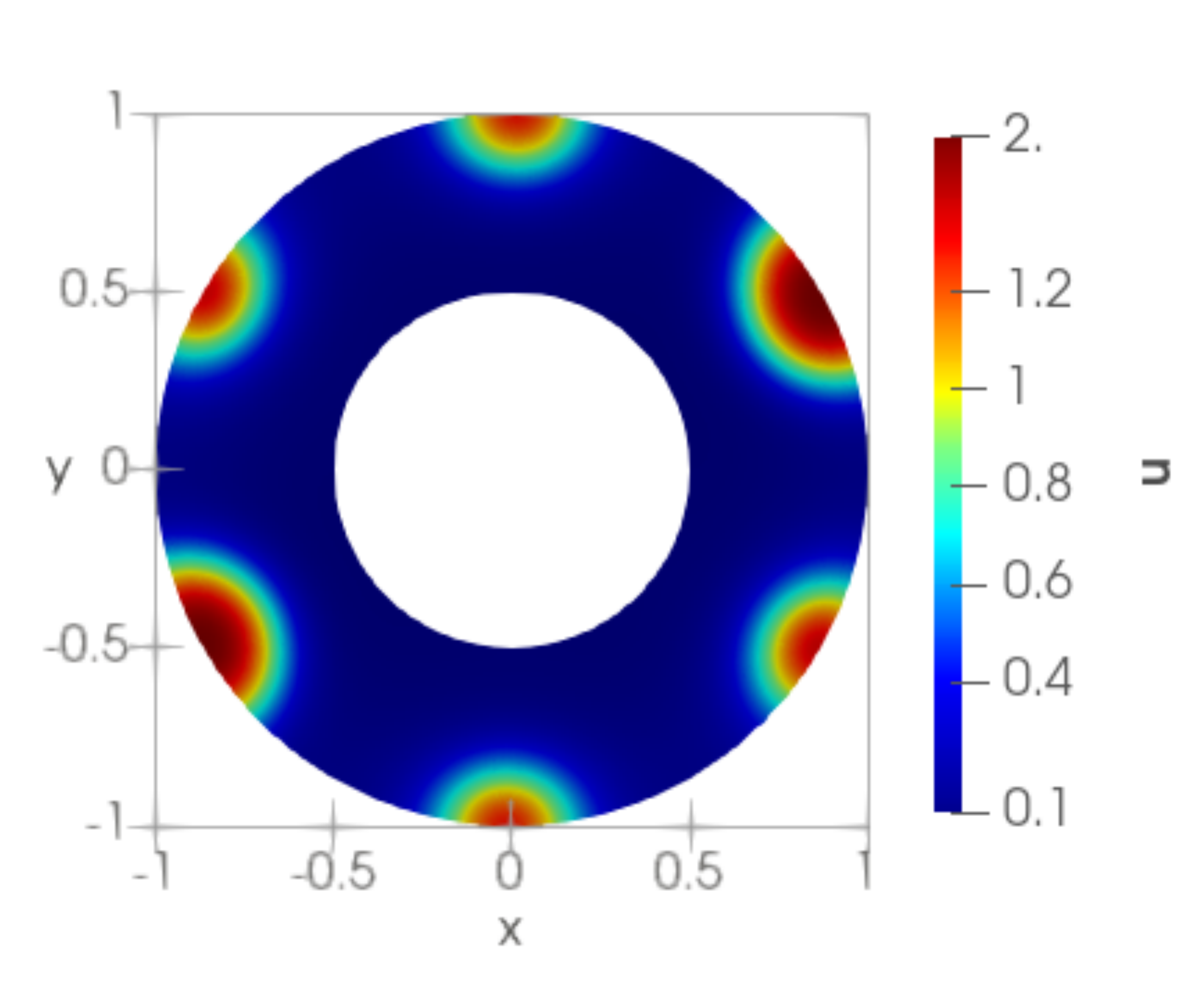} \\
{\small (f) t=1.8}
\end{tabular}
\begin{tabular}{cc}
\includegraphics[width=0.3\linewidth]{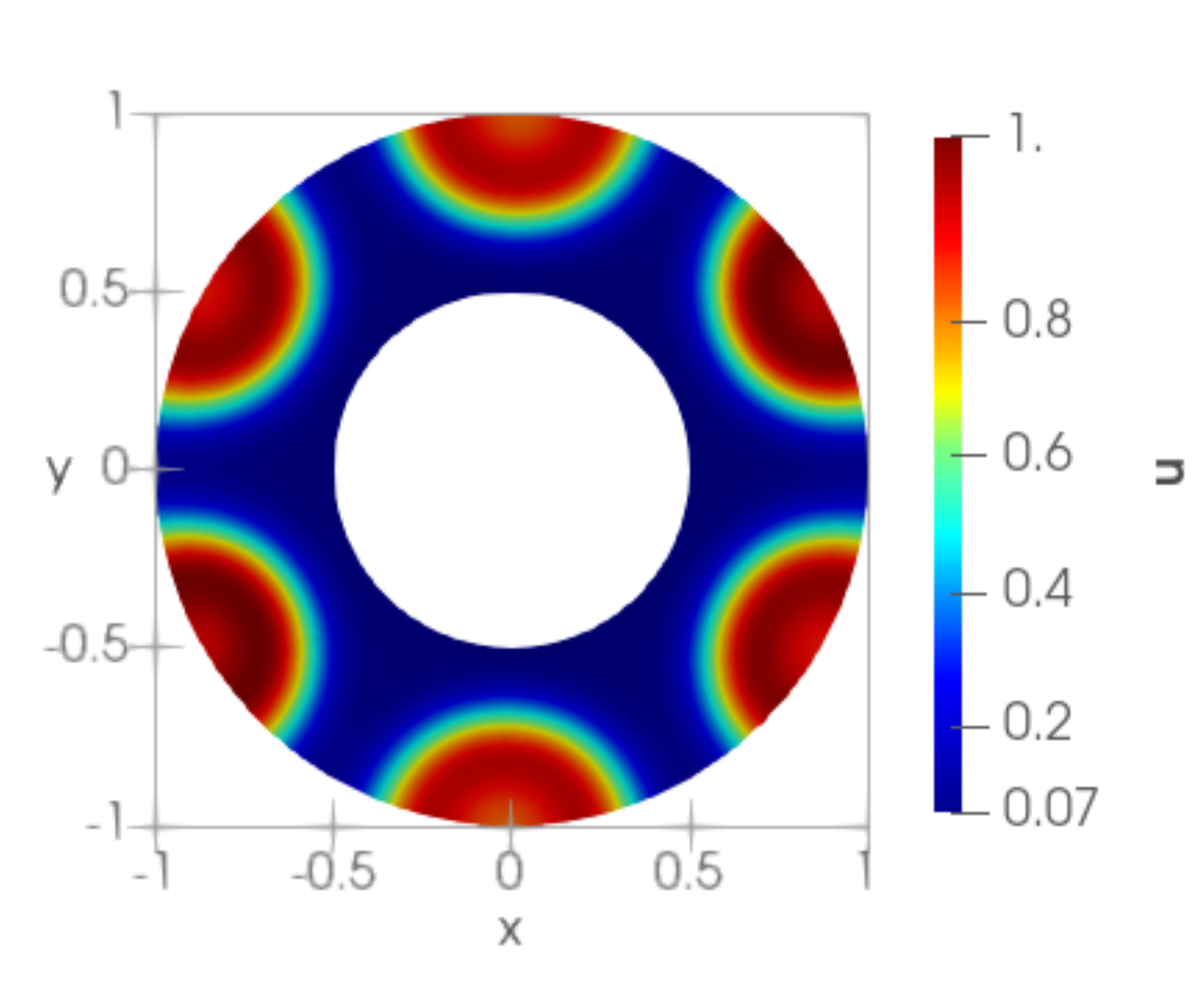} \\
{\small (g) t=1.965}
\end{tabular}
\begin{tabular}{cc}
\includegraphics[width=0.3\linewidth]{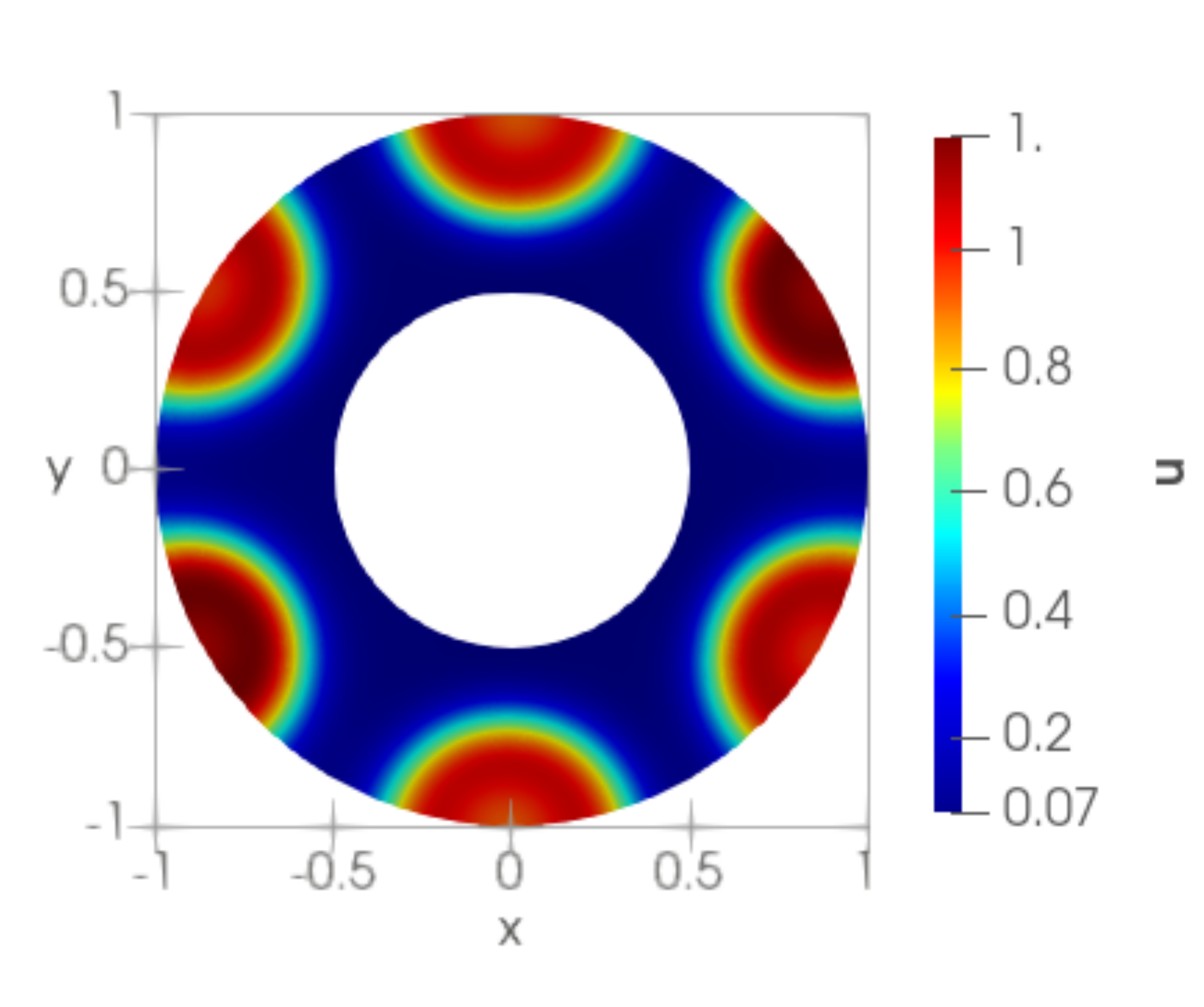} \\
{\small (h) t=2.12}
\end{tabular}
\begin{tabular}{cc}
\includegraphics[width=0.3\linewidth]{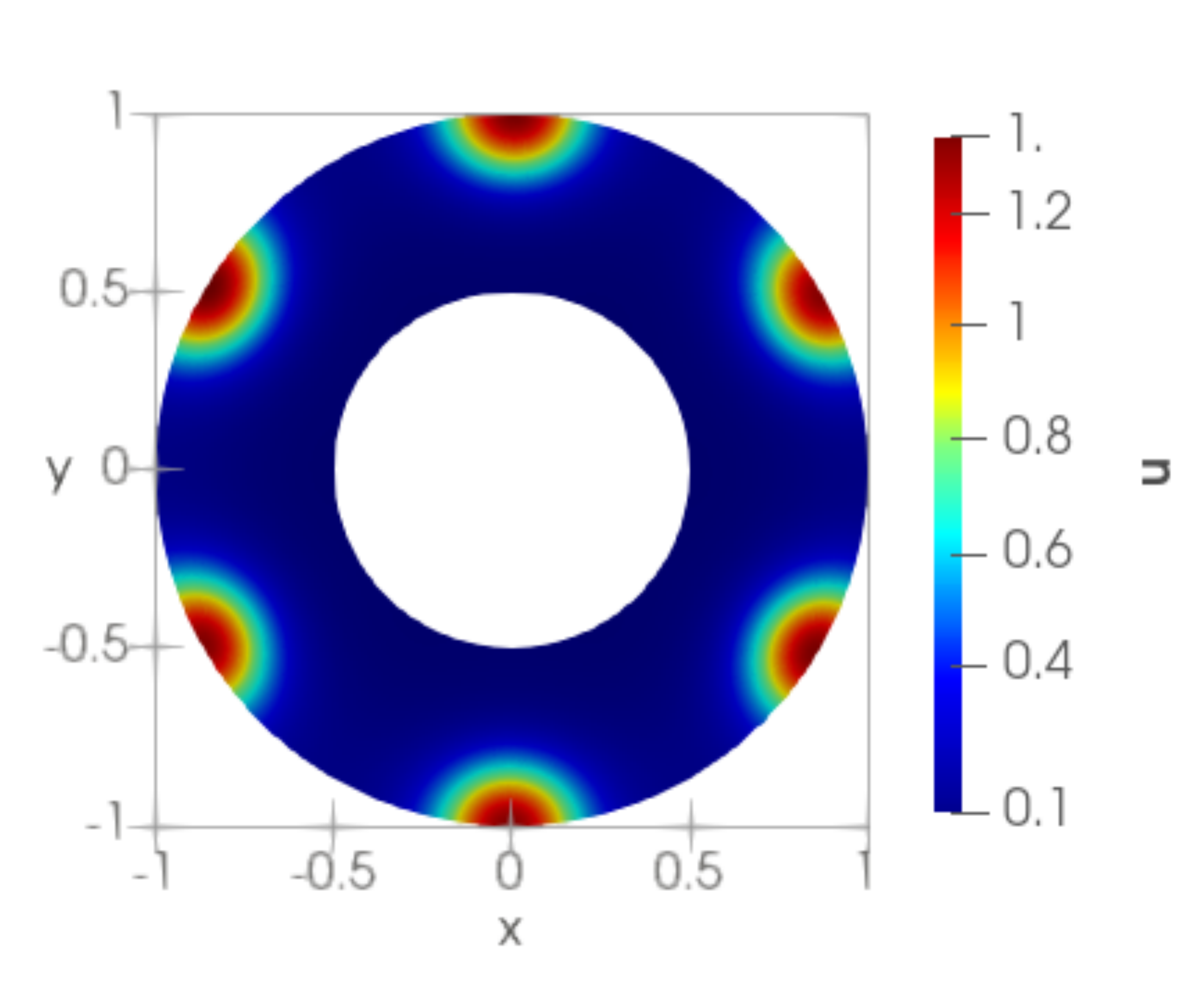} \\
{\small (i) t=2.265}
\end{tabular}
\begin{tabular}{cc}
\includegraphics[width=0.3\linewidth]{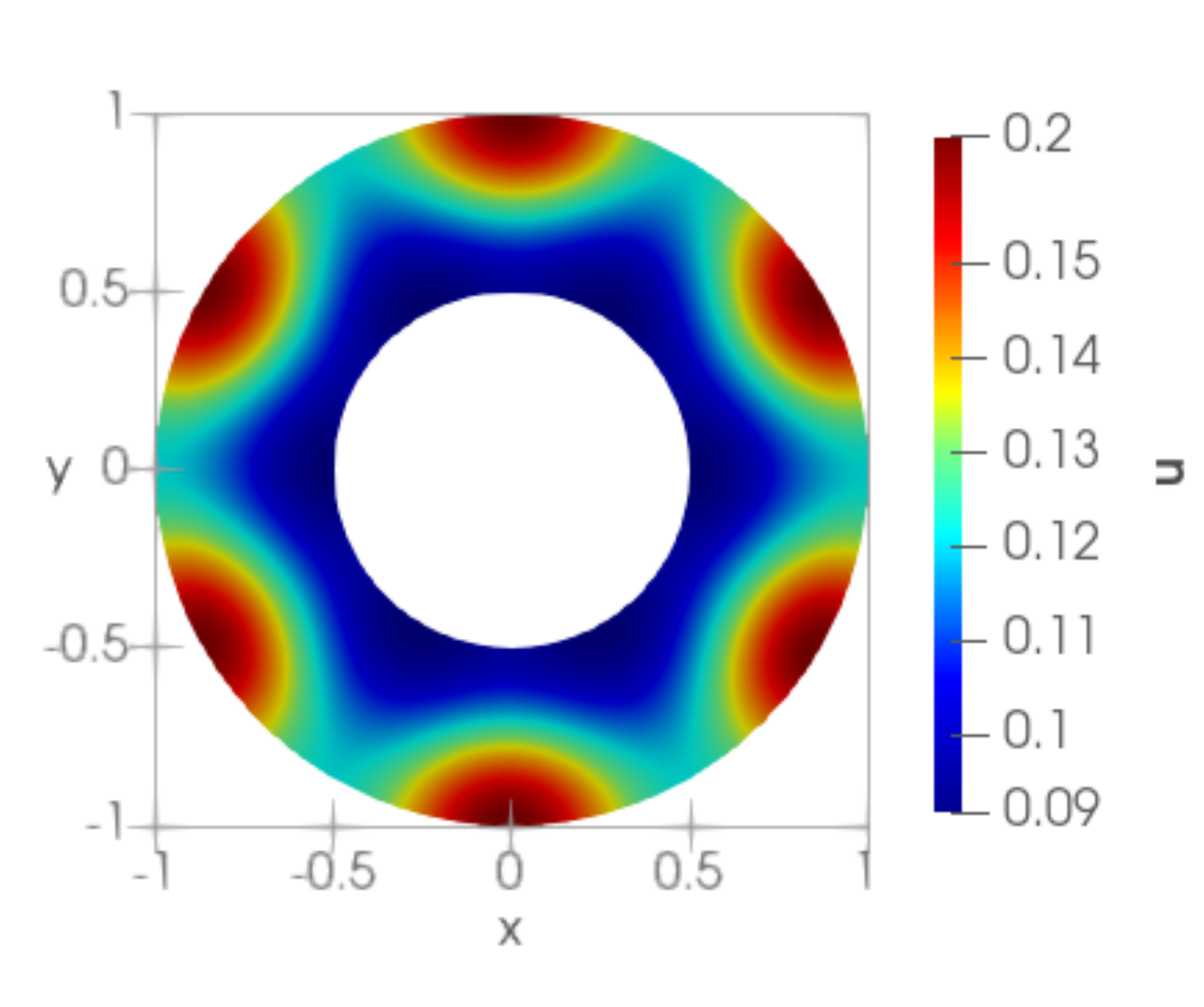} \\
{\small (j) t=2.5}
\end{tabular}
\begin{tabular}{cc}
\includegraphics[width=0.64\textwidth]{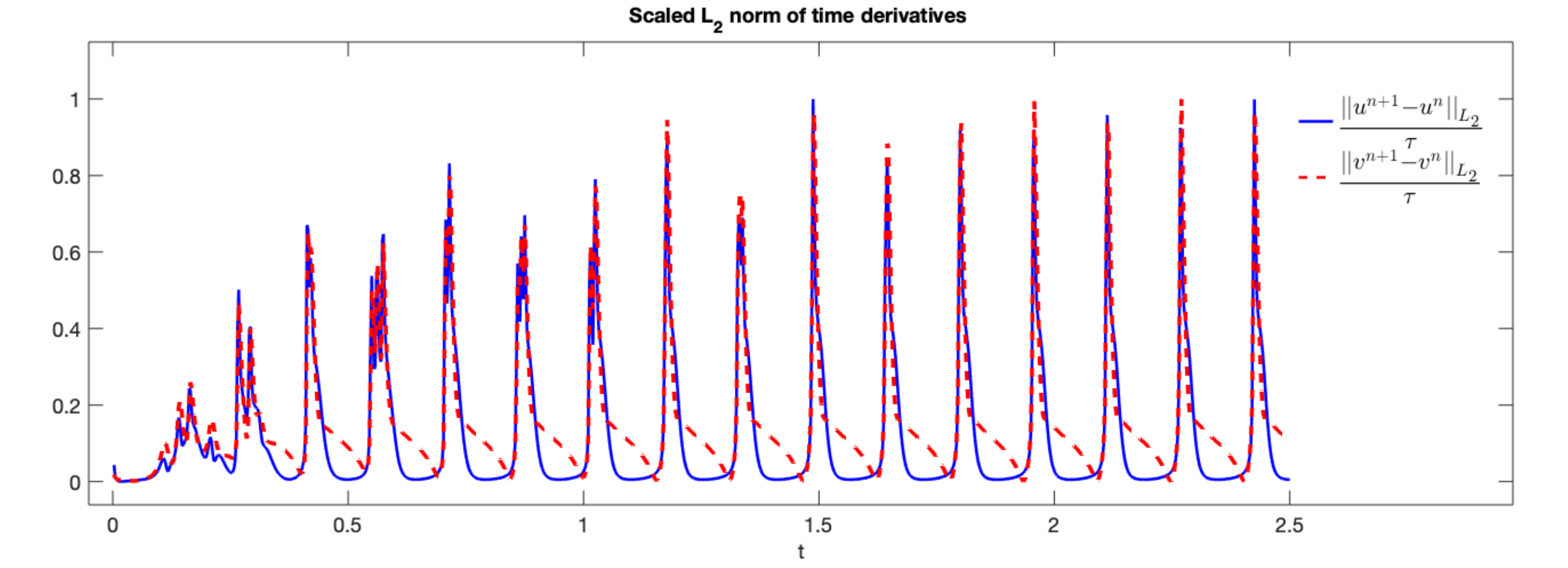} \\
{\small (k) Discrete time derivative of the solutions $u$ and $v$}
\end{tabular}
\caption{(a)-(j) Finite element simulations corresponding to the $u$-component of the cross-diffusive reaction-diffusion system on disc shape domain exhibiting spatiotemporal time-periodic pattern formation. Parameters are selected to satisfy conditions of Theorems \ref{theo1} and \ref{Maincond} on $\rho$ with $\Delta t =0.0025$, $d=1$, $\gamma=250$, $d_u=-0.9$, $d_v=0.55$, $\alpha=0.085$ and $\beta=0.1$, as shown in Fig. \ref{HopfTransParameter}. (k) Plot of the $L_2$ norms showing the periodicity of the  discrete time-derivative of the numerical solutions $u$ and $v$.}
\label{Limitcycle2}
\end{figure}

\begin{figure}[H]
\begin{tabular}{cc}
\includegraphics[width=0.28\linewidth]{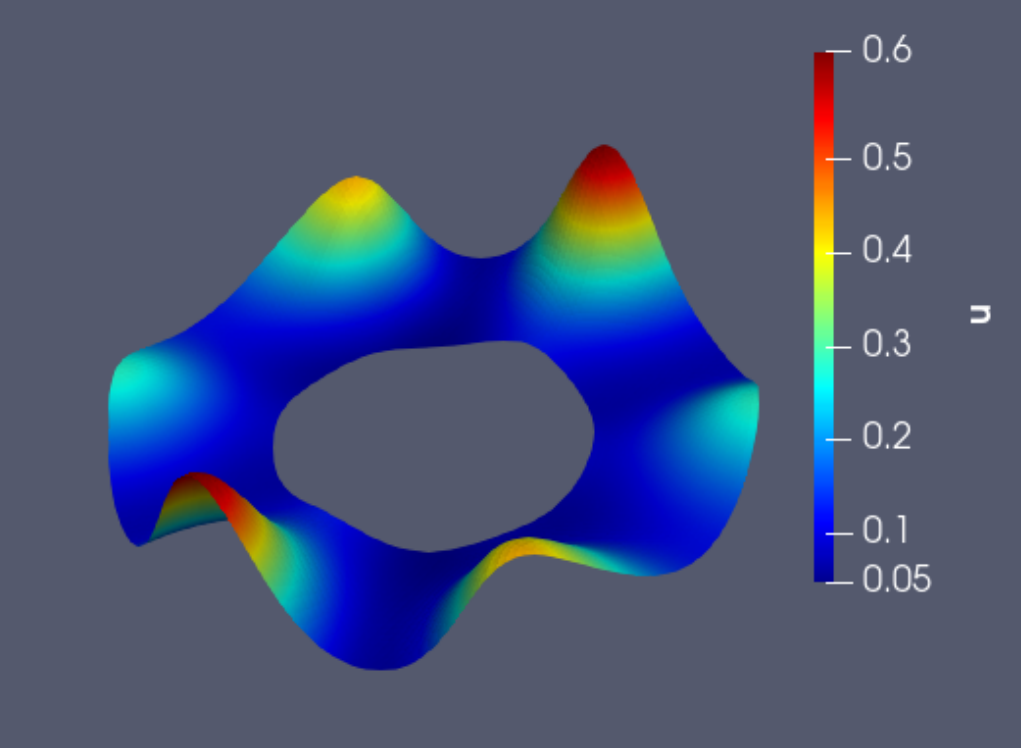} \\
{\small (a) t=0.2}
\end{tabular}
\begin{tabular}{cc}
\includegraphics[width=0.28\linewidth]{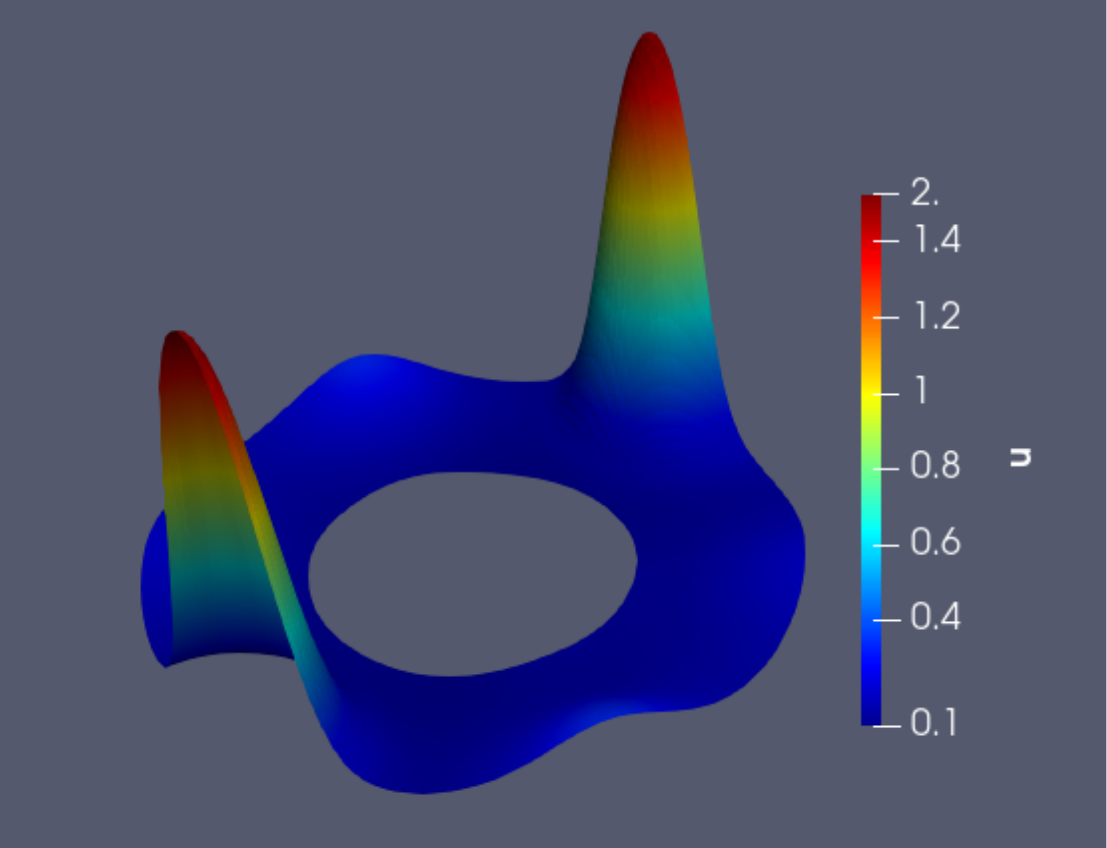} \\
{\small (b) t=0.55}
\end{tabular}
\begin{tabular}{cc}
\includegraphics[width=0.3\linewidth]{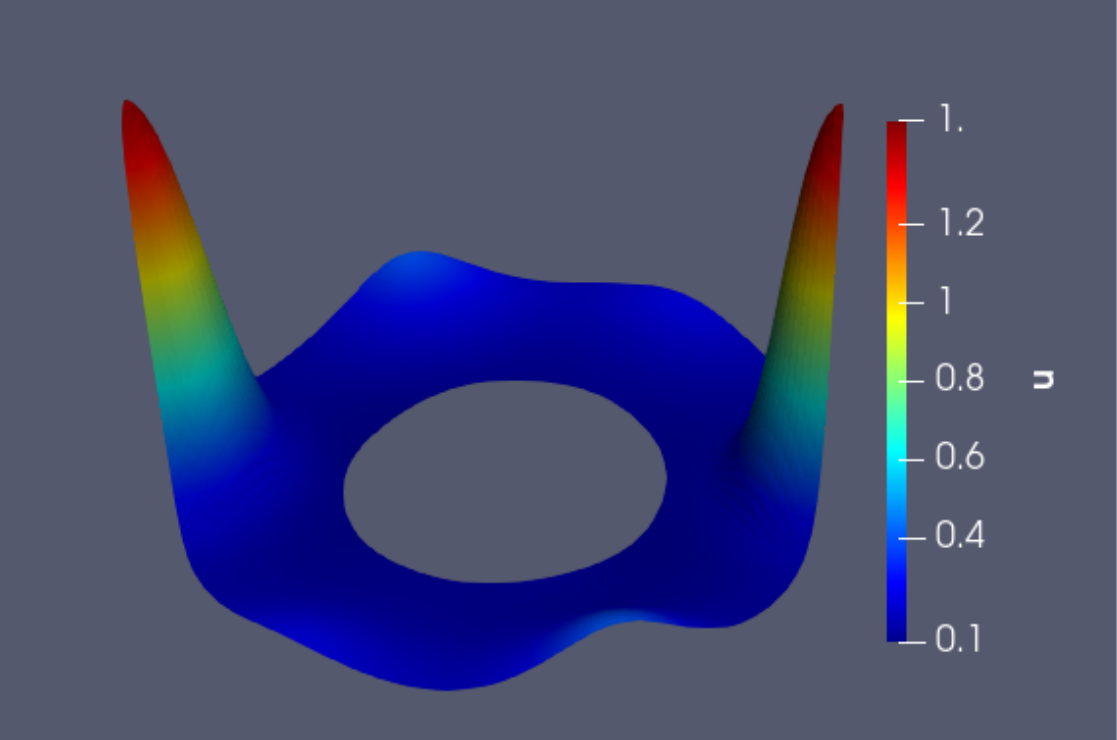} \\
{\small (c) t=0.86}
\end{tabular}
\begin{tabular}{cc}
\includegraphics[width=0.3\linewidth]{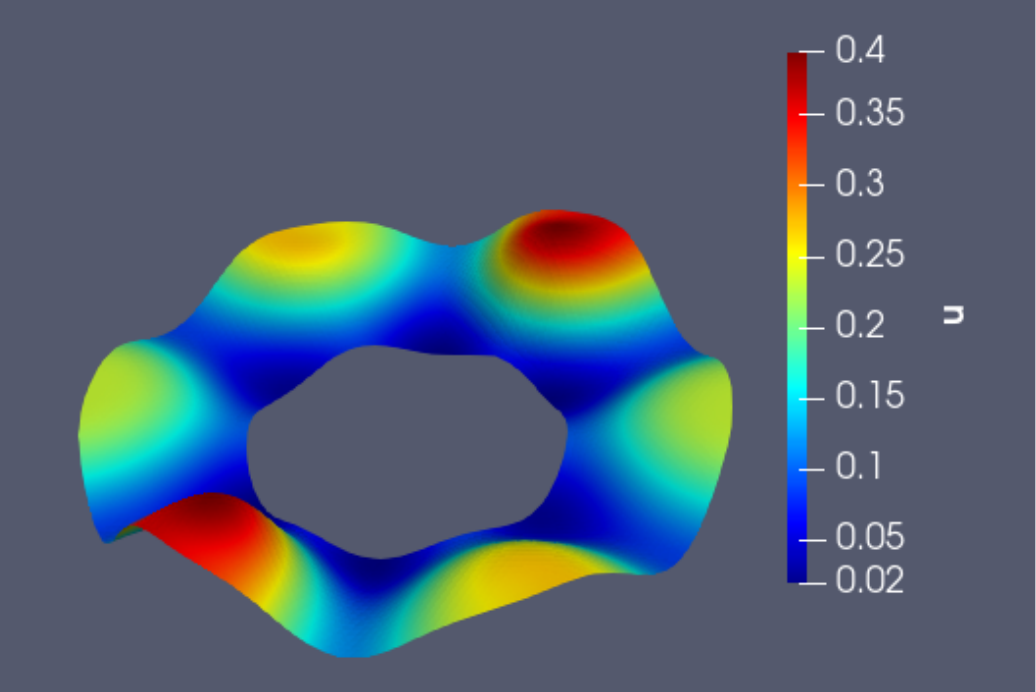} \\
{\small (d) t=0.9}
\end{tabular}
\begin{tabular}{cc}
\includegraphics[width=0.3\linewidth]{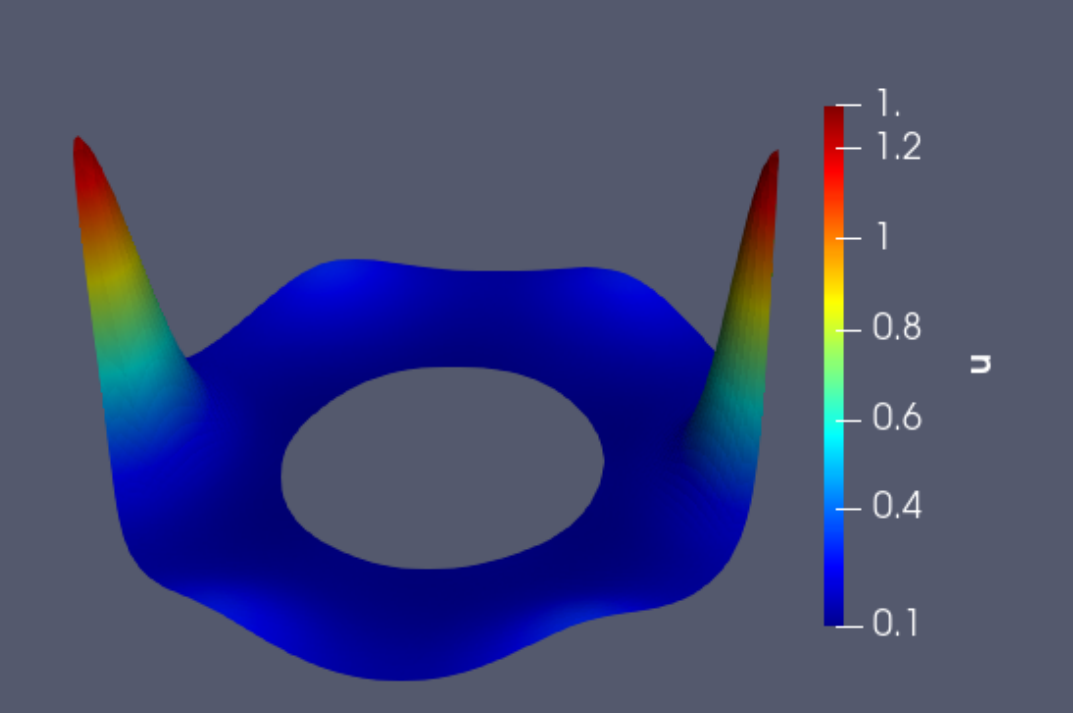} \\
{\small (e) t=1.025}
\end{tabular}
\begin{tabular}{cc}
\includegraphics[width=0.3\linewidth]{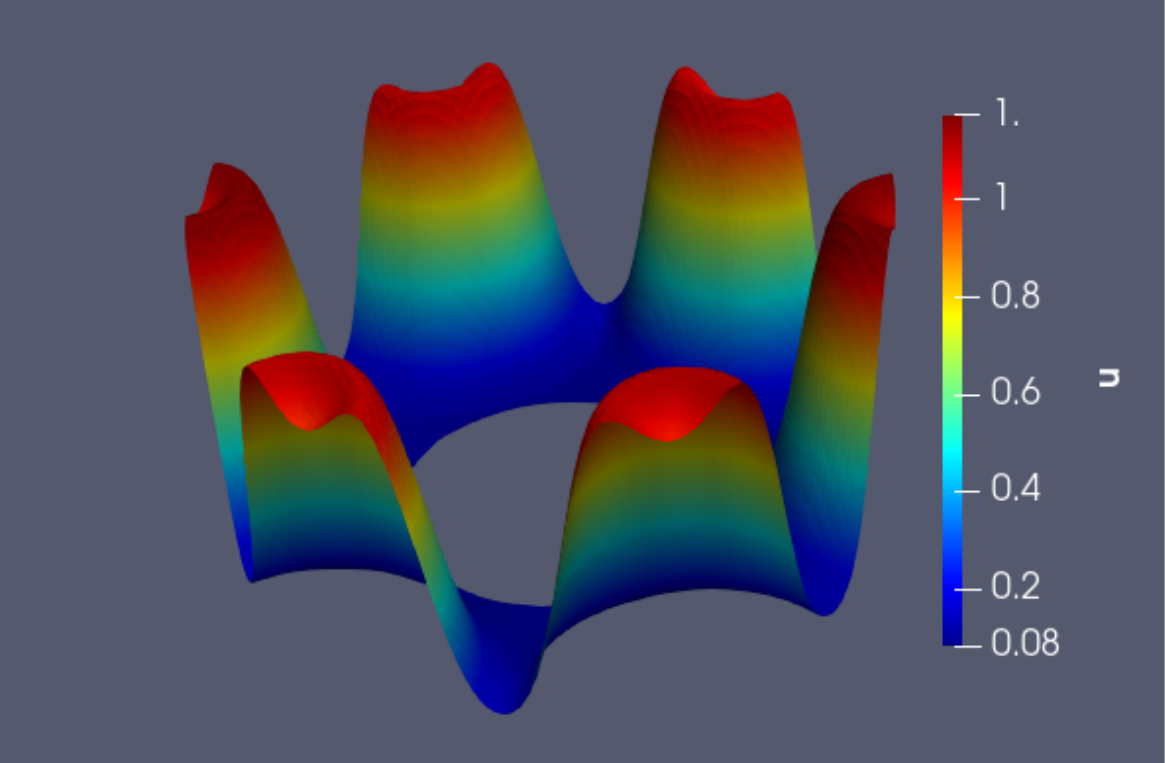} \\
{\small (f) t=1.495}
\end{tabular}
\begin{tabular}{cc}
\includegraphics[width=0.3\linewidth]{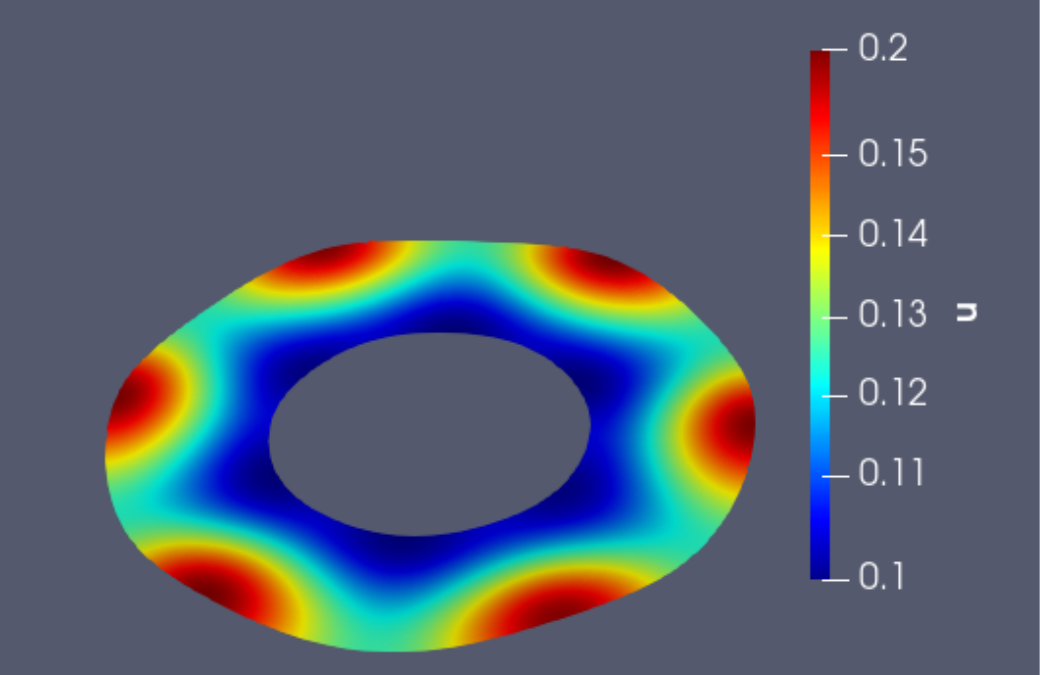} \\
{\small (g) t=1.725}
\end{tabular}
\begin{tabular}{cc}
\includegraphics[width=0.3\linewidth]{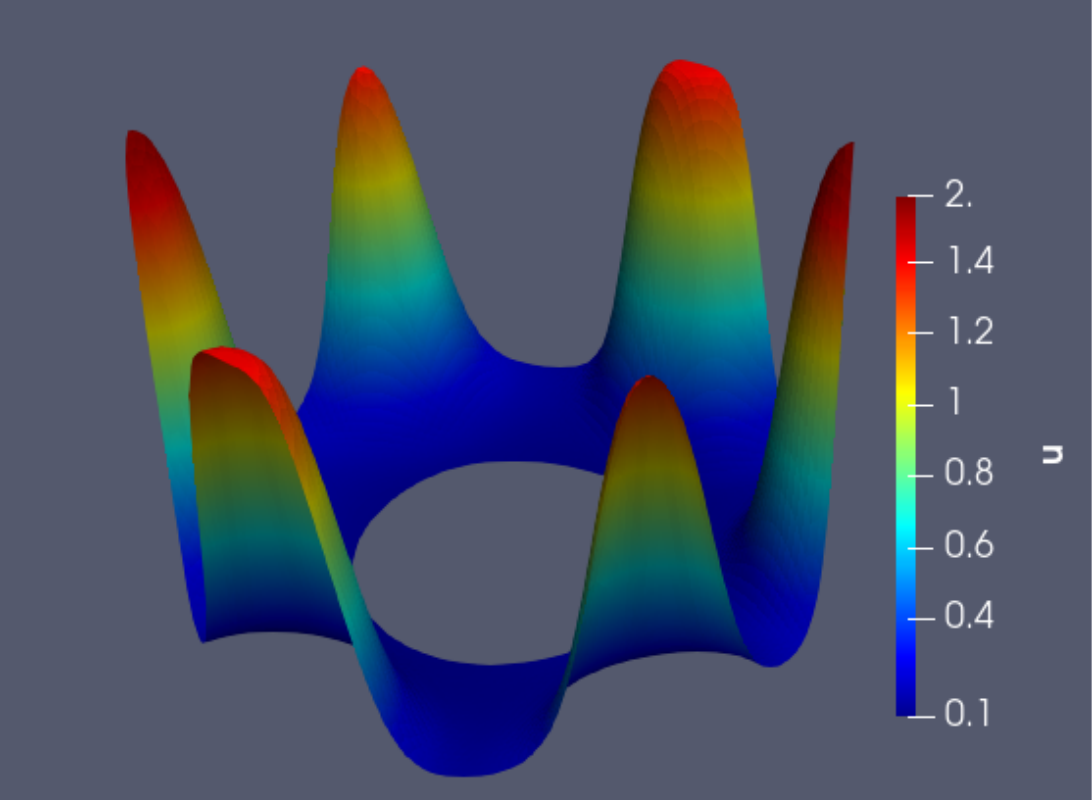} \\
{\small (h) t=1.775}
\end{tabular}
\begin{tabular}{cc}
\includegraphics[width=0.3\linewidth]{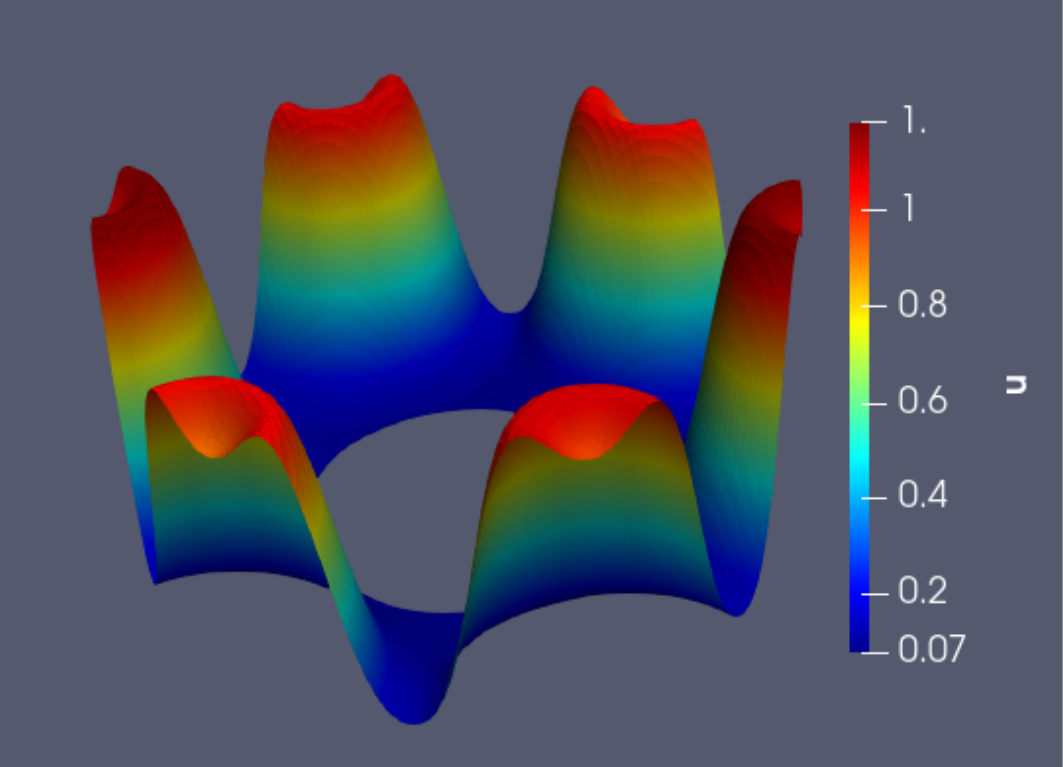} \\
{\small (i) t=2.12}
\end{tabular}
\begin{tabular}{cc}
\includegraphics[width=0.3\linewidth]{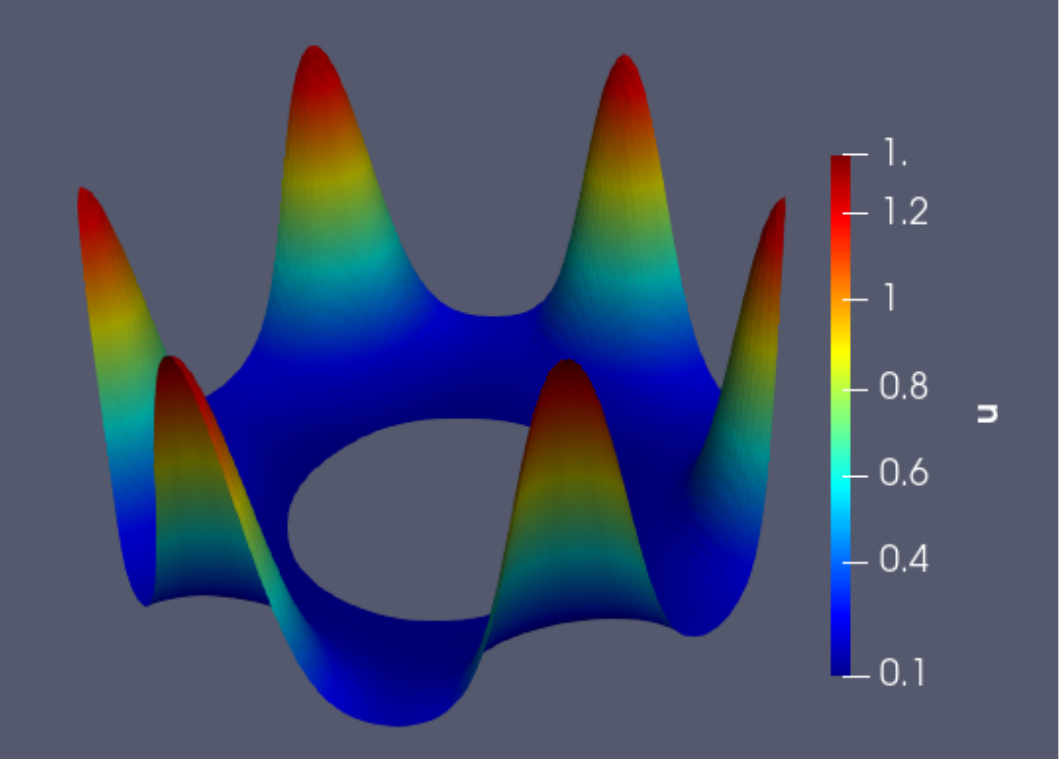} \\
{\small (j) t=2.2675}
\end{tabular}
\begin{tabular}{cc}
\includegraphics[width=0.3\linewidth]{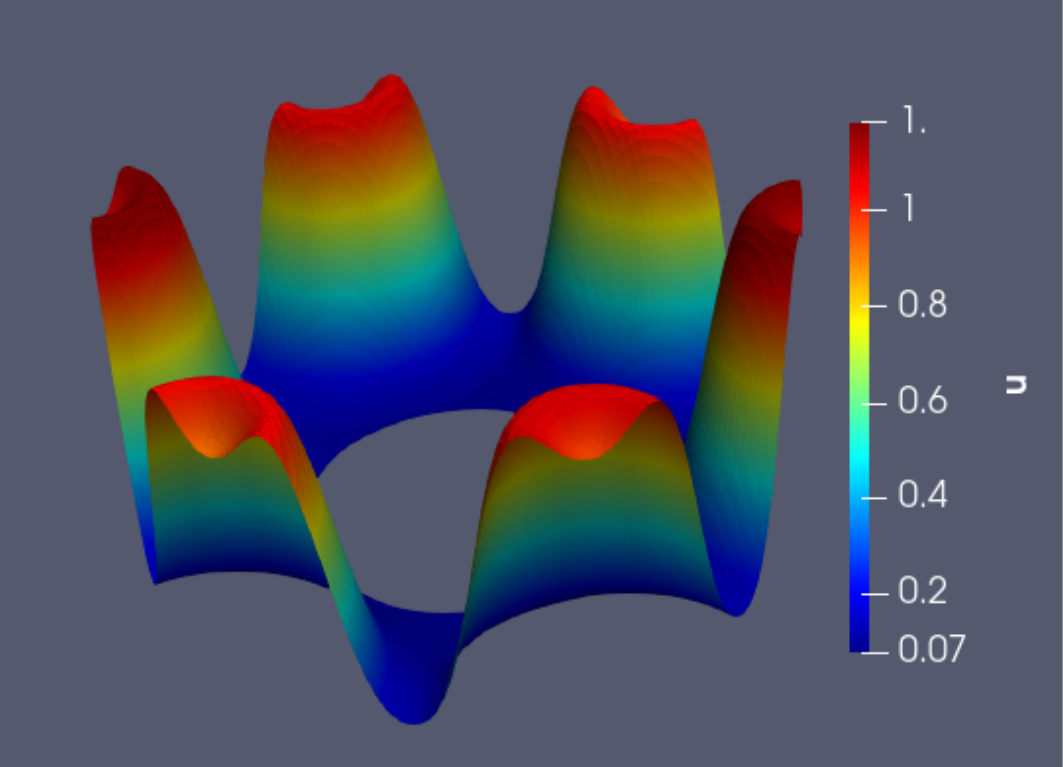} \\
{\small (k) t=2.4325}
\end{tabular}
\begin{tabular}{cc}
\includegraphics[width=0.3\linewidth]{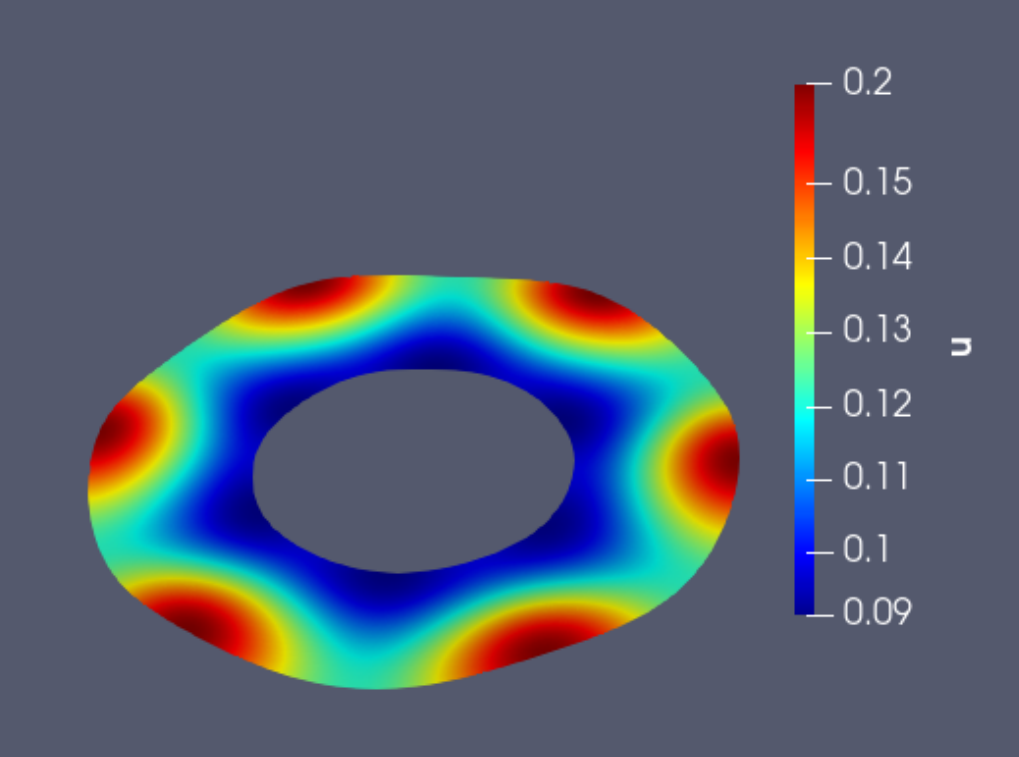} \\
{\small (l) t=2.5}
\end{tabular}
\caption{(a)-(j) 3D view of the finite element simulations corresponding to the $u$-component of the cross-diffusive reaction-diffusion system on disc shape domain exhibiting spatiotemporal time-periodic pattern formation. Parameters are selected to satisfy conditions of Theorems \ref{theo1} and \ref{Maincond} on $\rho$ with $\Delta t =0.0025$, $d=1$, $\gamma=250$, $d_u=-0.9$, $d_v=0.55$, $\alpha=0.085$ and $\beta=0.1$, as shown in Fig. \ref{HopfTransParameter}. We observe the classical phenomenon of "standing waves".}
\label{Limitcycle3}
\end{figure}

\begin{figure}[H]
\centering
\begin{tabular}{cc}
\includegraphics[width=1.05\textwidth]{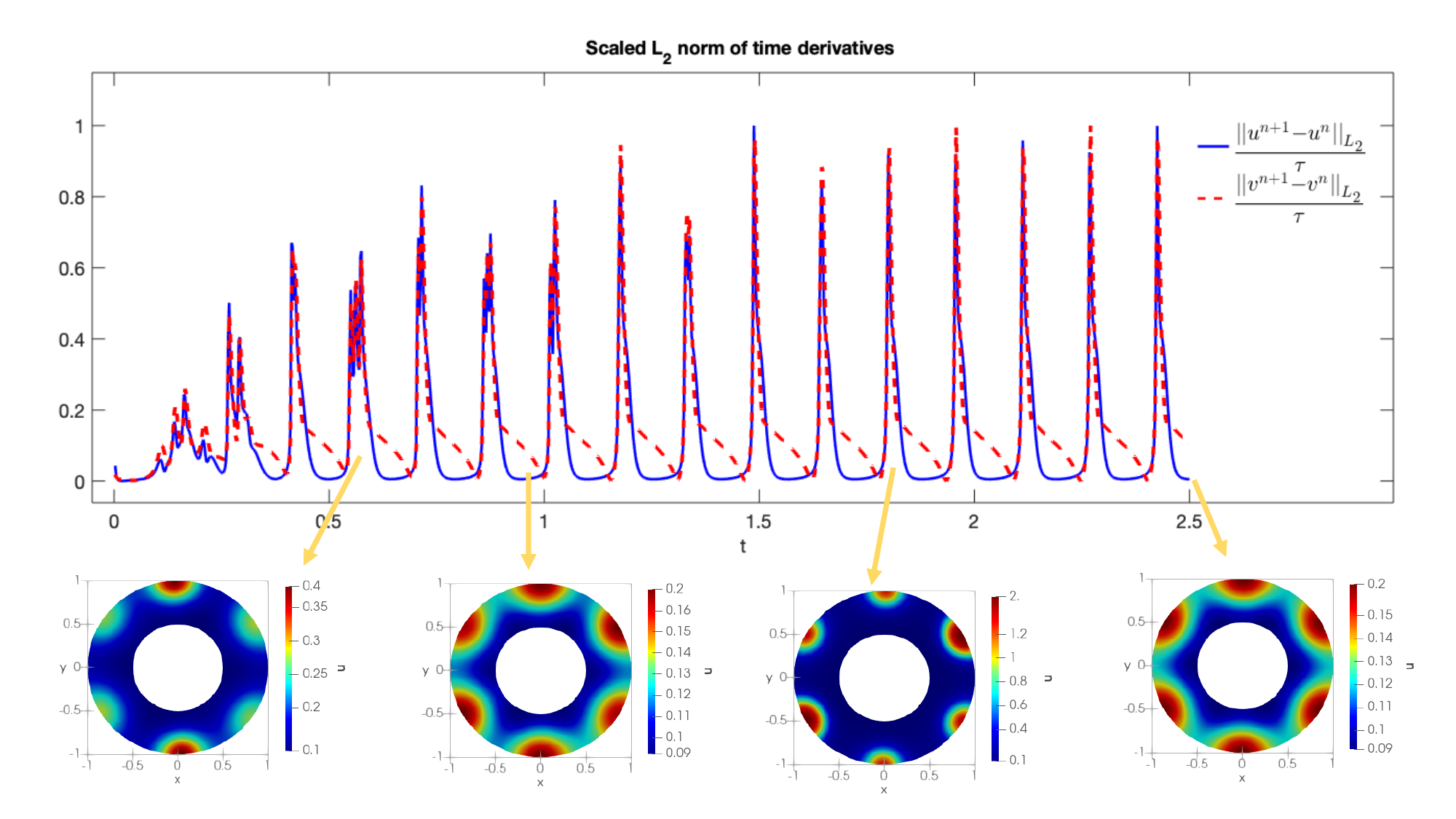} \\
\end{tabular}
\caption{The plot of the evolution of the $L_2$ norm of the discrete time-derivative of $u$ and $v$ finite element solutions with negative cross-diffusion. Model parameter values are selected as $\Delta t =0.0025$, $d=1$, $\gamma=250$, $d_u=-0.9$, $d_v=0.55$, $\alpha=0.085$ and $\beta=0.1$.} 
\label{NegativeCD}
\end{figure}

\section{Conclusion}\label{sec:Conclusion}
Investigating cross-diffusive reaction-diffusion systems is a rapidly emerging area of research for pattern formation and has numerous applications in many chemical and biological problems in developmental and cell biology, material, and plant sciences.  In this study,  we provide the importance of a proper parameter selection strategy through a rigorous understanding of the spatiotemporal behavior of solutions which is critically dependent on the size of the domain. An analytical linear stability method is applied to obtain the relationship between the domain size $\rho$ and the system parameters $d$, $d_u$,  $d_v$, and $\gamma$ on a flat ring shape domain. This work entails an important novelty by comparing the parameter spaces using the eigenvalues of the stability matrix on convex and non-convex geometries. To the authors' knowledge, this work is the first of its kind where such comparisons are made and detailed conclusions are derived. The detailed parameter spaces have been classified, taking into account the Hopf/transcritical bifurcations and Turing instabilities on the annulus.  It must be noted that most of the spaces generated when the self-diffusion coefficient $d=1$ only exist in the presence of cross-diffusion for a two-component reaction-diffusion system. We note that such a restriction does not hold for multi-component reaction-diffusion systems \cite{cotterell2015local,krause2021introduction,madzvamuse2010stability,klika2017history,van2021turing,krause2023concentration,tsubota2024bifurcation}. For two-component reaction-diffusion systems, such spaces do not exist in the absence of cross-diffusion. To confirm the predicted behaviour in the dynamics,  model parameter values from the parameter regions are hand-picked and the reaction-diffusion system with linear cross-diffusion is then solved with these parameter values by using the finite element method on an annular domain. Plots of the discrete $L_2$ norms of the discrete time-derivative of the solutions are generated to illustrate the temporal behaviour of the system as well as the formation of spatially inhomogeneous Turing patterns. For example, when parameter values are selected from the Turing parameter space, the $L_2$ norm exhibits rapid decay at the early stages, which is followed by rapid exponential growth associated with the growing modes of the linear reaction-diffusion system, when the real part of the eigenvalues is positive, and finally, the growth plateaus and starts to decay monotonically due to the effects of the nonlinear reaction-kinetics which act as bounds on the exponentially growing modes. Another example, is the periodic behaviour of the $L_2$ norm which is associated with the  periodicity of limit cycles and Hopf bifurcation types of dynamics, here model parameters are selected from the spaces satisfying Hopf and transcritical instabilities. Our numerical simulations provide the expected dynamical behaviour of the reaction-diffusion system considering the effect of cross-diffusion and the domain size on the ring shape domain.

Hence, the key research methodology and outcomes of our studies can be summarised as follows. A complete analytical exploration of the spatiotemporal dynamics in an activator-depleted reaction-diffusion system, a linear stability analysis providing the dual roles of cross-diffusion and domain size of an annulus,  the derivation of precise stability conditions through lower and upper bounds of the domain size,   a full classification of the model parameters,  and  a demonstration of how cross-diffusion relaxes the general conditions for the reaction-diffusion system to exhibit pattern formation.

Our current studies involve a multi-throned research study to (i) extend the analysis to stationary surfaces where analytical tractability can be established, (ii) understand the role of domain size during growth development on planar and evolving surfaces in the presence of linear cross-diffusion and (iii) to explore weakly nonlinear analysis when nonlinear cross-diffusion is considered. 

\section*{Acknowledgement}
GY would like to thank the Scientific and Technological Research Council of T\"{u}rkiye (T\"{U}BITAK) for their support.
GY would like to thank the Isaac Newton Institute for Mathematical Sciences, Cambridge, for support and hospitality during the programme Mathematics of movement: an interdisciplinary approach to mutual challenges in animal ecology and cell biology,  where partial work on this paper was undertaken and supported by EPSRC grant EP/R014604/1.  
This work (AM) was supported by the Canada Research Chair (Tier 1) in Theoretical and Computational Biology (CRC-2022-00147), the Natural Sciences and Engineering Research Council of Canada (NSERC), Discovery Grants Program (RGPIN-2023-05231), the British Columbia Knowledge Development Fund (BCKDF), Canada Foundation for Innovation – John R. Evans Leaders Fund – Partnerships (CFI-JELF), the British Columbia Foundation for Non-Animal Research, and the UKRI Engineering and Physical Sciences Research Council (EPSRC: EP/J016780/1).  RB was partially supported by National Funding from FCT - Fundac\~{a}o para a Ci\^{e}ncia e Tecnologia, Portugal, under the project UIDB/04561/2020: https://doi.org/10.54499/UIDB/04561/2020.

\bibliographystyle{ws-ijbc}
\bibliography{sample}

\end{document}